%% file: primed-shifted-insertion-draft4.tex
\documentclass[11pt]{article}

\input{preamble.tex}
\usepackage{fullpage}

\usepackage{listings}
\lstdefinelanguage{Sage}[]{Python}
{morekeywords={False,sage,True},sensitive=true}
\definecolor{dblackcolor}{rgb}{0.0,0.0,0.0}
\definecolor{dbluecolor}{rgb}{0.01,0.02,0.7}
\definecolor{dgreencolor}{rgb}{0.2,0.4,0.0}
\definecolor{dgraycolor}{rgb}{0.30,0.3,0.30}

\def\biword{compatible sequence}
\def\biwords{compatible sequences}

\def\swdiag{{\mathsf{sw}}}
\def\nediag{\mathsf{ne}}

\begin{document}
\title{Shifted insertion algorithms for primed words}
\author{Eric Marberg  \\ Department of Mathematics \\  Hong Kong University of Science and Technology \\ {\tt emarberg@ust.hk}}
\date{}

\maketitle

\begin{abstract}
This article studies some new insertion algorithms that associate pairs of shifted tableaux  to finite integer sequences in which certain terms may be primed. When primes are ignored in the input word these algorithms reduce to known correspondences, namely, a shifted form of Edelman-Greene insertion, Sagan-Worley insertion, and Haiman's shifted mixed insertion. These maps have the property that when the input word varies such that one output tableau is fixed, the other output tableau ranges over all (semi)standard tableaux of a given shape with no primed diagonal entries. Our algorithms have the same feature, but now with primes allowed on the main diagonal. One application of this is to give another Littlewood-Richardson rule for products of Schur $Q$-functions. It is hoped that there will exist set-valued generalizations of our bijections that can be used to understand products of $K$-theoretic Schur $Q$-functions. 
\end{abstract}

\setcounter{tocdepth}{2}
\tableofcontents

\section{Introduction}

This article studies some new insertion algorithms that generate pairs of shifted tableaux from finite integer sequences 
in which certain terms may be primed.
The first half of this introduction contains a quick summary of our main results. The second half discusses 
some open problems that motivate our constructions.

\subsection{Outline}

Let $S_\ZZ$ be the group of permutations of the integers with finite support,
and set $s_i := (i,i+1) \in S_\ZZ$ for $i \in \ZZ$.
 There is a unique associative product $\circ : S_\ZZ \times S_\ZZ\to S_\ZZ$
such that $\sigma \circ s_i = \sigma$ if $\sigma(i) > \sigma(i+1)$ and $\sigma\circ s_i = \sigma s_i$ if $\sigma(i)<\sigma(i+1)$ for each $i \in \ZZ$ \cite[Thm. 7.1]{Humphreys}. This so-called \defn{Demazure product} may be defined in terms of the \defn{Bruhat order} $\leq $ on $S_\ZZ$
by the set-wise product identity $\{ \sigma \in S_\ZZ : \sigma \leq v\}\{\sigma \in S_\ZZ : \sigma\leq w\} = \{ \sigma \in S_\ZZ : \sigma \leq v\circ w\}$ for $v,w \in S_\ZZ$.

A \defn{reduced word} for $\sigma \in S_\ZZ$
is an integer sequence $a_1a_2\cdots a_n$ of shortest possible length 
with $\sigma=s_{a_1}s_{a_2} \cdots s_{a_n}$, or equivalently with
\[\sigma =  s_{a_1} \circ s_{a_2}\circ \cdots\circ s_{a_n}.\]
Write $\cR(\sigma)$ for the set of reduced words for $\sigma \in S_\ZZ$.
Analogously, an \defn{involution word} for  $z \in S_\ZZ$ is a word $a_1a_2\cdots a_n$
of shortest possible length such that 
\[z=s_{a_n}\circ \cdots   \circ s_{a_2}\circ s_{a_1}  \circ s_{a_2} \circ \cdots\circ s_{a_n}.\]
Write $\iR(z)$ for the set of involution words for $z \in S_\ZZ$.
One can show that this set is nonempty if and only if $z=z^{-1}$ is an involution.
The empty word $\emptyset$ is both the unique reduced word and the unique involution word for $1 \in S_\ZZ$.

Involution words have been studied previously 
in different forms and under various names, for example, in \cite{CJW,HMP1,HanssonHultman,HuZhang1,RichSpring}.
We are concerned here with the following slight generalization.
An index $i \in [n]$ is \defn{commutation} in an involution word $a=a_1a_2\cdots a_n$ 
if $s_{a_i}$ commutes with $s_{a_{i-1}}\circ \cdots   \circ s_{a_2}\circ s_{a_1}   \circ s_{a_2} \circ \cdots\circ s_{a_{i-1}}$.
The index $i=1$ is a commutation whenever the word $a$ is nonempty.
A \defn{primed involution word} for $z=z^{-1}\in S_\ZZ$ is any word
formed by 
adding primes to the entries indexed by a subset of commutations in some $a \in \iR(z)$. Such a word is a sequence of letters in the
\defn{primed alphabet}
$\{ \dots < 1' < 1 < 2' < 2 < \dots\}$.
Write $\iR^+(z)$ for the set of primed involution words for $z$.
As we will explain in Section~\ref{pr-words-sect}, all involution words
for a given $z=z^{-1} \in S_\ZZ$ have the same number $k$ of commutations, so we have $|\iR^+(z)| = 2^{k} |\iR(z)|$.
For example, if $z=321\in S_3 \subset S_\ZZ$,
then 
\[\cR(z) = \{121, 212\},\quad\iR(z) = \{12,21\},\quand\iR^+(z) = \{12, 1'2,21,2'1\}.\]

For any word $a$, let $\Incr_\infty(a)$ denote the set of sequences $(a^1,a^2,a^3,\dots)$
where each $a^i$ is a weakly increasing possibly empty word such that $a=a^1a^2a^3\cdots$.
For a set of words $\cA$, let 
\[\textstyle\Incr_\infty(\cA) = \bigsqcup_{a \in \cA} \Incr_\infty(a).\]
Fix an involution  $z =z^{-1} \in S_\ZZ$. In Section~\ref{main-sect} we describe a specific map \[
a \mapsto (\PO(a),\QO(a))\] that takes an element of 
$\iR^+(z)$ or $\Incr_\infty(\iR^+(z))$ as its input and gives a pair of shifted tableaux as its output.
Our first main result is the following theorem about this operation.

\begin{theorem}[See Theorems~\ref{o-thm1} and \ref{o-summary-thm}]\label{intro-thm1}
The map $a \mapsto (\PO(a),\QO(a))$ is a bijection from $\iR^+(z)$ 
(respectively, $\Incr_\infty(\iR^+(z))$)
 to the set of pairs $(P,Q)$
where $P$ is a shifted tableau with increasing rows and columns,
 no primed entries on the main diagonal, and row reading word  in $\iR^+(z)$,
and $Q$ is a standard (respectively, semistandard) shifted tableau of the same shape.  
\end{theorem}

In this context, a \defn{shifted tableau} of a strict partition shape $\lambda = (\lambda_1>\lambda_2>\dots>\lambda_k>0)$
is a filling of the \defn{shifted diagram} 
$\SD_\lambda := \{ (i,i+j-1) \in \ZZ\times \ZZ : 1\leq i \leq k\text{ and }1 \leq j\leq \lambda_i\}$
by elements of $\{ \dots < 1' < 1 < 2' < 2 < \dots\}$. If we draw such a tableau in French notation, 
then its \defn{row reading word} is formed by 
reading each of its rows in the usual way from left to right, starting with the top row.\footnote{Equivalently, if the tableau has entry $T_{ij}$ in box $(i,j)$, then the row reading word is formed by arranging the numbers $T_{ij}$ in the order that makes $(-i,j)$ increase lexicographically, as $(i,j)$ varies over all boxes.}
A shifted tableau is \defn{semistandard} if its entries are positive and its rows and columns are weakly increasing as indices increase,
with no primed number repeated in a row and no unprimed number repeated in a column.
A semistandard shifted tableau with $n$ boxes is \defn{standard} if it contains exactly one of $i$ or $i'$ 
as an entry for each $i=1,2,\dots,n$.

\begin{example}
We present a very simple case of the map $a \mapsto (\PO(a),\QO(a))$ just to illustrate its domain and codomain.
If $z=321\in S_3$ then the elements of $\Incr_\infty(\iR^+(z))$ have one of $6$ forms:
\[
\ba
a &= (\emptyset,\emptyset,\emptyset,\dots,\emptyset, 1\phantom{'}, \emptyset,\emptyset,\emptyset,\dots,\emptyset, 2, \emptyset,\emptyset,\emptyset,\dots),
\\
b &= (\emptyset,\emptyset,\emptyset,\dots,\emptyset, 1', \emptyset,\emptyset,\emptyset,\dots,\emptyset, 2, \emptyset,\emptyset,\emptyset,\dots),
\\
c &= (\emptyset,\emptyset,\emptyset,\dots,\emptyset, 2\phantom{'}, \emptyset,\emptyset,\emptyset,\dots,\emptyset, 1, \emptyset,\emptyset,\emptyset,\dots),
\\
d &= (\underbrace{\emptyset,\emptyset,\emptyset,\dots,\emptyset}_{p-1\text{ terms}}, 2', \underbrace{\emptyset,\emptyset,\emptyset,\dots,\emptyset}_{q-p-1\text{ terms}}, 1, \emptyset,\emptyset,\emptyset,\dots),
\ea
\]
for some integers $0< p< q$,
or
\[
\ba
e &= (\emptyset,\emptyset,\emptyset,\dots,\emptyset,1\phantom{'}2, \emptyset,\emptyset,\emptyset,\dots),
\\
f &= (\underbrace{\emptyset,\emptyset,\emptyset,\dots,\emptyset}_{p-1\text{ terms}}, 1'2,  \emptyset,\emptyset,\emptyset,\dots),
\ea
\]
for some $p>0$. We have $ \PO(a)=\PO(b)=\PO(c)=\PO(d)=\PO(e) = \PO(f) = \ytab{1 & 2}$ as
this is the unique shifted tableau with increasing rows and columns
and no primed entries on the main diagonal
whose row reading word is in $\{12,1'2,21,2'1\}$.
On the other hand,  it will follow from the definitions in Section~\ref{main-sect} that
\[\ba
\QO(a) &= \ytab{p & q},
\quad
\QO(b) = \ytab{p' & q},
\quad
\QO(c) = \ytab{p & q'},
\\
\QO(d)& = \ytab{p' & q'},
\quad
\QO(e) = \ytab{p & p},
\quad
\QO(f) = \ytab{p' & p}.
\ea
\]
As $0<p<q$ vary, these outputs range over all semistandard shifted tableaux of shape $\lambda=(2)$.
\end{example}

Restricting
$a \mapsto (\PO(a),\QO(a))$
 to unprimed words 
gives the map called \defn{involution Coxeter-Knuth insertion} in \cite{HMP4,Marberg2019a}
and \defn{orthogonal Edelman-Greene insertion} in \cite{Marberg2019b}.
The latter, in turn, is
a special case of the \defn{shifted Hecke insertion algorithm} from 
\cite{HKPWZZ,PatPyl2014}. 
Our correspondence  
 is the ``orthogonal'' counterpart
to a ``symplectic'' shifted insertion algorithm
studied in \cite{Hiroshima,Marberg2019a,Marberg2019b}; see Remark~\ref{o-sp-rmk}.

It is an open problem to find a ``primed'' generalization of shifted Hecke insertion
that extends our bijection $a \mapsto (\PO(a),\QO(a))$. 
The image of such a map should consist of pairs of 
shifted tableaux $(P,Q)$ of the same shape, in which $P$ has increasing rows and columns
with no primed entries on the main diagonal,
and $Q$ is an arbitrary \defn{(semistandard) set-valued shifted tableau} in the sense
of \cite[\S9.1]{IkedaNaruse}.
It is less clear what superset of $ \iR^+(z)$ should be the domain 
of such a correspondence.
As discussed in the next section,
generalizing shifted Hecke insertion in this way would have 
 interesting applications.

Besides constructing the map 
$a \mapsto (\PO(a),\QO(a))$, we also seek to 
understand how 
$a$ can vary when $\PO(a)$ is held constant,
and how such changes affect $\QO(a)$.
Our second set of results, sketched below and 
explained more thoroughly in Sections~\ref{ock-sect} and \ref{dual-sect},
fully solves this problem.

\begin{theorem}[See Theorem~\ref{ck-fkd-thm} and Corollary~\ref{o-cor2}]
\label{intro-thm2}
There are explicit operators $\ck_i$ on primed words
which act by changing at most three consecutive letters,
along with operators $\fkd_i$ on standard shifted tableaux 
which act by changing at most three consecutive entries,
such that if $a$ is a primed involution word
then
$\QO(\ck_i(a)) = \fkd_i(\QO(a))$,
and 
if $a$ and $b$ are both primed involution words then
$\PO(a) = \PO(b)$ if and only 
$a = \ck_{i_1} \ck_{i_2} \cdots \ck_{i_k}(b)$
for some sequence $  i_1,i_2,\dots,i_k$.
\end{theorem}

Section~\ref{main-sect} contains these and our other main results, following
some preliminaries in Section~\ref{prelim-sect},
The proof of Theorem~\ref{intro-thm2} 
is unexpectedly difficult and takes up all of Section~\ref{proofs-sect}.
 We use Theorems~\ref{intro-thm1} and \ref{intro-thm2} to derive some 
 additional results in Section~\ref{other-sect}. Specifically, 
 in Section~\ref{modif-sect}, 
we describe a variation of 
\defn{Sagan-Worley insertion} from \cite{Sag87,Worley}
whose domain is the set of all \defn{primed \biwords}.
Then in Section~\ref{mixed-sect} we investigate two related 
extensions of Haiman's \defn{shifted mixed insertion algorithm} from \cite{HaimanMixed}.

\subsection{Motivation}\label{motiv-sect}

We use the second half of this introduction to explain some of our motivations 
for considering  
the insertion algorithm in Theorems~\ref{intro-thm1} and \ref{intro-thm2}.
These motivations are related to the problem of finding a combinatorial rule to multiply certain ``$K$-theoretic'' symmetric functions.

The \defn{Schur $P$-function}  of a strict partition $\lambda$
is the generating function $P_\lambda=\sum_{T} x^T$ for all 
semistandard shifted tableaux $T$ of shape $\lambda$ with no primed entries on the main diagonal;
here one sets $x^T := \prod_i x_i^{m_i}$ where $m_i$ is the number of entries of $T$ equal to $i$ or $i'$.
The \defn{Schur $Q$-function} $Q_\lambda$ is defined in the same way but without excluding primes from the main diagonal,
or more directly as the scalar multiple $Q_\lambda = 2^{\ell(\lambda)} P_\lambda$.
It is well-known that both power series are symmetric functions that are Schur positive,
and that the set of all $P_\lambda$'s (respectively, all $Q_\lambda$'s) is a $\ZZ$-basis for a ring with nonnegative integer structure constants \cite{S28}.

Ikeda and Naruse introduced $K$-theoretic analogues $\GP_\lambda$ and $\GQ_\lambda$
for the Schur $P$-functions and $Q$-functions in \cite{IkedaNaruse}.
These power series are also symmetric, and
may be defined similarly as the generating functions for all \defn{semistandard set-valued shifted tableaux} of a given shape, 
where for $\GP_\lambda$ primed entries are again prohibited from appearing in diagonal positions \cite[\S9.1]{IkedaNaruse}. The precise definition 
involves a bookkeeping parameter $\beta$, which makes both power series homogeneous if one sets $\deg(\beta) = -1$ and $\deg(x_i) =1$.
For simplicity, we take $\beta=1$ in our discussion here.
With this convention, one recovers $P_\lambda$ and $Q_\lambda$ by taking the homogeneous
terms of lowest degree in $\GP_\lambda$ and $\GQ_\lambda$, respectively.

It was conjectured in \cite{IkedaNaruse} that the set of all $\GP_\lambda$'s (respectively, all $\GQ_\lambda$'s) is 
a basis for a ring. For the $\GP_\lambda$'s 
this follows from the main result in \cite{CTY}; other proofs also appear in \cite[\S4]{HKPWZZ} and \cite[\S8]{PechenikYong}. For the $\GQ_\lambda$'s, surprisingly, Ikeda and Naruse's conjecture is technically 
still unresolved, though it is known from \cite{IkedaNaruse} that each product $\GQ_\lambda\GQ_\mu$ is a possibly infinite linear combination of $\GQ_\nu$'s.
However, in general,
it remains to show that this expansion has finitely many terms and to give an interpretation
 of its coefficients.\footnote{There is at least a Pieri rule to expand $\GQ_\lambda\GQ_{\mu}$ into $\GQ_{\nu}$'s  when $\mu=(p)$ has a single part  \cite[Cor. 5.6]{BuchRavikumar}. There is also a formula for the expansion of 
 $\GQ_\lambda\GQ_\mu$ for any strict $\lambda$, $\mu$ into monomials \cite[Cor. 7.8]{MGP}.}
 These difficulties have to do with the fact that $\GQ_\lambda$ is no longer a scalar multiple of $\GP_\lambda$.

There is a bijective approach to
proving that the $K$-theoretic Schur $P$-functions generate a ring,
which we sketch below.
 The results in this article are a first step toward extending this strategy
 to handle the $K$-theoretic Schur $Q$-functions.
 
For each even integer $n>0$, let $\Ifpf_n$ denote the set of fixed-point-free involutions in the symmetric group $S_n := \langle s_1,s_2,\dots,s_{n-1}\rangle$.
Each element $z \in \Ifpf_n$ has an associated set of \defn{symplectic Hecke words} $\cHfpf(z)$ defined in \cite[\S1.3]{Marberg2019a}.
This set is infinite unless $z$ is $ \wfpf := (1,2)(3,4)\cdots (n-1,n)$.
Each word in $\cHfpf(z)$ is a finite integer sequence that does not begin with an odd letter.
The shortest words in $\cHfpf(z)$ are the minimal length sequences $a_1a_2\cdots a_k$ with $z = s_{a_k} \cdots s_{a_2} s_{a_1} \wfpf s_{i_1} s_{a_2}\cdots s_{a_k}$. 
Given $z \in \Ifpf_n$ and a strict partition $\lambda$, define
\[ \KP_z := \sum_{\phi \in \Incr_\infty(\cHfpf(z))} x^\phi\]
where  $x^\phi := \prod_i x_i^{\ell(a_i)}$
for $\phi=(a^1,a^2,a^3,\dots)$.

A  \defn{semistandard weak set-valued shifted tableaux} of strict partition shape $\lambda$
is a filling of $\SD_\lambda$ by elements of $\{ 1' < 1 < 2' < 2 < \dots\}$, with multiple elements and repetitions allowed in each box,
but with no primed numbers repeated in a row and no unprimed numbers repeated in a column.
The entries of such a tableau $T$ are required to be weakly increasing in the sense that the largest entry in one box
cannot be greater that the smallest entry in the next box in the same row or column.
The weight of $T$ is again the monomial $x^T := \prod_i x_i^{m_i}$ where $m_i$ is the number of entries of $T$ equal to $i$ or $i'$.
Let $\KP_\lambda := \sum_T x^T
$
where the sum is over all semistandard weak set-valued shifted tableaux of shape $\lambda$ with no primed entries on the main diagonal.
By \cite[Cor. 6.6]{LM}, we have
$\GP_\lambda = \omega(\KP_\lambda)$,
where $\omega$ is the automorphism of the algebra of symmetric functions  sending each Schur function $s_\lambda\mapsto s_{\lambda^\top}$.
In turn, each $\KP_z$ is related to $\KP_\lambda$ by the following theorem:

\begin{theorem}[{See \cite[Thm. 4.5]{Marberg2019a}}]
Let $z \in \Ifpf_n$.
There is a bijection 
$\phi \mapsto (P_\Sp(\phi), Q_\Sp(\phi))$
from $\Incr_\infty(\cHfpf(z))$ to the set of pairs $(P,Q)$ where $P$ is a shifted tableau with increasing rows and columns whose row reading word is in $\cHfpf(z)$, and $Q$ is a weak set-valued shifted tableau of the same shape with no primed entries on the main diagonal.
Moreover, one has 
$x^\phi = x^{Q_\Sp(\phi)}$.
\end{theorem}

This bijection is called  \defn{symplectic Hecke insertion} 
in \cite{Marberg2019a}.
If $a =a_1a_2\cdots a_k \in \cHfpf(z)$ then
we set $P_\Sp(a) =P_\Sp(\phi)$ and 
$Q_\Sp(a) =Q_\Sp(\phi)$
for $\phi = (a_1, a_2,\dots,a_k,\emptyset,\emptyset,\dots)$.
The value of $P_\Sp(\phi)$ depends only on the underlying word,
but not on its division into weakly increasing factors.
All letters in a symplectic Hecke word for $z \in \Ifpf_n$ are in $\{1,2,\dots,n-1\}$,
so there are only finitely many shifted tableaux with increasing rows and columns
that can have row reading words in $\cHfpf(z)$.
It follows that  $\KP_z$ is the finite sum 
$\sum_{T \in \{ P_\Sp(a) : a \in \cHfpf(z)\}} \KP_{\mathsf{shape}(T)}.$

Assume $y \in \Ifpf_m$ and $z \in \Ifpf_n$ for even integers $m,n \geq 0$.
Let $y\times z  \in \Ifpf_{m+n}$ be the permutation
mapping $i\mapsto y(i)$ for $1\leq i \leq m$ and $i + m \mapsto z(i) + m$ for $1\leq i \leq n$.
Next, for $\phi = (a^1,a^2,\dots) \in \Incr_\infty(\cHfpf(y))$
and 
$\psi = (b^1,b^2,\dots) \in \Incr_\infty(\cHfpf(z))$,
let $\phi \oplus \psi = ( a^1c^1, a^2c^2,\dots)$
where $c^i$ is formed by adding $m$ to each letter of $b^i$.

It is clear from the results about  symplectic Hecke words in \cite[\S1.3]{Marberg2019a}
that $(\phi,\psi) \mapsto \phi \oplus \psi $ is a bijection
$\Incr_\infty(\cHfpf(y))\times \Incr_\infty(\cHfpf(z)) \xrightarrow{\sim}
\Incr_\infty(\cHfpf(y\times z))$.
  Therefore
$\KP_y\KP_z = \KP_{y\times z}.$
In turn, if the largest part of $\lambda$
is less than $n-1$, then there 
exists  $z^\fpf_\lambda \in \Ifpf_n$ (with an explicit formula) such that $\KP_\lambda = \KP_{z^\fpf_\lambda}$ \cite[Thm. 4.17]{MP2021}.
Since $\omega$ is an algebra automorphism, we have
\be
 \KP_\lambda \KP_\mu = \sum_\nu e_{\lambda\mu}^{\nu} \KP_\nu
 \quand
  \GP_\lambda \GP_\mu = \sum_\nu e_{\lambda\mu}^{\nu} \GP_\nu
  \ee
  where $e_{\lambda\mu}^\nu$  is the number of tableaux in  
$\{ P_\Sp(a) : a \in\cHfpf(z^\fpf_\lambda \times z^\fpf_\mu)\}$ of shape $\nu$.\footnote{
This becomes a Littlewood-Richardson rule for the symmetric functions
 $\GP^{(\beta)}_\lambda$ defined in
\cite{IkedaNaruse},
which involve a formal parameter $\beta$,
via the identity
$\GP^{(\beta)}_\lambda =\beta^{-|\lambda|} \GP_\lambda( \beta x_1, \beta x_2 , \beta x_3, \dots) $.}

Here is how one could try to adapt this argument to 
derive an analogous formula for the coefficients expanding $\GQ_\lambda \GQ_\mu$ into $\GQ$-functions.
The appropriate analogue of $\KP_\lambda$ is 
the generating function $\KQ_\lambda := \sum_T x^T$
for all weak set-valued shifted tableaux $T$ of shape $\lambda$, now with primed entries allowed on the main diagonal. We have 
$\GQ_\lambda = \omega(\KQ_\lambda)$ by \cite[Cor. 6.6]{LM}.

There is a natural candidate for the $Q$-form of $\KP_z$.
When $n$ is even, the symplectic group
$\Sp_n(\CC)$
 acts on the type $\mathsf{A}_{n-1}$ flag variety $\Fl_n$ with finitely many orbits indexed by
 $\Ifpf_n$. The closures of these orbits have canonical 
 representatives in the connective $K$-theory ring of $\Fl_n$
 satisfying a certain stability property \cite{WyserYong}. These representatives are  polynomials  
 $\fk G^\Sp_z \in \ZZ[\beta][x_1,x_2,\dots]$, and their ``stable limits'' give 
 certain symmetric functions $\GP^\Sp_z$ that satisfy 
 $\KP_z = \omega(\GP^\Sp_z|_{\beta=1}) $
 (compare \cite[Cor. 4.6]{MP2019a} with the results in \cite[\S5]{Marberg2019a}).

For any positive integer $n$, the orthogonal group $\O_n(\CC)$ likewise acts on $\Fl_n$
with finitely many orbits, now indexed by $I_n := \{ z \in S_n : z=z^{-1} \}$.
The closures of these orbits again have canonical 
 representatives in the connective $K$-theory ring of $\Fl_n$
 satisfying a certain stability property \cite{MP2019a}.
 These are inhomogeneous polynomials  
 $\fk G^\O_z\in \ZZ[\beta][x_1,x_2,\dots]$ indexed by $z \in I_n$.
Mimicking the properties of $\KP_z$, one would like to define
the ``stable limit''
 \[\GQ^\O_z:=\lim_{m\to \infty} \fk G^\O_{1^m \times z}\] for $z \in I_n$,
 where $1^m$ is the identity permutation in $S_m$,
 and  then set 
\[\KQ_z :=\omega(\GQ^\O_z|_{\beta=1}).\]
These definitions would be appropriate because
if $z$ is \defn{vexillary}, that is, $2143$-avoiding,
then the limit giving
$\GQ^\O_z$ converges, the resulting power series
$\KQ_z$ is equal to $\KQ_\lambda$ for a certain strict partition $\lambda$,
and any $\KQ_\lambda$ can be attained in this way 
\cite[Thm. 4.11]{MP2019a}.
Some difficulties remain, however:
\begin{itemize}
\item[(a)] No proof is yet known that $\lim_{m\to \infty} \fk G^\O_{1^m \times z}$
converges if $z$ is not vexillary \cite[Prob. 5.3]{MP2019a}.

\item[(b)] There should exist a set of \defn{orthogonal Hecke words} $\iH(z)$, analogous to $\cHfpf(z)$,
such that $\KQ_z = \sum_{\phi \in \Incr_\infty(\iH(z))} x^\phi$
and $\KQ_y \KQ_z = \KQ_{y \times z}$
for all $y \in I_m$ and $z \in I_n$.
It is not known how to define this set even when $z$ is vexillary.

\item[(c)] If the first two issues can be addressed, then to prove that the $\GQ_\lambda$'s generate a ring and to
find a combinatorial interpretation of the $\GQ$-expansion of $\GQ_\lambda \GQ_\mu$, it remains only to find
an appropriate \defn{orthogonal Hecke insertion} algorithm.
This 
should bijectively map elements of $\Incr_\infty(\iH(z))$  
to pairs $(P,Q)$ of shifted tableaux with the same shape,
where now $Q$ is weak set-valued but with primed entries allowed on the main diagonal.
\end{itemize}
The results in this paper provide a base case 
for the last item. 

Specifically, 
$\iH(z)$ should be a superset of $\iR^+(z)$ and
the 
definition of orthogonal Hecke insertion should be an extension
of our map $a \mapsto (\PO(a),\QO(a))$.
This is because if we replace the inhomogeneous
polynomial $\fk G^\O_{1^m \times z}$ by its terms of lowest degree,
then the desired stable limit  does always converge as $m\to \infty $ (see \cite[\S1.5]{HMP1}), so at least the lowest degree terms of $\GO_z$ and $\KQ_z$ are well-defined.
Both of these give the same homogeneous symmetric function (by \cite[Cor. 4.62]{HMP4},
since $\omega$ fixes every Schur $Q$-function),
which we denote by $Q_z$.

As explained 
in Section~\ref{qapp-sect}, it 
  further holds that 
$Q_z =  \sum_{\phi \in \Incr_\infty(\iR^+(z))} x^\phi$
and $Q_yQ_z = Q_{y\times z}$
for all $y \in I_m$ and $z \in I_n$.
This resolves the ``homogeneous'' forms of (a) and (b), and
our first main theorem  gives a homogeneous version of the correspondence desired in (c).
As an application, this leads to another Littlewood-Richardson rule for the Schur $Q$-functions (see 
Corollary~\ref{lrq-cor}).
One hopes that this rule can be generalized to the $\GQ_\lambda$'s in future work.

\subsection*{Acknowledgments}

This work was partially supported by Hong Kong RGC grants ECS 26305218 and GRF 16306120.
I am especially grateful to
Travis Scrimshaw 
for many useful conversations, and for hosting  
a productive visit to the University of Queensland.
I also thank Zach Hamaker, Joel Lewis, and Brendan Pawlowski for helpful discussions.

\section{Preliminaries}\label{prelim-sect}

In this section we review some preliminary facts and background material.
Section~\ref{invw-sect} surveys the basic theory of involution words.
Section~\ref{pr-words-sect} then discusses primed words and primed involution words.
In Section~\ref{tab-sect} we set up our conventions for shifted tableaux.
Throughout, we write $\ZZ$ for the set of all integers.
When $n\in \ZZ$ is nonnegative, we let $[n] := \{i \in \ZZ : 0 < i \leq n\}$.

\subsection{Involution words}\label{invw-sect}

We use the term \defn{word} to mean any finite sequence of integers $a=a_1a_2\cdots a_n$. 
We write $\ell(a) := n$ for the length of a word.
Recall from the introduction that $\cR(\sigma)$ denotes
 the set of reduced words for a permutation $\sigma \in S_\ZZ:= \langle s_i : i \in \ZZ\rangle$,
 while $\iR(z)$  denotes the set of involution words   for an involution $z=z^{-1} \in S_\ZZ$.

Let $\simA$ be the equivalence relation on words that has 
$a X(X+1)X b \simA a (X+1)X(X+1) b$ and 
$a XYb \simA a YXb $
 for all words $a$, $b$ and all $X,Y \in \ZZ$ with $|X-Y| > 1$. 
For each $\sigma \in S_\ZZ$, the set $\cR(\sigma)$ is an equivalence class under $\simA$,
and an arbitrary word belongs to $\cR(\sigma)$ for some $\sigma \in S_\ZZ$ if and only if its $\simA$-equivalence class
contains no words with equal adjacent letters \cite[\S3.3]{CCG}.
We review a similar result that holds for  involution words.

 Let 
 $I_\ZZ := \{ \sigma \in S_\ZZ : \sigma=\sigma^{-1}\}$ and   $I_n := S_n \cap I_\ZZ$ when $0<n \in \ZZ$. 
If $z \in I_\ZZ$ and $i \in \ZZ$ then  
$s_i \circ z \circ s_i=z$ when $z(i)>z(i+1)$,  
while $s_i \circ z \circ s_i=zs_i = s_iz$ when $i$ and $i+1$ are fixed points of $z$, 
and otherwise $s_i \circ z \circ s_i=s_izs_i$.
It follows (see \cite[Lem. 2.1]{HMP1}) that if $z \in I_\ZZ$ and $a_1,a_2,\dots,a_n \in \ZZ$
then the word
$a_1a_2\cdots a_n$ belongs to $ \iR(z)$ if and only if 
$z =  s_{a_n}\circ \cdots \circ s_{a_2}\circ s_{a_1}\circ s_{a_2}\circ \cdots\circ s_{a_n}$ 
 and for each $i \in[n]$ it holds that
\[(s_{a_{i-1}}\circ \cdots \circ s_{a_2} \circ s_{a_1} \circ s_{a_2} \circ \cdots\circ s_{a_{i-1}})(a_i) < (s_{a_{i-1}}\circ \cdots \circ s_{a_2} \circ s_{a_1} \circ s_{a_2}  \circ \cdots\circ s_{a_{i-1}})(1+a_i).\]
For example, we have $1232 \in \iR(4321)$ since
$s_1 = 2134$, $s_2\circ s_1 \circ s_2 = s_2s_1s_2= 3214$, $s_3\circ s_2\circ s_1 \circ s_2\circ s_3 = s_3s_2s_1s_2s_3=4231$,
and
$s_2\circ s_3\circ s_2\circ s_1 \circ s_2\circ s_3 \circ s_2 = s_3s_2s_1s_2s_3 s_2 = 4321$.

\begin{lemma}\label{new-lem}
 Suppose $z\in S_\ZZ$ has $z(i) > z(i+1)$ for some $i \in \ZZ$. Then some $a \in \iR(z)$ ends in $i$.
\end{lemma}

\begin{proof}
Let $y = zs_i = s_i z$ if $z(i) = i+1$ and otherwise let $y=s_izs_i$. 
Then $y\in I_\ZZ$ and adding $i$ to any of its involution words gives an involution word $z$
in view of the remarks above.
\end{proof}

Define $\iisim$ to be the transitive closure of $\simA$  and the relation with
$
XYa \iisim  YXa $
for all words $a$ and all letters $X,Y\in \ZZ$.
Hu and Zhang   prove the first claim in the following result
in \cite{HuZhang1}:
 
 \begin{proposition}[{\cite{HuZhang1}}] \label{mat3-prop}
Each set $\iR(z)$ for $z \in I_\ZZ$ is an equivalence class under $\iisim$.
An arbitrary word is an involution word for some element of $ I_\ZZ$ if and only if its $\iisim$-equivalence class
contains no words with equal adjacent letters.
\end{proposition}

For example, we have $ \iR(3412) = \{ 132 \iisim 312\}$
and $\iR(4231) = \{
123 \iisim 213 \iisim 231 \iisim 321
\}$.

\begin{proof}
The first assertion is \cite[Thm. 3.1]{HuZhang1}.
The second assertion may be proved from the first by induction in the following way.
Suppose $a_1a_2\cdots a_n$ is a word whose $\iisim$-equivalence class
contains no words with equal adjacent letters.
Then the subword $a_1a_2\cdots a_{n-1}$ has the same property, so by induction it is an involution word for
some $z \in I_\ZZ$. 
By the remarks before Lemma~\ref{new-lem}, to show that $a_1a_2\cdots a_n$ is an involution word (necessarily for $s_{a_n}\circ z \circ s_{a_n}$)
it suffices to check that $z(a_n) < z(1+a_n)$. But if this inequality does not hold then $z$ has an involution word $b_1b_2\cdots b_{n-1}$
with $b_{n-1} = a_n$ by Lemma~\ref{new-lem}, and by induction $a_1a_2\cdots a_{n-1} \iisim b_1b_2\cdots b_{n-1}$,
so $a_1a_2\cdots a_{n} \iisim b_1b_2\cdots b_{n-1}a_n$, contradicting our hypothesis
about the $\iisim$-equivalence class of $a_1a_2\cdots a_{n}$.
\end{proof}

\subsection{Primed words}\label{pr-words-sect}

Let $\ZZ' :=\ZZ-\frac{1}{2}$ and given $i \in \ZZ$ define $i' := i - \frac{1}{2} \in \ZZ$.
This convention means that 
$(i+1)' = i' + 1$ and $\lceil i'\rceil = \lceil i \rceil = i$ for all $i \in \ZZ$, and that
$\ZZ\sqcup \ZZ' = \{ \cdots  < 0' < 0 < 1'<1<2'<2< \cdots\}= \frac{1}{2}\ZZ$.
We refer to elements of $\ZZ'$ as \defn{primed letters},
and 
we view all primed involution words in $\iR^+(z)$ as finite sequence of elements of $\frac{1}{2}\ZZ$.

``Removing the prime'' from $x \in \ZZ\sqcup \ZZ' $
means to replace $x$ by $\lceil x \rceil$. 
``Reversing the prime'' on  $x \in \ZZ\sqcup \ZZ' $
means to replace $x$ by the unique element of $\{ \lceil x \rceil- \frac{1}{2},  \lceil x \rceil \}  \setminus\{x\}$,
so that $i \in \ZZ$ becomes $i' \in \ZZ'$ and vice versa.
When working with a pair of numbers $x,y \in \ZZ\sqcup \ZZ'$,
we will refer to the operation that reverses the primes on both numbers if exactly one is unprimed and 
leaves them unchanged otherwise as ``switching their primes.''

We use the term \defn{primed word} to mean a finite sequence $a=a_1a_2\cdots a_n$
with letters $a_i \in \ZZ\sqcup \ZZ'$.
The \defn{unprimed form} of $a$ is the word $\unprime(a) := \lceil a_1\rceil \lceil a_2\rceil \cdots \lceil a_n\rceil$ obtained by removing the primes from all letters.

Let $z \in I_\ZZ$. In the introduction we defined an index $i$ to be a \defn{commutation} in an involution word $a_1a_2\cdots a_n  \in \iR(z)$ if 
$s_{a_i}$ commutes with $y:=s_{a_{i-1}}\circ \cdots \circ s_{a_2} \circ s_{a_1}  \circ s_{a_2} \circ\cdots \circ s_{a_{i-1}}$.
Because $s_{a_i}$ must also be a left and right ascent of $y$, it follows that $i \in [n]$ is commutation in $a_1a_2\cdots a_n $
if and only if 
 $a_i$ and $1+a_{i}$ are both fixed points of $y$,
 in which case $(a_i,1+a_i)$ is a 2-cycle of $s_{a_i}\circ y \circ s_{a_i}=s_{a_i}y=ys_{a_i}$. 
 On the other hand, if $i$ is not a commutation then $s_{a_i}\circ y \circ s_{a_i} = s_{a_i} y  s_{a_i}$
 has the same number of 2-cycles as $y$.
Thus the number of commutations in $a_1a_2\cdots a_n$ is the number 
 of 2-cycles of $z$.
 
 Recall from the introduction that the set of \defn{primed involution words} $\iR^+(z)$ consists of all primed words formed by adding primes to letters 
 indexed by commutations in involution words. 
 
 \begin{remark}\label{o-sp-rmk}
As explained in \cite[\S2.2-2.3]{Wyser} or \cite[\S8.1]{HM2021}, 
the set  $I_n \subset S_n$ indexes the orbits of the orthogonal group $\O_n(\CC)$ acting on the type $\mathsf{A}_{n-1}$ flag variety $\Fl_n:=\GL_n(\CC)/B$.
In \cite{Brion2001}, Brion derives a formula for the cohomology classes of the closures of these orbits, involving a certain directed graph
on the set of orbits. The directed paths that arise in Brion's cohomology formula
(from the orbit indexed by $z$ to the unique dense orbit) are in bijection with $\iR^+(z)$.
This is one motivation for studying these sets.
This is also why we will often include the adjective ``orthogonal'' with  constructions involving $\iR^+(z)$. 
There is a parallel ``symplectic'' story for 
a different analogue of reduced words corresponding to the orbits of $\Sp_{2n}(\CC)$ acting on $\Fl_{2n}$ (see, e.g., \cite{HMP5,Marberg2019a,MP2021,WyserYong}).
\end{remark}

In a few places we will need the following additional properties of commutations from \cite{Marberg2021}.

\begin{proposition}[{\cite[Prop. 8.2]{Marberg2021}}] \label{despite-prop}
Let $a=a_1a_2\cdots a_n  \in \iR^+(z)$  for some $z\in I_\ZZ$.
\ben
\item[(a)] Suppose $\lceil a_i \rceil =  \lceil a_{i+1} \rceil \pm 1$ for  $i \in [n-1]$. Then at most one of 
 $a_i$ or $a_{i+1}$ is primed, so at most one of the indices
$i$ or $i+1$ is a commutation in  $a$, and if $i=1$ then $a_{i+1}\in \ZZ$.

\item[(b)] Suppose $\lceil a_{i}\rceil = \lceil a_{i+2}\rceil$ for  $i \in [n-2]$. Then $i>1$,
$a_{i+1} = \lceil a_{i}\rceil \pm 1\in  \ZZ$, and at most one of 
 $a_i$ or $a_{i+2}$ is primed, so at most one of the indices
$i$ or $i+2$ is a commutation in  $a$.
\een
\end{proposition}

Write $\isim$ for the transitive closure of the relation with
$ aXYb \isim aYXb
$ for all $X,Y\in \ZZ\sqcup \ZZ'$
such that $|\lceil X\rceil- \lceil Y\rceil | > 1$, as well as with
$
aXYXb \isim aYXYb $ and $ aX'YXb\isim aYXY'b
$
for  unprimed numbers $X,Y \in \ZZ$ such that $|X-Y|=1$, 
and finally with
$ Xa \isim X'a$ and $ XYa\isim YXa 
 $
for unprimed numbers $X,Y \in \ZZ$. In these relations $a$ and $b$ are arbitrary primed words.
For example, we have
$ 1'232' \isim 1'3'23 \isim 13'23 \isim 3'123 \isim 3123 \isim 1323 \isim 1232
\isim 2132
\isim 2312 
\isim 3212\isim 3121.
$
The following extension of Proposition~\ref{mat3-prop} is a corollary of more general results in \cite{Marberg2021}.

 \begin{proposition}[{\cite[Cor. 8.3]{Marberg2021}}] 
 \label{imat-prop}
Each set $\iR^+(z)$ for $z \in I_\ZZ$ is an equivalence class under $\isim$.
\end{proposition}

\subsection{Tableaux}\label{tab-sect}

A \defn{partition} of an integer $n\geq 0$ is a finite 
sequence of integers $\lambda =(\lambda_1 \geq \lambda_2 \geq \dots \geq \lambda_k >0)$
that sum to $n$. In this event we set $\ell(\lambda):= k$,  $\lambda_i :=0$ for $i>\ell(\lambda)$, and $|\lambda|:=\sum_i\lambda_i=n $.
A partition is \defn{strict} if the parts $\lambda_i$ are all distinct.
The \defn{diagram} of a partition $\lambda$ is 
the set of positions $\D_\lambda:= \{ (i,j) \in \ZZ\times \ZZ :1 \leq j \leq \lambda_i\}.$
The \defn{shifted diagram} of a strict partition $\mu$
is the set $\SD_\mu:= \{ (i,i+j-1) : (i,j) \in \D_\mu\}$.

In this article, a \defn{tableau} of shape $\lambda$  means an arbitrary map $ \D_\lambda \to \ZZ$
and a \defn{shifted tableau} of shape $\mu$ means an arbitrary map $ \SD_\mu \to \ZZ \sqcup \ZZ'$.
If $T$ is a (shifted) tableau then we write $T_{ij}$ for the value assigned to some position $(i,j)$ in its domain.
 The \defn{(main) diagonal} of a shifted tableau
is the set of positions $(i,j)$ in its domain with $i=j$.
We often refer to positions $(i,j)$ in the domain of a tableau as 
its \defn{boxes}.

A (shifted) tableau is \defn{increasing} if its rows and columns are strictly increasing as indices increase.
An increasing (shifted) tableau of shape $\lambda$ is \defn{standard} 
if it contains an entry equal to $i$ or $i'$ for each $i \in [|\lambda|]$.
A (shifted) tableau is \defn{semistandard} if its entries are all positive, its rows and columns are weakly increasing,
no primed entry is repeated in a row, and no unprimed entry is repeated in a column.\footnote{
Semistandard shifted tableaux are sometimes required to have no primed entries on the main diagonal,
or no primed entries in any boxes. 
Our conventions, which do not impose either condition, follow
references like \cite{Sag87,Worley}.} 
We draw tableaux in French notation, so that row indices increase from bottom to top and column indices increase from left to right.
If
\be\label{tableau-ex}
A = \ytab{  4 \\  3 & 3& 7 \\ 1 & 1 & 6 & 6},
\quad
S = \ytab{ 8 \\  3 & 5& 7 \\ 1 & 2 & 4 & 6},
\quad 
B = \ytab{ \none & \none & 8 \\ \none & 2' & 7 & 7 \\ 1' & 2' & 4' & 6},
\quand
T = \ytab{ \none & \none & 8' \\ \none & 3 & 5' & 7 \\ 1' & 2' & 4' & 6},
\ee
then 
$A$ is a semistandard tableau and $B$ is a semistandard shifted tableau,
while 
$S$ is a standard tableau and $T$ is a standard shifted tableau. All four tableaux are of shape $\lambda=(4,3,1)$.
We have $A_{23}= B_{23} = S_{23} = 7$ while $T_{23} = 5'$.

Suppose $T$ is a tableau, or more generally any map from a finite subset of $\ZZ\times \ZZ$ to a totally ordered set.
The \defn{row reading word} of $T$ is the sequence $\row(T)$ formed by
reading the entries of $T$  from left to right, row by row, starting with the top row (in French notation).
Above, we have $\row(A) = 43371166$, $\row(S) = 83571246$,  $\row(B) = 82'771'2'4'6$, and $\row(T) = 8'35'71'2'4'6$.
The \defn{column reading word} of $T$ is the sequence $\col(T)$ formed by
reading the entries of $T$ from top to bottom, column by column, starting with the first column.
Above, we have $\col(A) = 43131766$, $\col(S) = 83152746$,  $\col(B) = 1'2'2'874'76$, and $\col(T) = 1'32'8'5'4'76$.

When $T$ is a shifted tableau, we form $\unprime( T )$
 by removing all primes from $T$'s entries.

\begin{proposition}\label{unprime-incr-prop}
Suppose $T$ is a shifted tableau such that $\row(T)$ or $\col(T)$ is a primed involution word
 for an element of $I_\ZZ$.
Then $T$ is increasing if and only  $\unprime( T )$ is increasing.
\end{proposition}

\begin{proof}
If $\unprime(T)$ is increasing then $T$ is clearly increasing.
Assume that $T$ is increasing and $\row(T)$ is a primed involution word.
Since $\row(\unprime(T))$ is an involution word and therefore reduced, 
no row of $T$ can contain both $x \in \ZZ$ and $x' \in \ZZ'$, so the rows of $\unprime(T)$ are (strictly) increasing.
It remains to show that this property also applies to the columns of $T$.
Arguing by contradiction, suppose there is a box 
$(i,j) \in T$ such that $T_{ij} = x$ and $T_{i-1,j} = x'$ for some $x \in \ZZ$.
Assume $(i,j)$ is the first such box in the row reading order, so that the box is maximally northwest in French notation.

Let $l \geq 0$ be maximal such that $(i,j+l)$ is occupied in $T$ with  $ T_{i,j+l}  \leq x+l$.
Then we must have $T_{i-1,j+k}=x+k'$ and $T_{i,j+k}=x+k$ for each $0\leq k \leq l$ since $T$ is increasing and $\unprime(T)$ has increasing rows.
If $l>0$ then box $(i,j+l+1)$ is either unoccupied in $T$ or filled with a number greater than $x+l+1$.
In this case,
we can use $\isim$ to commute $T_{i,j+l}=x+l$ to the right in  $\row(T)$ past the remaining letters in row $i$ and then also past the letters in columns $i-1,i,\dots,j+l-2$ of row $i-1$
to obtain a primed involution word with $(x+l)(x+l-1)'(x+l)'$ as a consecutive subsequence.
This is impossible by Proposition~\ref{despite-prop}, so we conclude that $l=0$.

Having $l=0$ means that box $(i,j+1)$ is either unoccupied in $T$ or filled with a number greater than $x+1$.
It therefore follows by similar reasoning that $\lceil T_{i-1,j-1} \rceil = x-1$,
as otherwise $\row(T)$ would be equivalent under $\isim$ to a primed involution word with adjacent letters equal to $x$ and $x'$, which is impossible.
We now reach one of two contradictions.
If $i=j$ then we can use $\isim$ to commute $T_{ij}$, $T_{i-1,j-1}$, and $T_{i-1,j}$ past all earlier lettters in $\row(T)$ 
to obtain a primed involution word starting with $T_{ij} T_{i-1,j-1} T_{i-1,j} \in \{ x(x-1)x', x(x-1)'x'\}$, which contradicts Proposition~\ref{despite-prop}(b).
If instead $i < j$, then since we cannot have  $T_{i,j-1}=x'$ as the rows of $\unprime(T)$ are increasing, the inequalities $x'-1\leq T_{i-1,j-1} < T_{i,j-1} <T_{ij}=x$
can only hold if $T_{i-1,j-1} = x'-1$ and $T_{2,j-1}=x-1$, which contradicts the minimality of $(i,j)$.
We conclude that $\unprime(T)$ is increasing.

The argument to show that $\unprime(T)$ is increasing when $T$ is increasing and $\col(T)$ is a primed involution word
 is similar to the previous case. One simply ``conjugates'' all of the preceding statements,
where if  $T$ is contained in the square $[N-1]\times[N-1]$, then \defn{conjugation} applies the transformation
 $(i,j) \mapsto (N-j,N-i)$ to the boxes of $T$ and   $x \mapsto 1' - x$ to the entries of $T$.
\end{proof}

In the following lemma, 
let $\simK$ denote the transitive closure of the symmetric relation on primed words that 
has $  uACB v \simK u CAB w$ and $u BCA v \simK u BAC v$ 
whenever $u$ and $v$ are primed words and $A,B,C \in \ZZ\sqcup\ZZ'$ are such that $\lceil A\rceil <\lceil B\rceil <\lceil C\rceil $.  
This is similar to \defn{(strict) Knuth equivalence}.

\begin{lemma}\label{simK-lem}
Let $T$ be a shifted tableau. If $\unprime(T)$ is increasing then $\row(T) \simK \col(T)$.
Consequently, if $T$  is increasing and $z \in I_\ZZ$, then $\row(T) \in \iR^+(z)$ if and only if $\col(T) \in \iR^+(z)$,
and in this case $\row(T) \simK \col(T)$.
\end{lemma}

\begin{proof}
Let $w$ be the last column of $T$ read in reverse order.
Construct $U$ from  $T$ by removing the last column. Then by induction $\col(T) = \col(U) w \simK \row(U) w$ and 
it remains to check that $\row(U) w \simK \row(T)$. 
For this, observe that if $T$ has $j$ columns and $i := \ell(w)$, then starting from $\row(T)$,
we can use $\simK$ first to commute $w_1 = T_{ij}$ to the right past the entries in columns $i-1,i,\dots,j-1$ of row $i-1$,
 then to commute $w_2= T_{i-1,j}$ followed by $w_1$ to the right past the entries in columns $i-2,i-1,\dots,j-1$
of row $i-2$, and so forth, until we are left with $\row(U)$ followed by $w$.

If $T$  is increasing and $\row(T)$ or $\col(T)$ is in $ \iR^+(z)$, then $\unprime(T)$ is increasing 
by Proposition~\ref{unprime-incr-prop},
so $\row(T) \simK \col(T)$ and both reading words are in $\iR^+(z)$ as $u\simK v$
implies $u \isim v$.
\end{proof}

\section{Shifted Edelman-Greene insertion}\label{main-sect}

This section contains our main results, 
 which are organized around 
 a shifted version of \defn{Edelman-Greene insertion} \cite{EG}
 that sends primed involution words
to pairs of shifted tableaux.
Section~\ref{shifted-eg-sect} gives the precise definition of this insertion algorithm,
along with some examples and basic properties.
Section~\ref{semist-sect} then describes its ``semistandard'' extension.
Section~\ref{qapp-sect} explains an application of the semistandard algorithm
to formulating a Littlewood-Richardson rule for Schur $Q$-functions.
Sections~\ref{ock-sect} and \ref{dual-sect}
explore some related operators on primed involution words and standard shifted tableaux.
 Section~\ref{marked-sect}, finally, examines how the primes in a primed involution word may be used to 
 label the 2-cycles of the corresponding involution.
 
\subsection{Definitions for the standard case}\label{shifted-eg-sect}

This section give the definition of \defn{orthogonal Edelman-Greene insertion} and a few of its basic properties.
Suppose $T$ is an increasing shifted tableau with no primed entries on the main diagonal
and a number $u\in \ZZ\sqcup \ZZ'$ such that $\row(T)u \in \iR^+(z)$ for some $z \in I_\ZZ$.
We first explain how to insert $u$ into $T$ to 
obtain another shifted tableau $T\iarrow u$ that is increasing with no  primed entries on the main diagonal.
Later, we will see that this new tableau also has $\row(T\iarrow u) \in \iR^+(z)$.

\begin{definition}
\label{iarrow-def}
Suppose $T$ is an increasing shifted tableau with no primed entries on the main diagonal  
and $u\in \ZZ\sqcup \ZZ'$ is   such that
$\row(T)u$ is a primed involution word for some element of $I_\ZZ$.\footnote{
As $\row(T)$ is also a primed involution word in this case,
Proposition~\ref{unprime-incr-prop}
implies that $\unprime(T)$ is increasing.}
We construct  another shifted tableau $T\iarrow u$ by the following iterative process: 
\begin{itemize}

\item[(1)] 
On the $i$th iteration, an entry $w \in \ZZ\sqcup \ZZ'$ is inserted into row or column $i$,
which we refer to as the \defn{current segment}. The entries in the current segment will always be strictly increasing, even after removing all primes.
The process begins 
with $u$ inserted into the first row of $T$.

\item[(2)] 
Suppose  $\lceil w\rceil $ is less than some entry in the current segment.
   Let $m\leq M$ denote the smallest entries in the current segment
with $\lceil w \rceil \leq \lceil m\rceil $ and $\lceil w \rceil < \lceil M \rceil$. 
If $m< M$, then $M$ will be unprimed and in the box directly after $m$, and we will have 
$\lceil w \rceil = \lceil m\rceil =  M-1$.\footnote{
This claim only holds since we assume that $\row(T)u$ is a primed involution word; see Remark~\ref{iarrow-rmk}.}
\ben
\item[(a)] If $m=M$ is off the main diagonal then $w$ replaces $m$ and we insert $m$ into   
 the next row (respectively, column) if the current segment is a row (respectively, column).

\item[(b)] If $m=M$  is on the main diagonal then $m $ will be unprimed. In this case,  
we replace $m$ by $\lceil w \rceil$ and insert $m+1$ if $w \in \ZZ$ (respectively, $m'+1$ if $w \in \ZZ'$) into the next column.

\item[(c)]  If $m $ and $ M$ are distinct, then we switch the primes on these entries,
and continue by inserting $w+1$ into either the next column (if  $m$ is on the main diagonal or the current segment is a column) or the next row (otherwise).

\een 

\item[(3)] If $\lceil w\rceil $ is not less than some entry in the current segment, 
then we place $w$ in the segment's first empty box $(x,y)$ with $x\leq y$.\footnote{
It is not obvious, but such a box will always exists and adding it to $T$ will give the diagram of a shifted partition.
}
If $x=y$ and $w $ is primed, then we change the box's entry from $w$ to $\lceil w\rceil$
and say that the insertion process \defn{ends in column insertion}.
We also say the   process \defn{ends in column insertion} if $x <y$ and the current segment
is a column. In all other cases, the   process \defn{ends in row insertion}.
We define $T\iarrow u$ to be the result of this step. 

 \end{itemize}
If this process lasts for $N$ iterations, then we define $(x_i,y_i)$ and $(\tilde x_i, \tilde y_i)$ for $i\in [N-1]$ to be the respective positions of $m$ and $M$
in step (2) on iteration $i$, and let $(x_N,y_N)=(\tilde x_N, \tilde y_N)$ be the new box $(x,y)$ added to the tableau in step (3).
We call the sequences 
\[\path^\leq(T,u):=\((x_i,y_i): i=1,2,\dots,N\) \quand \path^<(T,u):=\((\tilde x_i,\tilde y_i) :i=1,2,\dots,N\)\]
  the \defn{weak} and \defn{strict bumping paths} 
that result from inserting $u$ into $T$.
\end{definition}

\begin{example}
\def\gap{\leadsto\ }
The following examples illustrate most of the cases that can occur in Definition~\ref{iarrow-def}.
\ben
\item[(a)] 
If 
$T = \ytab{ 1 & 3 & 4 }$ and $u=2$ then  $T\iarrow u$ is computed as
\[ \ba   \ytab{ \none & \none & \none \\ 1 & 3 & 4  & \none & \none[\leftarrow 2\ \ \ \ \ ] }
&\gap
  \ytab{ \none & \none & \none & \none& \none[\leftarrow 3\ \ \ \ \ ] \\ 1 & 2 & 4  & \none  }
\gap
  \ytab{
 \none & 3   \\ 1 & 2 & 4     }= T \iarrow u.
\ea
\]
Here, the insertion process ends in row insertion and the bumping paths are 
\[ \ba
\path^\leq(T,u) = \path^<(T,u) = \( (1,2),(2,2) \).
\ea
\]

\item[(b)] If
$T = \ytab{ 1 & 3' & 4 }$ and $u=2$ then  $T\iarrow u$ is computed as
\[ \ba  \ytab{ \none & \none & \none \\ 1 & 3' & 4  & \none & \none[\leftarrow 2\ \ \ \ \ ] }
&\gap
  \ytab{ \none & \none & \none & \none& \none[\leftarrow 3'\ \ \ \ \ ] \\ 1 & 2 & 4  & \none  }
\gap
  \ytab{
 \none & 3   \\ 1 & 2 & 4     }= T \iarrow u.
\ea
\]
Here, the insertion process ends in column insertion and the bumping paths are 
\[ \ba
\path^\leq(T,u) = \path^<(T,u) = \( (1,2),(2,2) \).
\ea
\]

\item[(c)] If $T = \ytab{ \none & 4 & 5 \\ 1 & 3 & 4' }$ and $u=2$ then  $T\iarrow u$ is computed as
\[ \ba  \ytab{ \none \\ \none & 4 & 5 \\ 1 & 3 & 4'  & \none & \none[\leftarrow 2\ \ \ \ \ ] }
&\gap
  \ytab{ \none \\ \none & 4 & 5 & \none& \none[\leftarrow 3\ \ \ \ \ ] \\ 1 & 2 & 4'  & \none  }
\gap
  \ytab{\none &\none & \none[\barr{c} \\  4 \\ \downarrow \earr] \\
\none &\none & \none \\
 \none & 3 & 5   \\ 1 & 2 & 4' &\none  \\ \none   }
 \gap
  \ytab{\none &\none&\none & \none[\barr{c} \\  5 \\ \downarrow \earr] \\
\none &\none & \none \\
 \none & 3 & 5'   \\ 1 & 2 & 4  & \none \\ \none  }
 \\&\gap
  \ytab{
 \none & 3 & 5'   \\ 1 & 2 & 4  & 5  
} = T \iarrow u.
\ea
\]
In this case the insertion process ends in column insertion and the bumping paths are 
\[ \ba
\path^\leq(T,u) = \( (1,2),(2,2),(1,3),(1,4) \)
\text{ and }
\path^<(T,u) = \( (1,2),(2,2),(2,3),(1,4) \).
\ea
\]

\item[(d)]  If $T= \ytab{ \none & 5 & 6 \\ 1 & 3' & 4}$ and $u=2$ then  $T\iarrow u$ is computed as
\[ \ba   \ytab{\none \\ \none & 5 & 6 \\ 1 & 3' & 4  & \none & \none[\leftarrow 2\ \ \ \ \ ] }
&\gap
  \ytab{ \none \\ \none & 5 & 6 & \none& \none[\leftarrow 3'\ \ \ \ \ ] \\ 1 & 2 & 4  & \none  }
\gap
  \ytab{\none &\none & \none[\barr{c} \\  5' \\ \downarrow \earr] \\
\none &\none & \none \\
 \none & 3 & 6   \\ 1 & 2 & 4  & \none \\ \none   }
 \gap
  \ytab{\none &\none&\none & \none[\barr{c} \\  6 \\ \downarrow \earr] \\
\none &\none & \none \\
 \none & 3 & 5'   \\ 1 & 2 & 4  & \none \\ \none  }
\\&  \gap
  \ytab{
 \none & 3 & 5'   \\ 1 & 2 & 4  & 6
} = T \iarrow u.
\ea
\]
In this case the insertion process ends in column insertion and the bumping paths are 
\[ \ba
\path^\leq(T,u)=\path^<(T,u)  = \( (1,2),(2,2),(2,3),(1,4) \).
\ea
\]
\een
\end{example}

Proposition~\ref{o-lem2}
will show that if $T$ and $u$ are as in Definition~\ref{iarrow-def} then $\row(T\iarrow u)$ is also a primed involution word.
We can therefore iterate the above insertion process as follows:

\begin{definition}
\label{o'-eg-def}
Given a primed involution word $a=a_1a_2\cdots a_n$ for some element of $I_\ZZ$,
let $\PO(a)$ be the shifted tableau $\emptyset \iarrow a_1 \iarrow a_2 \iarrow \cdots \iarrow a_n$ and let $\QO(a)$ be the standard shifted tableau with the same shape
as $\PO(a)$ 
that has $i$ (respectively, $i'$) in the box added by inserting   $a_i$ into $\PO(a_1a_2\cdots a_{i-1})$ when this
ends in row insertion (respectively, column insertion).
\end{definition}

We refer to $a \mapsto (\PO(a),\QO(a))$ as \defn{orthogonal Edelman-Greene insertion}.
There is a similar correspondence called \defn{symplectic Edelman-Greene insertion},
with a different domain containing only unprimed words, which is denoted $a \mapsto (\PSp(a),\QSp(a))$ in \cite[Def. 3.23]{Marberg2019b}.
For more about the connection between these maps and the orthogonal and symplectic groups, see Remark~\ref{o-sp-rmk}.
\begin{example}\label{po-ex}
The  words $a=134524'$, $b=5'431'4'2$, and $c=41'354'2$ all have
\[
\PO(a) = \PO(b)=\PO(c) =  \ytab{
 \none & 3 & 5'   \\ 1 & 2 & 4  & 5  
}
\]
while
$
\QO(a) =  \ytab{
 \none & 5 & 6  \\ 1 & 2 & 3  & 4
},
$
$
\QO(b) =  \ytab{
 \none & 5' & 6'  \\ 1' & 2' & 3'  & 4'
},
$
and
$
\QO(c) =  \ytab{
 \none & 3' & 5  \\ 1 & 2' & 4  & 6'
}.
$
\end{example}

\begin{remark}
The original Edelman-Greene correspondence $a \mapsto (P_\EG(a),Q_\EG(a))$
from \cite{EG}, sending
 reduced words $a \in \cR(\sigma)$ for $\sigma \in S_n$ to pairs of unshifted tableaux of the same shape,
may be embedded in Definition~\ref{o'-eg-def} in the following way. 
Fix $\sigma \in S_n$ and choose
an involution word $b$ for 
\[z := (0,n)(-1,n-1)(-2,n-2)\cdots(-n+1,1) \in I_\ZZ.\]
Then $a \mapsto ba$ is an injective map $\cR(\sigma) \hookrightarrow \iR(\sigma^{-1}z\sigma)$, 
and when we carry out the bumping process to compute $\PO(ba)$,
the first $\ell(b)$ insertions will result in a shifted tableau of shape $(n,\dots,3,2,1)$
whose last column is $0,1,2,\dots,n-1$.
This part of the insertion tableau $\PO(ba)$ will remain fixed during the remaining $\ell(a)$ insertions,
which will only involve row bumping operations that follow the rules of the original Edelman-Greene correspondence.
We recover $P_\EG(a)$ from $\PO(ba)$ by omitting the first $n$ columns,
while $Q_\EG(a)$ is given by omitting the first $n$ columns of $\QO(ba)$ and subtracting $\ell(b)$ from the remaining entries, 
which are all unprimed numbers.
\end{remark}

\begin{example}
When $n=4$ we can take $ b=-3,-1,-2,1,0,-1,3,2,1,0$. Then for the reduced word $a= 23121 \in \cR(3412)$,
we have
\[
\PO(ba) = \ytab{
\none & \none & \none & 3 \\
\none & \none & 1 & 2 & 3 \\
\none & -1 & 0 & 1 & 2 & 3 \\
-3 & -2 & -1 & 0 & 1 & 2 
}
\quand
\QO(ba) = \ytab{
\none & \none & \none & 10 \\
\none & \none & 6 & 9 & 15 \\
\none & 3 & 5 & 8 & 13 & 14 \\
1 & 2 & 4 & 7 & 11 & 12 
}
\]
compared to $P_\EG(a) = \smalltab{3 \\ 2 & 3 \\ 1 & 2}$
and $Q_\EG(a) = \smalltab{5 \\ 3 &4 \\ 1 & 2}$.
\end{example}

As noted in the introduction,
$a \mapsto (\PO(a),\QO(a))$
restricted to unprimed involution words
reduces to a map 
previously studied in \cite{HMP4,Marberg2019a,Marberg2019b}.
Our inclusion of primes
may seem like a minor generalization.
However, there seems to be no simple way to derive our main results as corollaries 
of what is known about the unprimed form of orthogonal Edelman-Greene insertion.

\begin{remark}\label{iarrow-rmk}
Suppose $T$ is an increasing shifted tableau with no primed entries on the main diagonal  
and $u\in \ZZ\sqcup \ZZ'$ is   such that
$\row(T)u\in \iR^+(z)$ for some $z \in I_\ZZ$.
Since $a \mapsto (\PO(a),\QO(a))$ restricted to unprimed words coincides with \cite[Def. 3.20]{Marberg2019b},
 \cite[Rem. 3.25]{Marberg2019b}
implies the following properties  concerning the process to construct $T \iarrow u$,
stated in the notation of Definition~\ref{iarrow-def}:
\bei
\item[(a)] Denote the intermediate tableau created by the $i$th iteration in Definition~\ref{iarrow-def} by $T_i$, so that $T=T_0$ and $T \iarrow u = T_N$ if $N>0$ is the length of the two bumping paths.
Then each $T_i$ is a shifted tableau with no primes on the main diagonal, and $\unprime(T_i)$ is increasing.

\item[(b)] If    $m$ and $M$ in step (2) on iteration $i$ are distinct, then the boxes $(x_i,y_i) $ and $ (\tilde x_i, \tilde y_i)$ containing these entries are adjacent, 
and the number $w$ being inserted has $\lceil w \rceil = \lceil m \rceil = \lceil M\rceil - 1$.

\item[(c)] Suppose  $m$ and $M$ in step (2) on iteration $i$ are distinct and $m$ is on the main diagonal.
Then 
$m=T_{ii}$ is unprimed and 
$M= T_{i,i+1}$, and we have $ T_{i+1,i+1} = \lceil M\rceil +1 = m +2$.

\eei
There is one final property that will be demonstrated in the proof of Proposition~\ref{o-lem2}:
\bei
\item[(d)] If the $i$th iteration has a number $w$ being inserted into row (respectively, column) $i$,
then placing $w$ between rows (respectively, columns) $i-1$ and $i$
in the row (respectively, column) reading word of $T_{i-1}$ gives another primed involution word in $\iR^+(z)$ by \eqref{will-be-shown-eq0} and \eqref{will-be-shown-eq}.
In view of this observation, if the numbers $m$ and $M$ in step (2) on iteration $i$
are distinct, then since we already know that $\unprime(T_{i-1})$ is increasing and $\lceil w \rceil = \lceil m\rceil  = \lceil M \rceil - 1$,
it follows from Proposition~\ref{despite-prop} that $M$ must be unprimed and that $w$
can only be primed if $m$ is unprimed and not on the main diagonal (as if $m$ is on the diagonal then its index in the reading word mentioned above is already a commutation). 
\eei
\end{remark}

We mention some other properties of $\PO(a)$ and $\QO(a)$ that readily follow from the definitions.
Given a shifted tableau $T$, form $\unprime_{\diag}(T)$ from $T$ by removing all main diagonal primes. 

\begin{proposition}\label{unprime-tab-prop}
If $a$ is a primed involution word then
\[
\PO(\unprime(a)) = \unprime(\PO(a))\quand \QO(\unprime(a)) = \unprime_{\diag}(\QO(a)).
\]
\end{proposition}

\begin{proof}
This follows from Definitions~\ref{iarrow-def} and \ref{o'-eg-def}: if all primes are removed from $a$
then the insertion process to compute $\PO(a)$ is unchanged except that no entries added to $\PO(a)$ are primed, and
all insertions that contribute new boxes to the main diagonal must end in row insertion.
\end{proof}

The first letter in a nonempty involution word is always a commutation.
Toggling the prime on this letter also has a predictable effect on the output of orthogonal Edelman-Greene insertion.
If $i \in \ZZ$ then $\PO(i) =\PO(i')= \smalltab{i}$ while $\QO(i) = \smalltab{1}$ and $\QO(i') = \smalltab{1'}$.
More generally:

\begin{proposition}\label{first-toggle-obs}
If $a$ is a nonempty primed involution word and $b$ is formed from $a$ by toggling the prime on its first letter,
then $\PO(a) = \PO(b)$ and $\QO(b)$ is formed from $\QO(a)$ by toggling the prime on the entry in box $(1,1)$.
\end{proposition}

\begin{proof}
This again follows directly from Definition~\ref{iarrow-def}.
\end{proof}

We can also say what happens to $\PO(a)$ and $\QO(a)$ when $a_1$ and $a_2$ are interchanged.

\begin{proposition}\label{second-toggle-obs}
If $a$ is a primed involution word with at least two letters and $b$ is formed from $a$ by 
interchanging its first two letters and then switching their primes,
then $\PO(a) = \PO(b)$ and $\QO(b)$ is formed from $\QO(a)$ by toggling the prime on the entry in box $(1,2)$.
\end{proposition}

This means  that if $a=1'3'\cdots $ then $b=3'1'\cdots$, while if $a= 13'\cdots$ then $b=31'\cdots$.

\begin{proof}
This also follows directly from Definition~\ref{iarrow-def}.
\end{proof}

Our first nontrivial result about orthogonal Edelman-Greene insertion is the following theorem.

\begin{theorem}\label{o-thm1}
For each $z \in I_\ZZ$,
the map $a \mapsto (\PO(a), \QO(a))$ is a bijection from the set of primed involution words $\iR^+(z)$ to the set of pairs $(P,Q)$ of 
shifted tableaux of the same shape, in which $P$ is increasing with no primes on the main diagonal, $Q$ is standard, and $\row(P) \in \iR^+(z)$.
 \end{theorem}

The theorem remains true
when we replace  $\iR^+(z)$ by $\iR(z)$ if we further require $Q$ to have no primes  on the main diagonal
 \cite[Thm. 5.19]{HMP4}.
It is routine, following   \cite[\S3.3]{Marberg2019a} or \cite[\S5.3]{PatPyl2014}, to describe a reverse insertion algorithm that gives the inverse map $(\PO(a), \QO(a)) \mapsto a$.
However, we will end up deriving Theorem~\ref{o-thm1} by another method in Section~\ref{proof-sect4}.
For the rest of this section, we will assume that Theorem~\ref{o-thm1} is given,
and then use this to develop a few other results.

\subsection{Extension to the semistandard case}\label{semist-sect}
 
In this section, we discuss a generalization of Definition~\ref{o'-eg-def}
that outputs a pair of shifted tableaux $(P,Q)$ in which $Q$ is semistandard rather than standard.

A positive integer $i$ is a \defn{descent} of a standard shifted tableau $T$ if
either (a) $i$ and $i+1$ both appear in $T$ with $i+1$ in a row strictly after $i$,
(b) $i'$ and $i'+1$ both appear in $T$ with $i'+1$ in a column strictly after $i'$,
or (c) $i$ and $i'+1$ both appear in $T$.
Let $\Des(T)$ denote the set of descents of $T$. 
If $T$ is as in \eqref{tableau-ex}, then $\Des(T) = \{1, 3,6\}$.

\begin{lemma}\label{udiag-lem}
If $T$ is a standard shifted tableau then $\Des(T) = \Des(\unprime_{\diag}(T))$.
\end{lemma}

\begin{proof}
Form $H_i$ by reading the primed entries up column $i$ of $T$ then the  unprimed entries across row $i$.
For $T$ in \eqref{tableau-ex}, this gives $H_3 = 4'5'$, $H_2 = 2'37$, and $H_1=1'6$.
Then $i \in \Des(T)$ if and only if $i+1$ precedes $i$ in $\unprime(\cdots H_3H_2H_1)$,
which is unchanged for $T$ replaced by $\unprime_{\diag}(T)$.
\end{proof}

If $a=a_1a_2\cdots a_n$ is a primed word then let $\Des(a) := \{ i \in [n-1] : a_i > a_{i+1}\}$.

\begin{proposition}\label{o-thm3}
Let $a \in \iR^+(z)$ for some $z \in I_\ZZ$. Then $\Des(a) = \Des(\QO(a))$.
 \end{proposition}
 
 \begin{proof}
 We have $\Des(a) = \Des(\unprime(a))$ since the word $\unprime(a) \in \iR(z)$ has no equal adjacent letters.
Next,
  \cite[Prop. 2.24]{HKPWZZ} asserts that $  \Des(\unprime(a)) = \Des(\QO(\unprime(a)))$.
Finally, we have $\QO(\unprime(a)) = \unprime_{\diag}(\QO(a))$ 
by Proposition~\ref{unprime-tab-prop}
and $\Des(\unprime_{\diag}(T)) = \Des(T)$ for all standard shifted tableaux $T$ by Lemma~\ref{udiag-lem}.
 \end{proof}
 
 When $a$ is a word in a totally ordered alphabet and $N$ is a nonnegative integer,
 we let $\Incr_N(a)$ denote the set of $N$-tuples of weakly increasing, possibly empty subwords $(a^1,a^2,\cdots ,a^N)$
 such that $a=a^1a^2\cdots a^N$.
Recall from the introduction that $\Incr_\infty(a)$ is the set of infinite sequences $(a^1,a^2,\cdots )$ of weakly increasing words 
 with $a=a^1a^2\cdots $; here, all but finitely many   $a^i$ must be empty.
 If $\cA$ is a set of words and $N \in \{0,1,2,\dots\}\sqcup\{\infty\}$ then we let $\Incr_N(\cA) = \bigsqcup_{a \in \cA} \Incr_N(a)$.
 
 \begin{definition}
Given $\phi = (a^1,a^2,\cdots ) \in \Incr_N(\iR^+(z))$ for  $z \in I_\ZZ$,
let
$\PO(\phi) := \PO(a^1a^2\cdots )$ and form $\QO(\phi)$ from $\QO(a^1a^2\cdots  )$ 
by replacing each entry $j\in\ZZ$ (respectively, $j'\in \ZZ'$) by $i$ (respectively, $i'$), where $i >0$
is minimal with $j \leq \ell(a^1)+\ell(a^2) +  \dots +\ell(a^i)$.
\end{definition}

For example, if $\phi = (\emptyset, 4, 1'3, \emptyset, 5, \emptyset,  4', 2) \in \Incr_8(41'354'2)$ then 
\[
\PO(\phi) =  \ytab{
 \none & 3 & 5'   \\ 1 & 2 & 4  & 5  
}
\quand
\QO(\phi) =  \ytab{
 \none & 3' & 7  \\ 2 & 3' & 5  & 8'
}.\]
If $ (a^1,a^2,\cdots ) \in \Incr_N(\iR^+(z))$ 
then each $\unprime(a^i)$ is strictly increasing as $a^1a^2\cdots \in \iR^+(z)$.

\begin{theorem}\label{o-summary-thm}
For each $z \in I_\ZZ$,
the map $\phi \mapsto (\PO(\phi), \QO(\phi))$ is a bijection from $\Incr_\infty(\iR^+(z))$ to the set of pairs $(P,Q)$
of shifted tableaux of the same shape in which $P$ is increasing with no primes on the main diagonal, $Q$ is semistandard, 
and $\row(P) \in \iR^+(z)$.
\end{theorem}

\begin{proof}
Let $T$ be a standard shifted tableau whose shape is a strict partition of $m$
and let $\alpha = (\alpha_1,\alpha_2,\dots)$ be a weak composition of $m$ 
such that $I(\alpha) := \{ \alpha_1+\alpha_2 + \dots +\alpha_i : i \geq 1\} \setminus\{m\}$
contains $\Des(T)$.

We claim that such pairs $(T,\alpha)$ are in bijection with semistandard shifted tableaux
via the map that replaces $j$ (respectively, $j'$) in $T$
by $i$ (respectively, $i'$)
where $i>0$ is minimal with $j \leq \alpha_1 +\alpha_2+ \dots +\alpha_i$.
The shifted tableau $U$ obtained from $(T,\alpha)$ in this way is semistandard
because $i \notin \Des(T)$ implies that $i$ and $i+1$ do not appear in the same column of $T$,
that $i'$ and $i'+1$ do not appear in the same row of $T$, and that $T$ does not contain both $i$ and $i'+1$.
In the reverse direction, one can recover $\alpha$ from $U$ as the sequence whose $i$th entry is the number of 
boxes containing $i$ or $i'$, and one can recover $T$ from $U$ by the \defn{standardization} process
that replaces each vertical strip of boxes containing $i'$ by consecutive primed numbers and each horizontal strip
of boxes containing $i$ by consecutive unprimed numbers.

By Proposition~\ref{o-thm3}, if $\phi =(a^1,a^2,\dots)\in \Incr_\infty(\iR^+(z))$,
then $\QO(\phi)$ is obtained by applying this bijection to $(T,\alpha)$
for $T = \QO(a^1a^2\cdots a^n)$
and $\alpha = (\ell(a^1),\ell(a^2),\dots)$. 
Given this observation, we deduce that 
 $\phi \mapsto (\PO(\phi), \QO(\phi))$ is injective and surjective from
Theorem~\ref{o-thm1}.
\end{proof}

 \subsection{Application to multiplying Schur $Q$-functions}\label{qapp-sect}

In this section,
we explain an application of Theorem~\ref{o-summary-thm} mentioned in the introduction.
Let $x_i$ for $i \in \ZZ$ be commuting indeterminates. 
Given a shifted tableau $T$, let $x^T :=\prod_{i \in \ZZ} x_i^{c_i}$
where $c_i$ is the number of entries in $T$ equal to $i$ or $i'$.
 The \defn{Schur $Q$-function} of a strict partition $\lambda$ is the formal power series 
 $Q_\lambda:= \sum_T x^T \in \ZZ[[x_1,x_2,\dots]]$ where $T$ ranges over all semistandard shifted tableaux of shape $\lambda$.
The Schur $Q$-functions are symmetric in the $x_i$ variables and linearly independent \cite{S28}. 
We present a new proof that they span a ring with nonnegative integer structure coefficients.

For $z \in I_\ZZ$, let $Q_z := \sum_{\phi \in \Incr_\infty(\iR^+(z))} x^\phi$
where $x^\phi := x_1^{\ell(a_1)} x_2^{\ell(a_2)}  \cdots$ if $\phi = (a^1,a^2,\dots)$. 
These power series are denoted $\hat G_z$ in \cite[\S4.5]{HMP4}.
The following is immediate from Theorem~\ref{o-summary-thm}.

\begin{corollary}[{\cite[Cor. 4.62]{HMP4}}]
\label{q-cor1}
We have $Q_z = \sum_{T \in \{ \PO(a) : a \in \iR^+(z)\}} Q_{\mathsf{shape}(T)}$.
\end{corollary}

Suppose $\lambda$ is a strict partition and $T_\lambda$ is the increasing  shifted tableau of shape $\lambda$
whose entry in box $(i,j)$ is $i+j-1$.
There exists a unique element $z \in I_\ZZ$ (called the  \defn{dominant} involution of shape $\lambda$) whose 
 \defn{involution Rothe diagram} $\hat\D(z) := \{ (i,j) \in \ZZ\times \ZZ : z(i) > j \leq i < z(j)\}$
 coincides with the transpose of $\SD_\lambda$  \cite[Prop. 4.16]{HMP6}.
 If we denote this element by $z_\lambda \in I_\ZZ$, then 
 $\row(T_\lambda)$ and $\col(T_\lambda)$ are both in $ \iR(z_\lambda)$ by \cite[Thm. 3.9 and Prop. 4.15]{HMP6}.\footnote{
 In the terminology of \cite{HMP6},
$\col(T_\lambda)$ is the standard reading word of the unique involution pipe dream for $z_\lambda$
described in \cite[Prop. 4.15]{HMP6},
while $\row(T_\lambda)$ is an alternate reading word in the sense of \cite[Def. 3.4]{HMP6}.}
For example, if $\lambda = (4,2,1)$ then 
\[
\ytab{
\none & \none & 5
\\
\none & 3 & 4 
\\
1 & 2 &3 &4 
}
\quand z_\lambda = (s_4s_3s_4 \widehat{s_5} s_2 \widehat{s_3}\widehat{s_1})(s_1 s_3 s_2 s_5 s_4s_3 s_4)= (1,5)(2,4)(3,6)\]
 where $\widehat{s_i}$ indicates the omission of that factor.
 
We need one more definition. Given any $z \in \I_\ZZ$, let $c_i$ be the number of positions in row $i$ of $\hat\D(z)$.
Then the \defn{involution shape} of $z$ \cite[Def. 4.38]{HMP4} is the transpose of the partition that sorts the sequence $(\dots,c_1,c_2,c_3,\dots)$.
 When $z=z_{(4,2,1)}$ the nonzero values of $c_i$ are $(c_1,c_2,c_3,c_4) = (1,2,3,1)$ so the involution shape is $(3,2,1,1)^\top = (4,2,1)$.
 This coincidence is a general phenomenon.
 
\begin{lemma}\label{zdom-lem}
Suppose $\lambda$ is a strict partition. Then the following properties hold:
\ben
\item[(a)] The involution shape of $z_\lambda$ is $\lambda$.
\item[(b)] It holds that $Q_{z_\lambda} = Q_{\lambda} $.
\item[(c)] We have $\PO(a) =T_\lambda$ for all $a \in \iR^+(z_\lambda)$.
\een
\end{lemma}

\begin{proof}
For part (a), observe that 
since $\hat\D(z_\lambda)$ is the transpose of $\SD_\lambda$,
the relevant value of $c_i$ is just height of column $i$ of $\SD_\lambda$. 
We claim that these numbers are a permutation of the heights of the columns of the unshifted diagram $\D_\lambda$.
As the latters heights are the parts of $\lambda^\top$, our claim implies that $(\dots,c_1,c_2,c_3,\dots)$ sorts to $\lambda^\top$
so the involution shape of $z_\lambda$ is $\lambda$ as desired.

To justify our claim, note that $\SD_\lambda$ can be formed by rearranging the columns of $\D_\lambda$ in the following way.
Since $\lambda$ is strict, $\D_\lambda$ has a column of height $k$ for each $k = 1,2,\dots,\ell(\lambda)$. Remove these columns from $\D_\lambda$
and then place them in ascending order on the left side of what remains. The result is $\SD_\lambda$. 

One can compute that $\PO(\col(T_\lambda)) = T_\lambda$ directly from the definition
 of $\PO$.
Given this observation and Corollary~\ref{q-cor1}, to prove parts (b) and (c) it suffices to show that  $Q_{z_\lambda} = Q_{\lambda} $.
We do this by appealing to results in \cite{HMP4}.
The permutation $z_\lambda$ is $132$-avoiding by  \cite[Ex. 2.2.2]{Manivel} and so 
also 2143-avoiding (i.e., \defn{vexillary}). 
 By \cite[Thm.~4.67]{HMP4}, the symmetric function  $Q_{y}$ is equal to a single Schur $Q$-function 
whenever $y \in \I_\ZZ$ is vexillary, and so in particular when $y=z_\lambda$.

Finally, \cite[Cor~4.42]{HMP4} identifies the top term 
in the Schur $Q$-expansion of $Q_y$  for any involution $y$: this is precisely
the Schur $Q$-function indexed by the involution shape of $y$.
Since this term is the only term when $y=z_\lambda$, we conclude from part (a) that 
$Q_{z_\lambda} = Q_{\lambda} $ as needed.
\end{proof}

As in the introduction, given elements $v \in S_m$ and $w \in S_n$, let $v \times w \in S_{m+n}$ be the permutation
mapping $i\mapsto v(i)$ for $i \in [m]$ and $m+j \mapsto m+ w(j)$ for $j \in [n]$.

\begin{corollary}\label{lrq-cor}
If $\lambda$ and $\mu$ are strict partitions then $Q_{\lambda} Q_\mu = \sum_\nu g^{\nu}_{\lambda\mu} Q_\nu$
where the sum is over strict partitions $\nu$
and $g^{\nu}_{\lambda\mu} $ is the number of elements in $\left\{ \PO(a) : a \in \iR^+(z_\lambda \times z_\mu)\right\}$ of shape $\nu$.
\end{corollary}

\begin{proof}
Let $y \in I_\ZZ \cap S_m$ and $z \in I_\ZZ \cap S_n$. It follows from Proposition~\ref{imat-prop} that $\Incr_\infty(\iR^+(y \times z))$
is in bijection with the product $\Incr_\infty(\iR^+(y )) \times \Incr_\infty(\iR^+(z ))$
via the map $((a^1,a^2,\dots),(b^1,b^2,\dots)) \mapsto ( a^1c^1, a^2c^2,\dots)$
where $c^i$ is formed by adding $m$ to each letter of $b^i$.
This implies that $Q_y Q_z = Q_{y \times z}$, and so the result follows from Corollary~\ref{q-cor1}.
\end{proof}

\subsection{Orthogonal Coxeter-Knuth equivalence}\label{ock-sect}

An essential property of orthogonal Edelman-Greene insertion is that the fibers of $ \PO$
are equivalence classes for a simple relation on primed words, which we define in this section.
Let $\ck$ denote the operator that acts on 1- and 2-letter primed words 
by interchanging  
\be\label{ck-eq1}
X \leftrightarrow X',
\quad
XY \leftrightarrow YX,\quad 
XY' \leftrightarrow YX',\quad 
X'Y \leftrightarrow Y'X,\quand 
X'Y' \leftrightarrow Y'X'
\ee
for all $X,Y \in \ZZ$. In addition, let $\ck$ act on 3-letter primed words as the involution interchanging
\be\label{ck-eq2}
XYX \leftrightarrow YXY,\quad 
X'YX \leftrightarrow YXY',\quad
ACB \leftrightarrow CAB,\quand 
BCA \leftrightarrow BAC
\ee
for all $X,Y \in \ZZ$ with $|X-Y|=1$ and all $A,B,C \in \ZZ\sqcup \ZZ'$ with $\lceil A \rceil < \lceil B \rceil  <\lceil C \rceil $,
while fixing 
any 3-letter words not of these forms. 
Given a primed word $a=a_1a_2a_3\cdots a_n$ and $i \in[n-2]$, we define
\[ 
\ba
\ck_{-1}(a) &:=\ck(a_1) a_2a_3\cdots a_n, \\
\ck_0(a) &:= \ck(a_1a_2)a_3\cdots a_n, \\
\ck_i(a) &:= a_1\cdots a_{i-1} \ck(a_ia_{i+1}a_{i+2}) a_{i+3}\cdots a_n,
\ea
\]
while setting $\ck_i(a):=a$ for $i \in \ZZ$ with $i+2 \notin[\ell(a)]$.
For example, if $a = 45'7121'$ then
\[
\barr{rrr}
 \ck_{-1}(a) = 4'5'7121',\quad
&
\ck_{0}(a) = 54'7121',\quad
&
\ck_{1}(a) = 45'7121',
\\
\ck_{2}(a) = 45'1721',\quad
&
\ck_{3}(a) = 45'1721',\quad
&
\ck_{4}(a) = 45'72'12.
\earr \]
The abbreviation ``$\ck$'' is for \defn{orthogonal Coxeter-Knuth operator}.

\begin{lemma}\label{unpri-word-lem}
If $i \geq 0$ and $a $ is a primed involution word
then $\unprime(\ck_i(a)) = \ck_i(\unprime(a))$.
\end{lemma}

\begin{proof}
This is clear  
unless $i \in [\ell(a)-2]$ and $\lceil a_i\rceil  = \lceil a_{i+2} \rceil$, but if this happens then 
Proposition~\ref{despite-prop} tells us that $a_{i+1} \in \ZZ$ and at most one of $a_i$ or $a_{i+2}$ is primed, so the result still holds.
\end{proof}

The transitive closure of the relation on unprimed words with $a \sim \ck_i(a)$ for all $i >0$ 
is often called \defn{Coxeter-Knuth equivalence} \cite[Def. 6.19]{EG}.
We define \defn{orthogonal Coxeter-Knuth equivalence} $\simICK$ to be 
the
 transitive closure of the relation on primed words with
 $a \simICK \ck_i(a)$ for all $i \in \ZZ$.
 
 \begin{lemma}
If $a \in \iR^+(z)$ for some $z \in I_\ZZ$ and $a \simICK b$, then $b \in \iR^+(z)$.
\end{lemma}

\begin{proof}
The first two relations in \eqref{ck-eq1} applied to the beginning of $a$
are special cases of $\isim$, while the last two relations in \eqref{ck-eq1} are compositions of the first three.
 The  word $a\in \iR^+(z)$ can only begin as $a=XY'\cdots$ for $X,Y \in \ZZ$
 if $|X-Y| > 1$, in which case applying the third relation in \eqref{ck-eq1} 
 corresponds to the $\isim$-equivalence $a = XY'\cdots \isim X'Y' \cdots \isim Y'X'\cdots  \isim YX'\cdots = \ck_0(a)$.
 The relations in \eqref{ck-eq2} are all special cases of $\isim$,
 so  
$\ck_i(a) \in \iR^+(z)$ for all $i$ by Proposition~\ref{imat-prop}.
\end{proof}


With the following result, we begin to see the close relationship between $\simICK$ and the map $\PO$.

\begin{proposition}\label{o-lem2}
Let $T$ be an increasing shifted tableau. Fix $z \in I_\ZZ$ and suppose $u \in \ZZ\sqcup \ZZ'$ has $\row(T)u \in \iR^+(z)$.
Then $\row(T)u \simICK \row(T\iarrow u)$.
Hence if $a\in \iR^+(z)$  then 
     $a\simICK \row(\PO(a))$.
 \end{proposition}

\begin{proof}
Let $T=T_0,T_1,T_2,\dots, T_N = T\iarrow u$ be the shifted tableaux formed by successive iterations of the algorithm in
Definition~\ref{iarrow-def},
and let $u=u_0,u_1,u_2,\dots,u_{N-1}$ be the numbers such that $u_{i-1}$ is inserted into row or column $i$ of $T_{i-1}$ on iteration $i$.
Let $R_i^{(j)}$ be the word formed by reading the $j$th row of $T_i$ from left to right and let $C_i^{(j)}$ be the word formed by reading the $j$th column of $T_i$
from top to bottom.
Finally, let $\tilde T_N := T_N = T\iarrow u$ and construct $\tilde T_i$ from $T_i$ for $i<N$ by adding $u_{i}$ to the end of row (respectively, column) $i+1$  if the insertion on iteration $i+1$ is into a row
(respectively, column).  Figure~\ref{running-fig1} shows two examples of these definitions.

\begin{figure}[h]
\ben
\item[(a)] If $T= \ytab{ \none & 5 & 6 \\ 1 & 3' & 4}$ and $u=2$ then $p=2<N=3$ and $u_0=2 <  u_1=3' < u_2=5'$ along with
\[ 
 \tilde T_0=  \ytab{ \none & 5 & 6 \\ 1 & 3' & 4&2  },
\quad
 \tilde T_1=   \ytab{ \none & 5 & 6 & 3' \\ 1 & 2 & 4   },
\quad
 \tilde T_2=   \ytab{
\none &\none & 5' \\
 \none & 3 & 6   \\ 1 & 2 & 4    },
 \quad
  \tilde T_3=
  \ytab{
 \none & 3 & 5'   \\ 1 & 2 & 4  & 6
}.
\]

\item[(b)] If $T= \ytab{ \none & \none & 7 \\  \none & 5 & 6 \\ 1 & 3' & 5}$ and $u=4$ then $p=2<N=3$ and $u_0=4 <  u_1=5 < u_2=6$ along with
\[ 
 \tilde T_0=  \ytab{ \none & \none & 7 \\  \none & 5 & 6 \\ 1 & 3' & 5 & 4},
\quad
 \tilde T_1=  \ytab{ \none & \none & 7 \\  \none & 5 & 6 & 5 \\ 1 & 3' & 4},
\quad
 \tilde T_2=  \ytab{ \none & \none & 6 \\  \none & \none & 7 \\  \none & 5 & 6 \\ 1 & 3' & 4},
 \quad
  \tilde T_3= \ytab{ \none & \none & 7 \\  \none & 5 & 6 \\ 1 & 3' & 4 & 7}.
  \]
\een
\caption{Examples for the proof of Proposition~\ref{o-lem2}.}\label{running-fig1}
\end{figure}

Suppose there are exactly $p\in[N]$ iterations involving row insertion.
We will show that if there are no iterations involving column insertion (so that $p=N$)
then 
\be\label{will-be-shown-eq0}
 \row(T)u= \row(\tilde T_0)
\simICK \row(\tilde T_1)
\simICK  \dots  \simICK  \row(\tilde T_N),
\ee
and if there is at least one iteration involving column insertion
then
\be\label{will-be-shown-eq}
  \row(T)u=\row(\tilde T_0)
\simICK \row(\tilde T_1)
\simICK  \dots \simICK \row(\tilde T_{p-1}) \simICK
\col(\tilde T_{p}) \simICK \col(\tilde T_{p+1}) \simICK \dots \simICK  \col(\tilde T_N).
\ee
The first case is precisely the desired identity as $ T\iarrow u = \tilde T_N$.
In the second case, it follows that $ \iR^+(z)$ contains   $\col(T\iarrow u)$, so by Lemma~\ref{simK-lem}
we  have $ \row(T)u\simICK \col(T\iarrow u) \simICK \row(T\iarrow u)$ as desired.

We argue by induction on $i$.
Assume the first $i-1$ equivalences hold in \eqref{will-be-shown-eq0} or \eqref{will-be-shown-eq}. 
Then $ \iR^+(z)$ contains the relevant reading word $\row(\tilde T_{i-1})$ or $\col(\tilde T_{i-1})$,
so the assertions in Remark~\ref{iarrow-rmk}(d) hold up to iteration $i$.
From these and the other properties in Remark~\ref{iarrow-rmk}, we see that 
if iteration $i$ involves row (respectively, column) insertion and the next iteration does not change the insertion direction,
then  $R_{i-1}^{(i)} u_{i-1} \simICK u_{i}R_{i}^{(i)}$
 (respectively, $u_{i-1}C_{i-1}^{(i)}  \simICK C_{i}^{(i)} u_{i}$).
In example (a) in Figure~\ref{running-fig1}, 
 \[ 
 R_0^{(1)} u_0 = 13'42 \simICK 13'24 \simICK 3'124 = u_1 R_1^{(1)}
 \quand
 u_2 C_2^{(3)} = 5'64 \simICK 5'46 = C_3^{(3)}u_2.
  \]
It follows that if $i<p-1$
then $ \row(\tilde T_{i-1})
\simICK \row(\tilde T_{i})$
and if $i\geq p$ then $ \col(\tilde T_{i-1})
\simICK \col(\tilde T_{i})$.

Suppose $p<N$ so that the insertion direction changes from rows to columns after iteration $p$.
It remains to show that  
 $\row(\tilde T_{p-1}) \simICK \col(\tilde T_{p})$.
 In this situation it must hold that $\lceil u_{p-1}\rceil \leq \min(R_{p-1}^{(p)}) = T_{pp}$, so
 there are two cases to consider according to whether $\lceil u_{p-1}\rceil < T_{pp} $ or $\lceil u_{p-1}\rceil = T_{pp} $. 
 
 First assume that $\lceil u_{p-1}\rceil < T_{pp} $.
 To show that $\row(\tilde T_{p-1}) \simICK \col(\tilde T_{p})$, we describe two enlarged ``tableaux'' 
with the same row and column reading words as $\tilde T_{p-1}$ and $\tilde T_{p}$, respectively,
that have certain diagonal reading words that are easily related.
 Let $D=\{ (2i,2j) : (i,j) \in T \} $.
Then define $V : D \sqcup\{(2p-1,2p-1)\} \to \ZZ\sqcup\ZZ'$
to be
the map with $V_{2i,2j}= (\tilde T_{p-1})_{ij}$ and $V_{2p-1,2p-1} = u_{p-1}$,
and
 define
  $W : D \sqcup\{(2p+1,2p+1)\} \to \ZZ\sqcup\ZZ'$
to be the map with
 $W_{2i,2j}= (\tilde T_{p})_{ij}$ and $W_{2p+1,2p+1} = u_{p}$.
 For example (a) in Figure~\ref{running-fig1}, we have  $p=2$, $u_{p-1}=3'$,  and $u_{p} = 5'$, along with
 \[
  V =  \smalltab{
     \none[\cdot] &  \none[\cdot] &  \none[\cdot] & \none[\cdot]&  \none[\cdot] & \none[\cdot] \\
  \none[\cdot] &  \none[\cdot] &  \none[\cdot] &5&  \none[\cdot] & 6 \\
  \none[\cdot] &  \none[\cdot] & 3'&  \none[\cdot] &  \none[\cdot] &  \none[\cdot] \\
  \none[\cdot] & 1 &  \none[\cdot] & 2 &  \none[\cdot] &4 \\
\none[\cdot]&  \none[\cdot] &  \none[\cdot] &  \none[\cdot] &  \none[\cdot] &  \none[\cdot]
  }
 \qquand
 W =  \smalltab{
   \none[\cdot] &  \none[\cdot] &  \none[\cdot] & \none[\cdot]& 5' & \none[\cdot] \\
  \none[\cdot] &  \none[\cdot] &  \none[\cdot] &3&  \none[\cdot] & 6 \\
  \none[\cdot] &  \none[\cdot] & \none[\cdot]&  \none[\cdot] &  \none[\cdot] &  \none[\cdot] \\
  \none[\cdot] & 1 &  \none[\cdot] & 2 &  \none[\cdot] &4 \\
\none[\cdot]&  \none[\cdot] &  \none[\cdot] &  \none[\cdot] &  \none[\cdot] &  \none[\cdot]
  }.
 \]
 Since $\row (\tilde T_{p-1}) = \row(V)$
 and
 $\col(\tilde T_{p}) = \col(W)$,
 it suffices to show that
  $\row(V) \simICK \col(W)$.
 Form the \defn{northeast} (respectively, \defn{southwest}) \defn{diagonal reading words} of $V$ (and similarly for $W$)
by reading the main diagonals of $V$ from left to right, going in the northeast (respectively, southwest) direction.
In our example, these words for $V$ are $13'5\hs 26\hs 4$ and $53'1\hs 62\hs 4$, respectively.
Finally define $\tilde V$ and $\tilde W$ by removing the main diagonals from $V$ and $W$.
Observe that  $\tilde V = \tilde W$.

Recall the definition of $\simK$ from Lemma~\ref{simK-lem}; this is a subrelation of $\simICK$.
First, we claim that $\row(V)$ is equivalent under $\simK$ to the southwest diagonal reading word of $V$.
To see this, start with $\row(V)$ and consider the diagonals of $V$ from left to right.
If $a_1a_2\cdots a_q$ is the first diagonal in increasing order, then we can use $\simK$ to commute $a_1$ backwards in $\row(V)$
until it is just after $a_2$, and then we can use $\simK$ to commute first $a_2$ and then $a_1$ backwards
until they after both just after $a_3$, and so on, until we are left with $a_q\cdots a_2a_1$ followed by 
the row reading word of $V$ with its first diagonal omitted. We then proceed in the same way over the remaining diagonals,
eventually reaching the southwest diagonal reading word of $V$ via $\simK$-equivalences.

It follows similarly that $\col(W)$ is equivalent under $\simK$ to the northeast diagonal reading word of $W$. 
One can repeat the argument in the previous paragraph, after replacing $\row$ by $\col$ and 
redefining $a_1a_2\cdots a_q$ to be the first diagonal in decreasing order.

The arguments above also show that the southwest (respectively, northeast) diagonal reading word of $\tilde V = \tilde W$
is equivalent under $\simK$ to its row (respectively, column) reading word.
But the row and column reading words of $\tilde V = \tilde W$ are equivalent under $\simK$
by Lemma~\ref{simK-lem}, since this tableau is increasing when all primes are removed from its entries
 by Remark~\ref{iarrow-rmk}(a).
Thus all four reading words for $\tilde V = \tilde W$ are equivalent under $\simK$.

The diagonal reading words of $V$ and $W$ are given by adding the first diagonal (in one of two orders) to the start of the corresponding diagonal reading words
of $\tilde V=\tilde W$.
Thus, to show that $\row(V) \simICK \col(W)$, 
we are reduced to checking the simpler property that the main diagonal of $V$ in the southwest reading order
  is equivalent under $\simICK$ to   the main diagonal of $W$ in the northeast reading order.
This is straightforward since both words have at most one primed letter; for example, 
$53'1 \simICK 35'1 \simICK 315' \simICK 135'$. It is only in this last step that we need to use the relation $\simICK$ instead of only $\simK$.
We conclude that $\row(\tilde T_{p-1}) \simICK \col(\tilde T_{p})$
 when $\lceil u_{p-1}\rceil < T_{pp} $.

We are left to consider the case when $\lceil u_{p-1}\rceil = T_{pp} $.
By Remark~\ref{iarrow-rmk} this can only occur when $ T_{pp} = T_{p,p+1} - 1 = T_{p+1,p+1}-2 \in \ZZ$.
Let $i$ be the index of $v := T_{pp}$ in $\row(\tilde T_{p-1})$. This index must be a commutation since all letters preceding $v$ are at least $ T_{p+1,p+1}= v+2$,
so truncating $\row(\tilde T_{p-1})$ just before $i$ gives a primed involution word for an element of $I_\ZZ$ that fixes $v$ and $v+1$.
By moving $u_{p-1}$ across row $p$ of $\tilde T_{p-1}$, we see that $\row(\tilde T_{p-1})$ is equivalent under $\simK$ to a word with letters $v(v+1)u_{p-1}$
in positions $i$, $i+1$, and $i+2$.
As $\lceil u_{p-1}\rceil = v$ and the index of $v$ is a commutation,  Proposition~\ref{despite-prop} implies that $u_{p-1} =v$ is unprimed.

Finally define $V : D \sqcup\{(2p+1,2p+1)\} \to \ZZ\sqcup\ZZ'$
to have $V_{2i,2j} = (T_{p-1})_{ij} = (T_{p})_{ij}$
 and $V_{2p+1,2p+1} =  u_{p-1}+1= u_{p}$.
 For example (b) in Figure~\ref{running-fig1}, this gives
 \[
  V =  \smalltab{
     \none[\cdot] &  \none[\cdot] &  \none[\cdot] & \none[\cdot]&  \none[\cdot] & \none[\cdot] \\
     \none[\cdot] &  \none[\cdot] &  \none[\cdot] & \none[\cdot]&  \none[\cdot] & 7\\
     \none[\cdot] &  \none[\cdot] &  \none[\cdot] & \none[\cdot]& 6& \none[\cdot] \\
  \none[\cdot] &  \none[\cdot] &  \none[\cdot] &5&  \none[\cdot] & 6 \\
  \none[\cdot] &  \none[\cdot] & \none[\cdot]&  \none[\cdot] &  \none[\cdot] &  \none[\cdot] \\
  \none[\cdot] & 1 &  \none[\cdot] & 3' &  \none[\cdot] &4 \\
\none[\cdot]&  \none[\cdot] &  \none[\cdot] &  \none[\cdot] &  \none[\cdot] &  \none[\cdot]
  }
  \]
 Then $\row(\tilde T_{p-1}) \simICK \row(V)$ and $\col(\tilde T_{p}) = \col(V)$,
 so it suffices to show that $\row(V) \simICK \col(V)$.
This follows by repeating the argument in the case when $\lceil u_{p-1}\rceil < T_{pp} $ but with $W:=V$
(and with $\tilde V=\tilde W$ again formed from $V$ by omitting the main diagonal).
That is, we first show that the row and southwest diagonal reading words of $V$ are equivalent under $\simK$,
as are the column and northeast diagonal reading words. Then we observe that the row, column, and both diagonal reading words of 
$\tilde V = \tilde W$ also equivalent under $\simK$. This reduces things to checking  
that reading the main diagonal of $V$ in increasing or decreasing order 
gives equivalent words under $\simICK$. This is straightforward as the main diagonal of $V$ has no primed entries.
\end{proof}

 \subsection{Dual equivalence operators for shifted tableaux}\label{dual-sect}

 Proposition~\ref{o-lem2} implies that if $a$ and $b$ are primed involution words with $\PO(a) = \PO(b)$ then $a\simICK b$. 
We will eventually prove the converse statement, that if $a\simICK b$ then $\PO(a)= \PO(b)$. 
The proof of this fact is more difficult, and requires us to understand 
 how the operators $\ck_i$ 
interact with $\PO$ and $\QO$. The results in this section precisely explain this interaction.

Assume $T$ is a standard shifted tableau. Choose $q>0$ such that the domain of $T$ fits inside $[q]\times [q]$.
Let $C_i$ be the increasing sequence of primed entries in column $i$ of $T$,
and let $R_i$ be the increasing sequence of unprimed entries in row $i$ of $T$.
The \defn{shifted reading word} of $T$ 
is \be \shword(T) := \unprime(C_qR_q\cdots C_2R_2C_1R_1).\ee
For example, if $T$ is the standard shifted tableau
\be\label{T-ex} 
T= \ytab{ \none & 3 & 5' & 7 \\ 1' & 2' & 4' & 6}\ee
then 
the nonempty sequences $C_iR_i$ are 
$C_1R_1 = 1'6$, $C_2R_2 = 2'37$, $C_3R_3=4'5'$, so the shifted reading word is
$\shword(T) = 4523716$.

A useful feature of this way of defining the shifted reading word is that it
automatically holds that $\shword(T) = \shword(\unprime_{\diag}(T))$, where as above $\unprime_{\diag}$ is the operation
removing all primes from the main diagonal. 
As noted in the proof of Proposition~\ref{udiag-lem}, we have $i \in \Des(T)$ if and only if $i+1$ appears before $i$ in  $\shword(T)$ when reading from left to right.

Let $n$ be the number of boxes in $T$. For each $i \in [n]$, 
write $\square_i$ for the unique position 
of $T$ containing $i$ or $i'$. Then define $\fks_i ( T)$ 
to be the shifted tableau formed from $T$ as follows:
\begin{itemize}

\item[(a)] If $\square_i$ and $\square_{i+1}$ are not in the same row or  column, then swap $i$ with $i+1$ and  $i'$ with $i+1'$.

\item[(b)] If $\square_i$ and $\square_{i+1}$ are in the same row or column and neither box is on the main diagonal,
then reverse the primes on the entries in both boxes.

\item[(c)] If $\square_i$ and $\square_{i+1}$ are in the same row or column but one box is on the main diagonal,
then reverse the prime on the entry in the non-diagonal box;
then, if both $\square_{i-1}$ and $\square_{i+1}$ are on the main diagonal when $i-1 \in [n]$
(respectively, if both $\square_i$ and $\square_{i+2}$ are on the main diagonal when $i+2 \in [n]$),
 switch the primes on the entries in these diagonal boxes.
 
\end{itemize}
Case (c) of this definition is illustrated by 
\[
\fks_1 \( \ytab{ \none & 3 & 5' & 7 \\ 1' & 2' & 4' & 6}\) 
=  \ytab{ \none & 3' & 5' & 7 \\ 1 & 2 & 4' & 6}
 \quand
 \fks_2 \( \ytab{ \none & 3' & 5' & 7 \\ 1' & 2' & 4' & 6}\) 
 = \ytab{ \none & 3' & 5' & 7 \\ 1' & 2 & 4' & 6}.\]
 Cases (a) and (b) are respectively illustrated by
 \[
\quad 
\fks_3 \( \ytab{ \none & 3 & 5' & 7 \\ 1' & 2' & 4' & 6}\)= 
\ytab{ \none & 4 & 5' & 7 \\ 1' & 2' & 3' & 6} 
\quand 
\fks_4 \( \ytab{ \none & 3 & 5' & 7 \\ 1' & 2' & 4' & 6}\)= 
\ytab{ \none & 3 & 5 & 7 \\ 1' & 2' & 4 & 6} .\]

Next, for each $i \in \ZZ$,  we construct a shifted tableau $\fkd_i(T)$ of the same shape from $T$ as follows.
If $i+2 \notin [n]$ then we set $\fkd_i(T) :=T$.
We form
$ \fkd_{-1}(T) $ (respectively, $ \fkd_0(T) $) from $T$ by reversing the prime on the entry in the first (respectively, second) box in the first
row, which is always the unique position containing $1$ or $1'$ (respectively, $2$ or $2'$).
For example
\[
\fkd_{-1}\( \ytab{ \none & 3 & 5' & 7 \\ 1' & 2' & 4' & 6}\)
=\ytab{ \none & 3 & 5' & 7 \\ 1 & 2' & 4' & 6}
\quand
\fkd_{0}\( \ytab{ \none & 3 & 5' & 7 \\ 1' & 2' & 4' & 6}\)
= \ytab{ \none & 3 & 5' & 7 \\ 1' & 2 & 4' & 6}.
\]
 Finally, if $i \in [n-2]$ then we set
\[\fkd_i(T) :=  \begin{cases}
\fks_{i} ( T) &\text{if $i+2$ is between $i$ and $i+1$ in $\shword(T)$}
\\
\fks_{i+1} ( T) &\text{if $i$ is between $i+1$ and $i+2$ in $\shword(T)$}
\\T &\text{if $i+1$ is between $i$ and $i+2$ in $\shword(T)$.}
\end{cases}\]
We refer to $\fkd_i$ as a \defn{dual equivalence operator} on standard shifted tableaux.

\begin{remark}\label{3spot-rmk}
If $\square_{i-1}$ and $\square_{i+1}$ are on the main diagonal,
then these boxes must be $(q-1,q-1)$ and $(q,q)$ for some $q$
and $\square_i = (q-1,q)$, in which case $\square_{i+2}$ cannot occur in row $q$, so $i+2$ is not between $i$ and $i+1$ in $\shword(T)$.
Similarly, if $\square_{i+1}$ and $\square_{i+3}$ are   on the main diagonal,
then these boxes must be $(q,q)$ and $(q+1,q+1)$ for some $q$
and $\square_{i+2} = (q,q+1)$, in which case
$\square_i$ cannot occur in column $q$, so
 $i$ is not between $i+1$ and $i+2$ in $\shword(T)$.
 Comparing these facts with the definition of $\fks_i$, we see that
 $\fkd_i(T)$ can only differ from $T$ in positions $\square_i$, $\square_{i+1}$, and $\square_{i+2}$.
\end{remark}

For the tableau $T$ in \eqref{T-ex}, our definition of $\fkd_i$ gives
\[\ba
\fkd_{1}\( \ytab{ \none & 3 & 5' & 7 \\ 1' & 2' & 4' & 6}\)
&=  \ytab{ \none & 3' & 5' & 7 \\ 1 & 2 & 4' & 6} = \fks_1(T)  
,\\[-10pt]\\
 \fkd_{2}\( \ytab{ \none & 3 & 5' & 7 \\ 1' & 2' & 4' & 6}\)= \fkd_{3}\( \ytab{ \none & 3 & 5' & 7 \\ 1' & 2' & 4' & 6}\)
&= \ytab{ \none & 4 & 5' & 7 \\ 1' & 2' & 3' & 6}= \fks_3(T)  
,\\[-10pt]\\
 \fkd_{4}\( \ytab{ \none & 3 & 5' & 7 \\ 1' & 2' & 4' & 6}\) &=  \ytab{ \none & 3 & 5' & 7 \\ 1' & 2' & 4' & 6}=T
,\\[-10pt]\\
 \fkd_{5}\( \ytab{ \none & 3 & 5' & 7 \\ 1' & 2' & 4' & 6}\) &= \ytab{ \none & 3 & 6' & 7 \\ 1' & 2' & 4' & 5}= \fks_5(T).
 \ea
\]

Given a shifted tableau $T$,
let $\primes(T)$ be the total number of boxes 
 in $T$ with primed entries    and let $\primesdiag(T)$ 
  be the number of such boxes that are on the main diagonal.
 Since we always have $\shword(T) = \shword(\unprime_{\diag}(T))$, 
 it holds by definition that if $i\neq -1$ then
 \be\label{up-fkd-eq}
 \unprime_{\diag}(\fkd_i(T)) = \fkd_i(\unprime_{\diag}(T))
 \quand
 \primesdiag(T) =  \primesdiag(\fkd_i(T)).
 \ee 
 It is also obvious that $\fkd_{-1}$ and $\fkd_{0}$ are involutions.
We note a few other properties of $\fkd_i$:

\begin{proposition}\label{square-lem}
Suppose $T$ is a standard shifted tableau with $n$ boxes.
Let $\square_j$ for $j \in [n]$ denote the unique box 
of $T$ containing $j$ or $j'$.
Finally choose $i \in [n-1]$. Then:
\ben
\item[(a)] The operator $\fkd_i$ is an involution which only changes the values of $T$ in   $\square_i$, $\square_{i+1}$, and $\square_{i+2}$.

\item[(b)] If $\square_i$ and $\square_{i+2}$ are not both on the main diagonal, then $\primes(T)  = \primes(\fkd_i(T)) $ and
the main diagonal positions with primed entries in $\fkd_i(T)$ are  the same as those in $T$.

\item[(c)] If $\square_i$ and $\square_{i+2}$ are both on the main diagonal, then   $\primes(T)  = \primes(\fkd_i(T)) \pm 1$.

\een
\end{proposition}

\begin{proof}
Part (a) is clear if $i+1$ is between $i$ and $i+2$ in $\shword(T)$.
Suppose instead that  $i+2$ is between $i$ and $i+1$ in $\shword(T)$.
If $\square_i$ and $\square_{i+1}$ are not  in the same row or column, then 
$\shword(\fks_i(T))$ is formed from $\shword(T)$ by swapping the positions of $i$ and $i+1$,
so $i+2$ is also between $i$ and $i+1$ in $\shword(\fks_i(T))$ and we have $\fkd_i(\fkd_i(T))=\fks_i(\fks_i(T))=T$.
If $\square_i$ and $\square_{i+1}$ are in the same row or column but at least one of the boxes is on the main diagonal, then 
our assumption that $i+2$ is between $i$ and $i+1$ in $\shword(T)$ forces 
$\square_i$, $\square_{i+1}$, and $\square_{i+2}$ to be arranged in $T$ as 
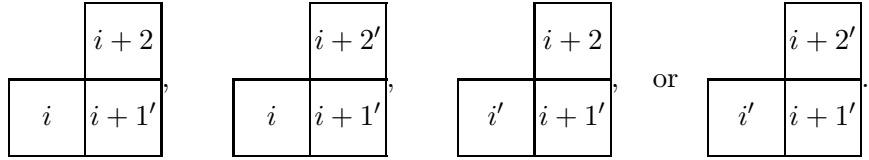
\[{ 
\ytableausetup{boxsize = 1cm,aligntableaux=center}
\begin{ytableau}
\none & i + 2 \\
i & i+1' 
 \end{ytableau},
 \qquad
 \begin{ytableau}
\none & i + 2' \\
i & i +1'
 \end{ytableau},
  \qquad
 \begin{ytableau}
\none & i + 2 \\
i' & i +1'
 \end{ytableau},
  \quord
 \begin{ytableau}
\none & i + 2' \\
i' & i +1'
 \end{ytableau}}.
 \]
In each of these cases we have
 $\fkd_i(\fkd_i(T))= \fks_{i+1}(\fks_i(T)) = T$.
 
Finally, suppose $\square_i$ and $\square_{i+1}$ are in the same row or column but neither box is on the main diagonal.
Then the entry in one box must be primed and the other must be unprimed for 
$i+2$ to be between $i$ and $i+1$ in $\shword(T)$.
If $\square_i$ and $\square_{i+1}$ are in the same column, then they must be some adjacent positions $(j,k)$ and $(j+1,k)$,
and $\fkd_i$ acts as $\fks_i$ by reversing the primes on both positions.
In this case, consider the sequence of unprimed boxes to the right of $\square_{i+1}$
in row $j+1$, followed by the primed boxes in column $j$, and then the unprimed boxes to the left of $\square_i$ in row $j$.
For example, if $\square_i$ and $\square_{i+1}$  are the boxes containing $\ast$ in 
\[
\ytab{
\none & \none & \none & \ & \ & \ast & 1 & 2 & 3
\\
\none & \none & 6 & 7 & 8 & \ast & \ & \ & \
\\
\none & \ & 5 & \ & \ & \ & \ & \ & \
\\
\ & \ & 4 & \ & \ & \ & \ & \ & \
}
\]
then the relevant sequence is a subsequence of the positions labeled $1,2,\dots,8$.
It is impossible for $\square_{i+2}$ to occur in this sequence, and if we ignore the entries it contributes to the shifted reading word
then $\shword(\fks_i(T))$ is obtained from $\shword(T)$ by swapping $i$ and $i+1$.

Thus
if $\square_i$ and $\square_{i+1}$ are in the same column,
then
 $i+2$ still appears between $i$ and $i+1$ in the shifted reading word of $\fkd_i(T)=\fks_i(T)$
 so $\fkd_i(\fkd_i(T))=\fks_i(\fks_i(T))=T$.
The same conclusion follows when $\square_i$ and $\square_{i+1}$ are the adjacent positions $(j,k)$ and $(j,k+1)$,
if we instead consider the sequence of primed boxes above $\square_{i+1}$ in column $k+1$,
followed by the unprimed boxes in row $k+1$, and then the primed boxes below $\square_i$ in column $k$.

 The argument to show that $\fkd_i(\fkd_i(T))=T$ when $i$ is between $i+1$ and $i+2$ in $\shword(T)$
 is similar. This concludes the proof of  part (a) by Remark~\ref{3spot-rmk}.

For part (b), suppose $\square_i$ and $\square_{i+2}$ are not both on the main diagonal.
Then at most one of the three boxes $\square_i$, $\square_{i+1}$, $\square_{i+2}$ that could change in $\fkd_i(T)$ compared to $T$ is on the main diagonal.
Since the operator $\fks_j$ changes the primes on either zero or two main diagonal boxes, it follows 
that 
the main diagonal positions with primed entries in $\fkd_i(T)$ are  the same as those in $T$

Additionally,  
if $i+2$ is between $i$ and $i+1$ in $\shword(T)$ and $\square_i$ and $\square_{i+1}$ are in the same row or column,
then neither box can be on the main diagonal and exactly one   must have a primed entry, so  $\primes(T)  = \primes(\fks_i(T)) $.
Likewise,
if $i$ is between $i+1$ and $i+2$ in $\shword(T)$ and $\square_{i+1}$ and $\square_{i+2}$ are  in the same row or column,
then neither box can be on the main diagonal and exactly one   must have a primed entry, so  $\primes(T)  = \primes(\fks_{i+1}(T)) $.
Therefore $\primes(T)  = \primes(\fkd_i(T)) $. This proves part (b).

Finally, for part (c), observe that if $\square_i$ and $\square_{i+2}$ are both on the main diagonal,
then we must have $\square_i = (q-1,q-1)$, $\square_{i+1} = (q-1,q)$, and $\square_{i+2} = (q,q)$ for some $q$.
No matter how the entries in these boxes are primed, we have $\fkd_i(T) = \fks_i(T) = \fks_{i+1}(T)$
so  $\primes(T)  = \primes(\fkd_i(T)) \pm 1$.
  \end{proof}

Our proof of the following theorem occupies all of Section~\ref{proofs-sect}.

\begin{theorem}\label{ck-fkd-thm}
Suppose $i \in \ZZ$ and $a$ is a primed involution word for an element of $I_\ZZ$. 
Then it holds that $\PO(\ck_i(a)) = \PO(a)$ and $  \QO(\ck_i(a)) = \fkd_i(\QO(a)) .$
%
\end{theorem}

When $a $ has no primed letters,
this theorem is equivalent to results in \cite{Marberg2019b}; see Proposition~\ref{unprime-prop}.
Extending these identities to primed involution words is surprisingly involved.
The proof of the unprimed version of Theorem~\ref{ck-fkd-thm} in \cite{Marberg2019b}
relies heavily on the \defn{involution Little map}, 
which gives a family of bijections $\bigsqcup_{z \in X} \iR(z) \leftrightarrow \bigsqcup_{z \in Y} \iR(z)$ for certain finite subsets  $X,Y\subset I_\ZZ$.
Describing a ``primed involution Little map'' does not appear to be straightforward; one difficulty is that
with primes allowed, the unions $\bigsqcup_{z \in X} \iR^+(z) $ and $ \bigsqcup_{z \in Y} \iR^+(z)$ often have different sizes.
As such, proving Theorem~\ref{ck-fkd-thm} requires a quite different strategy compared to \cite{Marberg2019b}.

\begin{corollary}\label{o-cor2}
Two primed involution words satisfy $a\simICK b$
      if and only if $\PO(a) =\PO(b)$.
 \end{corollary}
 
 \begin{proof} 
 Let $a$ and $b$ be two primed involution words.
 If $\PO(a) = \PO(b)$ then $a \simICK \row(\PO(a)) = \row(\PO(b)) \simICK b$ by Proposition~\ref{o-lem2}.
Conversely, if  $a \simICK b$
then $b = \ck_{i_1}\ck_{i_2}\cdots \ck_{i_k}(a)$ for some $i_1,i_2,\dots,i_k \in \ZZ$, so
$\PO(b) = \PO(\ck_{i_1}\ck_{i_2}\cdots \ck_{i_k}(a))  = \PO(a)$ by Theorem~\ref{ck-fkd-thm}.
\end{proof}

Recall the definition of the relation $\simK$ from Lemma~\ref{simK-lem}.

\begin{corollary}
Suppose $T$ is an increasing shifted tableau with $\row(T) \in \iR^+(z)$ for some $z \in I_\ZZ$.
Then $\row(T) \simK \col(T) \in \iR^+(z)$  and 
$\PO(\row(T)) = \PO(\col(T)) =\unprime_{\diag}(T)$.
\end{corollary}

\begin{proof}
We have $\row(T) \simK \col(T) \in \iR^+(z)$ by Lemma~\ref{simK-lem} so $\PO(\row(T)) = \PO(\col(T)) $.
When we compute $\PO(\col(T))$ using Definition~\ref{iarrow-def},
each column of $T$ contributes the same column to the output but with primes removed from any diagonal entries,
resulting in  
$\unprime_{\diag}(T)$. 
\end{proof}
 
 \subsection{Properties of marked cycles}\label{marked-sect}
 
 On standard shifted tableaux with no primes on the main diagonal, the operators $\fkd_i$ for $i>0$ coincide
with the maps $\psi_{i+1}$   in \cite[\S6]{Assaf14}.
The definitions of $\fkd_i$ and $\psi_{i+1}$ diverge when there are primed entries on the main diagonal,
as $\psi_{i+1}$ never changes the locations of these entries. However, 
 \cite[Thm. 6.3]{Assaf14} (stating that 
 $\{ \psi_{i}\}_{1<i<n}$ is a \defn{dual equivalence} for 
standard shifted tableaux)
is still true if one replaces $\psi_i$ by $\fkd_{i-1}$, as we explain in this section.
The results here will also be of use in Section~\ref{proofs-sect}.
 
 Let $\cyc(z) = \{ \{i,j\} : i<j=z(i)\} $ denote the set of 2-cycles in $z$.
Then for each (unprimed) involution word $a=a_1a_2\cdots a_n \in \iR(z)$ and  $i\in [n]$, let
\be\label{gammagamma-def}
\gamma_i(a) :=
\begin{cases}
 s_{a_n}\cdots s_{a_{i+2}} s_{a_{i+1}} (\{a_i,1+a_{i}\})
 &\text{if $i$ is a commutation in $a$} \\ 
 \emptyset &\text{otherwise}.
 \end{cases}
\ee
For example, if $z = 654321\in I_\ZZ$, then $\cyc(z) =\{\{1,6\},\{2,5\},\{3,4\}\}$
and for $a=513243541 \in \iR(z) $,
we have  $\gamma_1(a) = \{3,4\}$, $\gamma_2(a) = \{2,5\}$, $\gamma_3(a) = \{1,6\}$,
and $\gamma_i(a) = \emptyset$ for $i \in \{4,5,6,7,8,9\}$.

\begin{proposition}
The map
$i \mapsto \gamma_i(a)$
is a bijection from the set of commutations in $a$ to $\cyc(z)$.
\end{proposition}

\begin{proof}
We prove this by induction on the length $n$ of $a$.
The base case when $n=0$ holds trivially.
Assume $n>0$, define $b=a_1a_2\cdots a_{n-1}$,  and let $y \in I_\ZZ$ be such that $b \in \iR(y)$.
Suppose the result holds when $a$ and $z$ are replaced by $b$ and $y$.

If $n$ is a commutation in $a$ then 
$a_n$ and $1+a_n$ are fixed points of $y$, and the commutations in $a$
are just the commutations of $b$ plus $n$.
In this case we have  $z= ys_{a_n}$ and
 $\cyc(z) = \cyc(y) \sqcup \{ \{ a_n,1+a_n\}\}$,
along with
$\gamma_i(a) = s_{a_n}(\gamma_i(b)) = \gamma_i(b)$ for each commutation $i \in [n-1]$
(since $\gamma_i(b) \in \cyc(y)$ by induction)
and $\gamma_n(a) =\{ a_n,1+a_n\}$.
As $i\mapsto \gamma_i(b)$ is a bijection from commutations in $b$ to $\cyc(y)$,
it follows that 
$i \mapsto \gamma_i(a)$
is a bijection from commutations in $a$ to $\cyc(z)$.

If $n$ is not a commutation in $a$ then $z = s_{a_n} y s_{a_n}$
so $\cyc(z) = s_{a_n}(\cyc(y))$,
and the commutations in $a$ are the same as in $b$. 
As 
$\gamma_i(a) = s_{a_n}(\gamma_i(b))$ for $i \in [n-1]$,
the desired property clear.
\end{proof}

The following lemma lets us relate $\gamma_i(a)$ and $\gamma_i(b)$
when $a\iisim b$ in the sense of Proposition~\ref{mat3-prop}.

\begin{lemma}\label{marked-lem}
Suppose $a \in \iR(z)$ is an unprimed involution and $n=\ell(a)$. Fix $i\in[n]$.
\ben\item[(a)] If $j \in [n-1]$ and $|a_j - a_{j+1}| > 1$ then 
\[
\gamma_i(a_1\cdots a_{j-1} a_{j+1} a_j a_{j+2}\cdots a_n) = \begin{cases}
 \gamma_{j+1}(a) & \text{if $i=j$} \\
 \gamma_j(a) & \text{if $i=j+1$} \\
 \gamma_i(a) &\text{otherwise}.
 \end{cases}
 \]
\item[(b)] If $j \in [n-2]$ and $a_j = a_{j+2} = a_{j+1} \pm 1$ then 
\[
\gamma_i(a_1\cdots a_{j-1} a_{j+1} a_j a_{j+1} a_{j+3}\cdots a_n) = \begin{cases}
 \gamma_{j+2}(a) & \text{if $i=j$} \\
 \gamma_j(a) & \text{if $i=j+2$} \\
 \gamma_i(a) &\text{otherwise}.
 \end{cases}
 \]
\item[(c)] If $n\geq 2$ and $|a_1-a_2| = 1$ then 
$\gamma_i (a_2a_1a_3\cdots a_n) = \gamma_i(a)$
for all values of $i$.
\een
\end{lemma}

\begin{proof}
Suppose $j \in [n-1]$ and $|a_j - a_{j+1}| > 1$. 
Let $b= a_1\cdots a_{j-1} a_{j+1} a_j a_{j+2}\cdots a_n \iisim a$.
Since $s_{a_j}$ and $s_{a_{j+1}}$ commute, we have $\gamma_i(a) = \gamma_i(b)$ 
for $i \notin \{j,j+1\}$. In addition, the index $j$ (respectively, $j+1$) is a commutation in $a$
if and only if $j+1$ (respectively, $j$) is a commutation in $b$,
and the permutations
  $s_{a_j}$ and $s_{a_{j+1}}$ each preserve both of the sets $\{ a_{j} ,1+ a_{j} \}$ and $\{ a_{j+1} , 1+ a_{j+1} \}$. It follows from  \eqref{gammagamma-def} in this case that 
$\gamma_j(b) = \gamma_{j+1}(a)$ and $\gamma_{j+1}(b) = \gamma_{j}(a)$.

Next, suppose $j \in [n-2]$ and $a_j = a_{j+2} = a_{j+1} \pm 1$.
Let $b = a_1\cdots a_{j-1} a_{j+1} a_j a_{j+1} a_{j+3}\cdots a_n \iisim a$.
Then $i$ (respectively, $i+2$) is a commutation in $a$ if and only if $i+2$ (respectively, $i$) is a commutation on $\ck_i(a)$,
 while $i+1$ is not a commutation in either word,
 by Propositions~\ref{despite-prop} and \ref{imat-prop}.
The permutation $s_{a_{i+2}}s_{a_{i+1}}=s_{a_{i}}s_{a_{i+1}}$ transforms $\{  a_i  ,1+  a_i  \}$ to $\{  a_{i+1} ,1+  a_{i+1}  \}$
while $s_{a_{i+1}} s_{a_i}$  transforms $\{  a_{i+1}  ,1+  a_{i+1}  \}$ to  $\{  a_i  ,1+  a_i  \}$,
so it follows from  \eqref{gammagamma-def} that 
$\gamma_i(b) = \gamma_{i+2}(a)$ and $\gamma_{i+2}(b) = \gamma_{i}(a)$.

For part (c) we may assume that $n=2$, and then the desired result is clear from \eqref{gammagamma-def}.
\end{proof}

For a primed involution word $\hat a = \hat a_1\hat a_2\cdots \hat a_n \in \iR^+(z)$ with   $a = \unprime(\hat a)$, let  
\be\label{marked-def}\marked(\hat a) := \{ \gamma_i(a) : i \in[n]\text{ with } \hat a_i \in \ZZ'\}.\ee

\begin{proposition}\label{marked-prop}
Suppose $\hat a\in \iR^+(z)$ for $z \in I_\ZZ$ and $a = \unprime(\hat a)$. Let $i \in\ZZ$. 
\ben
\item[(a)] If $i=-1$ then $\marked(\ck_{i}(\hat a))=\marked(\hat a)\mathbin{\triangle}  \{ \gamma_1(a)\}$,
where  $\triangle$ is symmetric set difference.
\item[(b)] Suppose $i=0$ and $\hat a$ has at least two letters. If $| a_1  -  a_2  |>1$
 and exactly one of $\hat a_1$ or $\hat a_2$ is primed, then exactly one of $ \gamma_1(a)$ or $\gamma_2(a)$ belongs to $\marked(\hat a)$
and  it holds that $\marked(\ck_i(\hat a))=\marked(\hat a)\mathbin{\triangle}  \{ \gamma_1(a), \gamma_2(a)\} $.
\item[(c)] In all other cases $\marked(\ck_i(\hat a)) = \marked(\hat a)$.
\een
\end{proposition}


\begin{proof}
Parts (a) and (b) hold as   $\emptyset \neq \gamma_1(a) \in \cyc(z)$.
Part (c) follows directly from Lemma~\ref{marked-lem}.
 \end{proof}

Fix a strict partition $\lambda$ and define $z_\lambda$ as in Lemma~\ref{zdom-lem}.
For each   $S \subseteq \cyc(z_\lambda)$, let 
$\cA^\lambda_S$ be the set of standard shifted tableaux $\QO(a)$
for $a \in \iR^+(z_\lambda)$ with $\marked(a) = S$.
Proposition~\ref{unprime-tab-prop} implies that 
 $\cA^\lambda_\varnothing$ is set of  all standard shifted tableaux of shape $\lambda$
with no primed diagonal entries.

\begin{corollary}
Fix $S\subseteq \cyc(z_\lambda)$ and $1\leq i \leq |\lambda|-2$. Then $\fkd_i$ restricts to an involution of $\cA^\lambda_S$ and $\unprime_{\diag}$ defines a descent-preserving
bijection $\cA^\lambda_S \to \cA^\lambda_{\varnothing}$ that commutes with  $\fkd_i$.
\end{corollary}

\begin{proof}
We have $\fkd_i(\cA^\lambda_S) = \cA^\lambda_S$ by Proposition~\ref{marked-prop}.
The map $\unprime_{\diag}$ is a bijection since  $|\cA^\lambda_S| = |\cA^\lambda_\varnothing|= |\iR(z_\lambda)|$. It is descent-preserving by Proposition~\ref{udiag-lem} and 
 commutes with  $\fkd_i$  by \eqref{up-fkd-eq}.
 \end{proof}
 
 Assaf's result \cite[Thm. 6.3]{Assaf14} asserts that the maps $\{\fkd_{i-1} : 1<i<|\lambda|\}$ 
  give a \defn{dual equivalence} for $\cA^\lambda_\varnothing$.
The preceding corollary shows that these maps define isomorphic dual equivalences for each $\cA^\lambda_S$,
and therefore give a dual equivalence for all standard shifted tableaux of shape $\lambda$.
  
\section{Proofs of the two main theorems}\label{proofs-sect}

This section is devoted to proving Theorem~\ref{ck-fkd-thm}. We will also
end up deriving Theorem~\ref{o-thm1} as a corollary of our methods;
the proofs of these theorems are in Section~\ref{proof-sect4}.

\begin{remark*}
Many of the results leading up to these proofs only apply to unprimed words.
Accordingly, just for this section, we adopt the convention of  writing all primed words with $\hat{\ }$ symbols
(that is, as $\hat a$, $\hat b$, etc.) to distinguish them from unprimed words (which we write as $a$, $b$, etc.).
\end{remark*}

An outline of our proof strategy is as follows.
Underpinning everything is the following result,
which says that Theorem~\ref{ck-fkd-thm} holds for unprimed words.

\begin{proposition}[\cite{Marberg2019b}] \label{unprime-prop}
Suppose $i \geq 0$ and $a=\unprime(a)\in \iR(z)$ for $z \in I_\ZZ$.
Then \[\PO(\ck_i(a)) = \PO(a)\quand  \QO(\ck_i(a)) = \fkd_i(\QO(a)) .\]
\end{proposition}

\begin{proof}
The assertion that $\PO(\ck_i(a)) = \PO(a)$ follows from \cite[Thm.~3.31]{Marberg2019b}.
The assertion that $  \QO(\ck_i(a)) = \fkd_i(\QO(a)) $ follows from \cite[Thm.~5.11]{Marberg2019b}.
\end{proof}

Let $\hat a$ be a primed involution word with unprimed form $a=\unprime(\hat a)$.
In view of Proposition~\ref{unprime-prop}, to prove Theorem~\ref{ck-fkd-thm} we just need to understand  
the relationship between the indices of the primed letters in $\hat a$ and the locations of the primed entries in $\PO(\hat a)$ and on the main diagonal of $\QO(\hat a)$.
Sections~\ref{bumping-sect}, \ref{cyc-sect1}, and \ref{cyc-sect2}
are devoted to proving a result that expresses the positions of the relevant primes 
in terms of the set $\marked(\hat a)$ and a permutation $\tpi(a)$  that can be read off from the successive tableaux
$\PO(a_1a_2\cdots a_i)$ for $i \in [\ell(a)]$. Then, in Sections~\ref{proof-sect3}, \ref{inter-sect} and \ref{tech-sect},
we will prove a series of lemmas clarifying the relationship between $\tpi(a)$ and $\tpi(\ck_i(a))$.

\subsection{Properties of bumping paths}\label{bumping-sect}

We start by listing some properties of the bumping paths  
in Definition~\ref{iarrow-def}.
In this subsection, let $T$ be an increasing shifted tableau with no primes on the main diagonal and let $u \in \ZZ\sqcup\ZZ'$ be such that 
$\row(T)u$ is a primed involution word for an element of $I_\ZZ$.
We will only apply the results here when $T=\unprime(T)$ and $u \in \ZZ$, but we
will allow primes in our initial statements since the proofs are identical to the unprimed case.
Write 
\be
\label{denote-bp-eq}
\path^\leq(T,u):=\((x_i,y_i) : i=1,2,\dots,N\)\quand \path^<(T,u):=\((\tilde x_i, \tilde y_i): i=1,2,\dots,N\)
\ee
for the weak and strict bumping paths specified in Definition~\ref{iarrow-def}.

The algorithm in Definition~\ref{iarrow-def} starts by inserting entries into successive rows, 
and at some point may switch to inserting into successive columns.
Each iteration contributes one position to the weak and strict bumping paths,
and the switch from row to column insertion takes place at most once, 
directly after the weak bumping path meets the main diagonal.
It follows that
both $\path^\leq(T,u)$ and $\path^<(T,u)$  contain at most one position on the main diagonal.
Let $p$ be the unique index of the diagonal position in $\path^\leq(T,u)$  (which will have  $x_p = y_p=p$),
or set $p:=N$ if no such index exists.

The following additional observations are straightforward to derive from the definitions
and Remark~\ref{iarrow-rmk}. We omit a detailed proof.
For  $(x,y) \in \ZZ\times \ZZ$,
let 
\[\regionSW(x,y) := \{ (i,j) \in \ZZ\times \ZZ : x \geq i\text{ and }y \geq j\}
\quand
\regionNE(x,y) := \{ (i,j) \in \ZZ\times \ZZ : x \leq i\text{ and }y \leq j\}.\]
Define
$
\regionSW(T,u) := \bigcup_{1\leq i\leq p} \regionSW(\tilde x_i, \tilde y_i)
$
and
$
\regionNE(T,u) := \bigcup_{p < k \leq N} \regionNE( x_i,  y_i).
$ 

\begin{proposition}\label{bumping-prop}
 The following properties hold:
\ben
\item[(a)] If $1\leq i\leq p$ then $x_i=\tilde x_i=i$ and $\tilde y_i \in \{ y_i, y_i+1\}$, while
\[ y_1 \geq y_2 \geq \dots \geq y_p\quand \tilde y_1 \geq  \tilde y_2 \geq \dots \geq \tilde y_p.\]

\item[(b)] If $p< k\leq N$ then $y_k=\tilde y_k = k$ and $\tilde x_k \in \{ x_k, x_k+1\}$, while 
\[ p \geq x_{p+1} \geq  x_{p+2} \geq \dots \geq x_N\quand p+1 \geq \tilde x_{p+1} \geq \tilde x_{p+2} \geq \dots \geq \tilde x_N.\]

\item[(c)] If $(x_p,y_p) \neq (\tilde x_p, \tilde y_p)$, then $p<N$ and
 $\regionSW(T,u) \cap \regionNE(T,u) = \{(p,p+1)\}$ and
\[
\begin{aligned} 
(p,p) &= (x_p,y_p), \\
(p,p+1) & = (\tilde x_p, \tilde y_p) = (x_{p+1},y_{p+1}), \\
(p+1,p+1) &= (\tilde x_{p+1},\tilde y_{p+1}).
\end{aligned}
\]
If instead $(x_p,y_p)  = (\tilde x_p, \tilde y_p)$, then $\regionSW(T,u) \cap \regionNE(T,u) =\varnothing$.

\een
\end{proposition}


We sometimes treat the sequences $\path^\leq(T,u)$ and $\path^<(T,u)$ as sets. 
This practice is justified as Proposition~\ref{bumping-prop} shows that 
the positions in each path are all distinct and  their order   is uniquely determined.

With $p$ as above, write \be
\rwpath(T,u):=\((x_i,y_i): i=1,2,\dots,p\)\quand \rspath(T,u):=\((\tilde x_i, \tilde y_i): i=1,2,\dots,p\)
\ee
for the first $p$ terms of $\wpath(T,u)$ and $\spath(T,u)$, and let
\be\label{cpath-eq}
\ba \cwpath(T,u)& :=\((x_i,y_i) : i=p+1,p+2,\dots,N\),
\\
 \cspath(T,u)&:=\((\tilde x_i, \tilde y_i) : i=p+1,p+2,\dots,N\).
 \ea
\ee
We think of these subsequences as the ``row-bumping paths'' and ``column-bumping paths''
from inserting $u$ into $T$.

Finally, if $\hat a$ is a primed involution word with $n=\ell(\hat a)$
and $i\in[n]$, then we let 
\[
\wpath_i(\hat a) := \path^\leq(T,\hat a_i)
\quand
\spath_i(\hat a) := \path^<(T,\hat a_i)
\quad\text{for }T := \PO(\hat a_1\hat a_2\cdots \hat a_{i-1}).\]
We define the sequences $\rwpath_i(\hat a)$, $\cwpath_i(\hat a)$, $\rspath_i(\hat a)$, and $\cspath_i(\hat a)$ analogously.

\begin{proposition}\label{bumping-prop2}
Let $\hat a=\hat a_1\hat a_2\cdots \hat a_n$  be a primed involution word and choose $i \in [n-1]$.
\ben

\item[(a)] Suppose $\hat a_{i+1} < \hat a_i$.
 In each row where
 $\rwpath_i(\hat a)$ and $\rwpath_{i+1}(\hat a)$ both have positions, the position in
 $\rwpath_i(\hat a)$
 is weakly to the right of the position in $\rwpath_{i+1}(\hat a)$.
Consequently, if $\wpath_i(\hat a)$ has a diagonal position,
then   $\wpath_{i+1}(\hat a)$ has a non-terminal diagonal position.

\item[(b)] Suppose $\hat a_i<\hat a_{i+1}$. In each row where
 $\rwpath_i(\hat a)$ and $\rwpath_{i+1}(\hat a)$ both have positions, the position in
 $\rwpath_i(\hat a)$
 is strictly to the left of the position in $\rwpath_{i+1}(\hat a)$.
Consequently, if $\wpath_{i+1}(\hat a)$ has a diagonal position,
then   $\wpath_i(\hat a)$ has a non-terminal diagonal position.

\een
\end{proposition}

\begin{proof}
Both parts can be checked directly, using Remark~\ref{iarrow-rmk}
and Proposition~\ref{bumping-prop},
together with the general principle that in a given row, after inserting a number which bumps 
some box (and then possibly increasing entries to the right of this box as a result of subsequent
column insertions),
inserting a smaller number will always bump a box that is weakly farther to the left,
while inserting a larger number will always bump a box that is strictly farther to the right.
\end{proof}

\subsection{Controlling cycle migration}\label{cyc-sect1}

Fix $z \in I_\ZZ$ and recall the definition of $\gamma_i(a) \in \{\emptyset\}\sqcup\cyc(z)$ for   $a \in \iR(z)$ from \eqref{gammagamma-def}.
Suppose $T$ is a shifted tableau 
 and $b$ is a word 
such that $\row(T)b \in \iR(z)$.
For   $(i,j) \in \ZZ\times\ZZ$, 
let 
\be\label{garray-eq}\gamma_{ij}(T,b) := \begin{cases}
\emptyset &\text{if $(i,j)$ is not in the domain of $T$} \\
\gamma_k(\row(T)b) &\text{if $(i,j)$ is in the domain of $T$}, 
\end{cases}
\ee where 
 $k$ is the index of the letter in $\row(T)$ contributed by box $(i,j)$.
We also let $\gamma_{ij}(T) := \gamma_{ij}(T,\emptyset)$.

The main result of this section is a lemma that precisely describes 
how the values of \eqref{garray-eq} evolve when we insert the first letter of $b$ into $T$
via Definition~\ref{iarrow-def}. 
In Section~\ref{cyc-sect2},
we will use this lemma to explain how 
to compute 
 $\PO(\hat a)$ and $\QO(\hat a)$ 
from $\PO(a)$, $\QO(a)$, and the set $\marked(\hat a)$ when
$\hat a$ is primed involution word 
with $a=\unprime(\hat a)$.

\begin{example}
\label{Tb-ex-eq}
 If
 $T = \PO(51324) = \smalltab{ \none & 3 & 5 \\ 1 &2 & 4}
$
and $ b = 3154$
then 
we have
\[\ytableausetup{boxsize = 1.6cm,aligntableaux=center}
\begin{ytableau} \none & \gamma_{22}(T,b) & \gamma_{23}(T,b) \\ \gamma_{11}(T,b)  &\gamma_{12}(T,b)  & \gamma_{13}(T,b) 
\end{ytableau}
=
\begin{ytableau} \none &\{1,6\}& \{3,4\}\\ \{2,5\} & \emptyset & \emptyset
\end{ytableau}.\]
\end{example}

Below, we assume that the shifted tableau $T$ is increasing and the unprimed word $b$ is nonempty with first letter $u \in \ZZ$. 
Let $c$ be the subword of $b$
formed by removing its first letter. Denote  
the weak and strict bumping paths
resulting from inserting $u$ into $T$ as in \eqref{denote-bp-eq}, so that $N$ is 
the common length of both paths.
Set $u_0 = u$ and 
write $u_i$ for the entry of $T$ 
in position $(\tilde x_i, \tilde y_i)$ for $i \in [N-1]$. 
Then define $\theta_0 := \gamma_{|T|+1}(\row(T)b)$
where $|T|$ is the number of boxes in $T$
and let
\be\label{theta-def}
\theta_{i} := \begin{cases}
\gamma_{x_iy_i}(T,b) & \text{if }(x_i,y_i)=(\tilde x_i, \tilde y_i)\text{ and either }x_i \neq y_i\text{ or }u_{i-1} + 1 < u_i \\
\theta_{i-1} & \text{otherwise}\end{cases}
\ee
for $i \in [N-1]$. For each $0\leq i <N$ we have $\theta_{i} \in \{\emptyset\}\sqcup \cyc(z)$.

\begin{example}\label{Tb-ex-eq2}
Let $T  = \PO(51324) 
$
and $ b = 3154$
as in Example~\ref{Tb-ex-eq}. Then $u=3$  and
\[
\ba \path^\leq(T,u)&=\((x_i,y_i) : i=1,2,3\)
=\((1,3),(2,3),(3,3)\),
\\
 \path^<(T,u)&=\((\tilde x_i, \tilde y_i): i=1,2,3\)
 =\((1,3),(2,3),(3,3)\),
\ea
\]
so $u_0 = 3$, $u_1=4$, and $u_2=5$, while $\theta_0  = \emptyset$, $\theta_1 = \emptyset$, and $\theta_2 = \{3,4\}$.
\end{example}

\begin{lemma}\label{gamma-lem}  
For each position $(x,y)$ in the domain of $U := T\iarrow u$, the following holds:
\ben
\item[(a)] If $(x,y) = (x_i,y_i) = (\tilde x_i, \tilde y_i)$ for some $i \in [N]$, 
then
\[
U_{xy} = u_{i-1}\quand 
\gamma_{xy}(U, c) = \begin{cases} \gamma_{xy}(T, b) & \text{if $x=y$ and $i<N$ and $u_{i-1}+1=u_i$} \\
\theta_{i-1} &\text{otherwise}.\end{cases}\]

\item[(b)] If $(x,y) \in\{ (x_i,y_i) \neq (\tilde x_i, \tilde y_i)\}$ for some $i \in [N]$
with $x_i\neq y_i$ and $\tilde x_i \neq \tilde y_i$,
then
\[
U_{xy} = T_{xy}\quand
\gamma_{xy}(U, c) = \begin{cases} 
\gamma_{\tilde x_i \tilde y_i}(T,b) &\text{if }(x,y) = (x_i,y_i)
\\
\gamma_{x_i  y_i}(T,b) &\text{if }(x,y) = (\tilde x_i,\tilde y_i).
\end{cases}
\]

\item[(c)] If $(x,y) \in
\{(i,i),(i,i+1),(i+1,i+1)\}$ for  some $i \in [N]$
with $x_i= y_i \neq \tilde y_i$, 
then
\[
U_{xy} = T_{xy}\quand
\gamma_{xy}(U, c) = \begin{cases} 
\gamma_{i+1,i+1}(T,b) \neq \emptyset &\text{if }(x,y) = (i,i)
\\
\gamma_{i,i+1}(T,b) = \emptyset &\text{if }(x,y) = (i,i+1)
\\
\gamma_{ii}(T,b) \neq \emptyset &\text{if }(x,y) = (i+1,i+1).
\end{cases}
\]
In this case $(x_i,y_i) = (i,i)$, $(\tilde x_i, \tilde y_i) = (x_{i+1},y_{i+1}) = (i,i+1)$, and $(\tilde x_{i+1}, \tilde y_{i+1}) = (i+1,i+1)$.

\item[(d)] Otherwise,
$(x,y) \notin \path^\leq(T,u)\cup \path^<(T,u)$, 
$
U_{xy} = T_{xy}
$, and
$
\gamma_{xy}(U, c)  = \gamma_{xy}(T, b)
.$
\een
\end{lemma}


\begin{proof}
Suppose $V$ is a shifted tableau with all entries in $\ZZ$.
If we are given a total ordering $(i_1,j_1) < (i_2,j_2) < (i_3,j_3) <\dots$ of the boxes of $V$
such that the entries read in this order form an involution word $a$,
then we can define a tableau $\Gamma$ of the same shape as $V$ whose 
entry in box $(i_k,j_k)$ is the value of $\gamma_k(a)$.
Let $\Gamma^\row(V)$, $\Gamma^\col(V)$, $\Gamma^\swdiag(V)$, and $\Gamma^\nediag(V)$ denote the tableaux constructed in this way relative to the row, column, southwest diagonal, and northeast diagonal reading orders, respectively. These tableaux are only well-defined when the corresponding reading words are involution words.

If $V$ is an increasing shifted tableau with $\row(V) \in \iR(z)$,
then $\col(V)$ is also in $\iR(z)$ by Lemma~\ref{simK-lem},
so $\Gamma^\row(V)$ and $\Gamma^\col(V)$ are both defined.
In this case, since $\row(V)$ is transformed by $\col(V)$ by a sequence of swaps involving non-consecutive letters in adjacent positions (which we will refer to as ``commutations'' for the rest of this proof, slightly abusing our previous terminology),
it follows from part (a) of Lemma~\ref{marked-lem}
that we actually have $\Gamma^\row(V)=\Gamma^\col(V)$.

We now turn to the claims in lemma. The assertions about the values of $U_{xy}$
are straightforward from Definition~\ref{iarrow-def} since there are no repeated positions in the relevant bumping paths. It remains to justify the formulas for $\gamma_{xy}(U,c)$.
Define $T=T_0,T_1,T_2,\dots ,T_N = T\iarrow u = U$
and $\tilde T_i$ 
 as in the proof of Proposition~\ref{o-lem2}, and suppose there are exactly $p \in [N]$ iterations involving row insertion in the process to construct $T\iarrow u$.
 Because all of these tableaux have only unprimed entries, the numbers 
$u_i$ defined in the proof  of Proposition~\ref{o-lem2} coincide with the numbers $u_i$ defined above in this section.

Now consider the tableaux $\Gamma^\row(\tilde T_i)$ for $i<p$ and $\Gamma^\col(\tilde T_i)$ for $i\geq p$,
which are all well-defined by \eqref{will-be-shown-eq0} and \eqref{will-be-shown-eq}.
Figure~\ref{running-fig2} shows two examples of this sequence.
We may assume without loss
of generality that $b$ has length one so that $c$ is empty. 
Then the first tableau $\Gamma^\row(\tilde T_0)$ has value $\gamma_{xy}(T,b)$ for all $(x,y) \in T$
and its last box in the first row (containing $u=u_0$ in $\tilde T_0$) has value $\theta_0$. 
On the other hand, we have $\Gamma^\col(\tilde T_N)=\Gamma^\row(\tilde T_N) = \Gamma^\row(U)$ 
as $T_N = T\iarrow u = U$ is increasing with row reading word in $\iR(z)$.
Thus, each box $(x,y)$ in $\Gamma^\row(\tilde T_N)=\Gamma^\col(\tilde T_N)$ has 
entry $\gamma_{xy}(U,c)$ and our goal is to show that this value is as described by the given formulas.

\begin{figure}[h]
\[
\(T= \ytab{ \none & 5 & 6 \\ 1 & 3 & 4}\ \)\iarrow \(u=2\) \leadsto\left\{ \ba
 \Gamma^\row(\tilde T_0)&= \Gamma^\row\( \ytab{ \none & 5 & 6 \\ 1 & 3 & 4&2  }\) =   
 {\ytableausetup{boxsize = 0.75cm,aligntableaux=center}\scriptsize\begin{ytableau} 
 \none & \{4,7\} & \emptyset \\ \{1,3\} & \{2,5\} & \emptyset & \emptyset  \end{ytableau}},
\\[-8pt]\\
 \Gamma^\row(\tilde T_1)&=  \Gamma^\row\(\ytab{ \none & 5 & 6 & 3 \\ 1 & 2 & 4   }\)= {\ytableausetup{boxsize = 0.75cm,aligntableaux=center}\scriptsize\begin{ytableau} \none & \{4,7\} & \emptyset & \{2,5\} \\ \{1,3\} & \emptyset & \emptyset   \end{ytableau}},
\\[-8pt]\\
 \Gamma^\col(\tilde T_2) &=  \Gamma^\col\( \ytab{\none &\none & 5 \\ \none & 3 & 6   \\ 1 & 2 & 4    }\)   ={\ytableausetup{boxsize = 0.75cm,aligntableaux=center}\scriptsize\begin{ytableau}
\none &\none & \{4,7\} \\
 \none & \{2,5\} & \emptyset  \\ \{1,3\} & \emptyset & \emptyset    \end{ytableau}},
\\[-8pt]\\
  \Gamma^\col(\tilde T_3)&=\Gamma^\col\( \ytab{ \none & 3 & 5   \\ 1 & 2 & 4  & 6}\)=
{\ytableausetup{boxsize = 0.75cm,aligntableaux=center}\scriptsize\begin{ytableau}
 \none & \{2,5\} & \{4,7\}   \\ \{1,3\} & \emptyset & \emptyset  & \emptyset
\end{ytableau}}.
\ea\right.
\]
\caption{Example for the proof of Lemma~\ref{gamma-lem}; compare with Figure~\ref{running-fig1}(a).}\label{running-fig2}
\end{figure}

\begin{figure}[h]
\[ 
\(T= \ytab{ \none & \none & 7 \\  \none & 5 & 6 \\ 1 & 3 & 5} \ \) \iarrow \(u=4\)\leadsto\left\{
\ba
 \Gamma^\row(\tilde T_0)&= \Gamma^\row\( \ytab{ \none & \none & 7 \\  \none & 5 & 6 \\ 1 & 3 & 5 & 4}\)  ={\ytableausetup{boxsize = 0.75cm,aligntableaux=center}
\scriptsize\begin{ytableau} \none & \none & \{4,8\} \\  \none & \{6,7\} & \emptyset \\ \{1,2\} & \{3,5\} & \emptyset & \emptyset \end{ytableau}},
\\[-8pt]\\
 \Gamma^\row(\tilde T_1)&= \Gamma^\row\(\ytab{ \none & \none & 7 \\  \none & 5 & 6 & 5 \\ 1 & 3 & 4}\)=  {\ytableausetup{boxsize = 0.75cm,aligntableaux=center}\scriptsize\begin{ytableau} \none & \none & \{4,8\} \\  \none & \{6,7\} & \emptyset & \emptyset \\ \{1,2\} & \{3,5\} & \emptyset\end{ytableau}},
\\[-8pt]\\
 \Gamma^\col(\tilde T_2)&=  \Gamma^\col\(\ytab{ \none & \none & 6 \\  \none & \none & 7 \\  \none & 5 & 6 \\ 1 & 3 & 4}\)= {\ytableausetup{boxsize = 0.75cm,aligntableaux=center}
\scriptsize\begin{ytableau} \none & \none & \emptyset \\  \none & \none & \emptyset \\  \none & \{4,8\} & \{6,7\} \\ \{1,2\} & \{3,5\} & \emptyset\end{ytableau}},
\\[-8pt]\\
   \Gamma^\col(\tilde T_3)&=  \Gamma^\col\(\ytab{ \none & \none & 7 \\  \none & 5 & 6 \\ 1 & 3 & 4 & 7}\)={\ytableausetup{boxsize = 0.75cm,aligntableaux=center}
\scriptsize\begin{ytableau} \none & \none & \{6,7\} \\  \none & \{4,8\} & \emptyset \\ \{1,2\} & \{3,5\}& \emptyset & \emptyset\end{ytableau}}.
  \ea\right.\]
\caption{Example for the proof of Lemma~\ref{gamma-lem}; compare with Figure~\ref{running-fig1}(b).}\label{running-fig3}
\end{figure}

For each $i$ let $\varphi_i$ be the entry of $\Gamma^\row(\tilde T_i)$ in the unique box that is not in $T$,
so that $\varphi_0 =\theta_0$.
First choose $i \in[p-1]$ so that $(x_i,y_i)$ is not on the main diagonal.
If $(x_i,y_i)=(\tilde x_i,\tilde y_i)$,
then we can transform $\row(\tilde T_{i-1})$ to $\row(\tilde T_{i})$ using only commutations, so
it follows from part (a) of Proposition~\ref{o-lem2}
that $\Gamma^\row(\tilde T_i)$ is formed from $\Gamma^\row(\tilde T_{i-1})$
by moving box $(x_i,y_i)$ to the end of row $i+1$
and then moving $\varphi_{i-1}$ from the end of row $i$ to replace box $(x_i,y_i)$.
Likewise, if $(x_i,y_i)\neq (\tilde x_1,\tilde y_1)$,
then transforming $\row(\tilde T_{i-1})$ to $\row(\tilde T_i)$ will involve one braid relation as we must have 
$(\tilde x_i,\tilde y_i) = (x_i,y_i+1)$ and
$u_{i-1} = T_{x_iy_i} = T_{\tilde x_i\tilde y_i}-1$.
In this case it follows using parts (a) and (b) of Proposition~\ref{o-lem2}
that  $\Gamma^\row(\tilde T_i)$ is formed from $\Gamma^\row(\tilde T_{i-1})$
by moving $\varphi_{i-1}$ from the end of row $i$ to the end of row $i+1$ 
and switching the entries in the adjacent boxes 
$(x_i,y_i)$ and $ (\tilde x_i,\tilde y_i)$.

It follows by induction that $\varphi_i = \theta_i$ for all $i \in [p-1]$.
When $p=N$, 
these observations describe a precise sequence of transitions that take us from 
$\Gamma^\row(\tilde T_{0})$ to $\Gamma^\row(\tilde T_{N})$.
Comparing this process with  the definition of $\theta_i$
shows that the desired formulas for $\gamma_{xy}(U,c)$ all hold.

Assume instead that $p<N$. 
It follows by similar reasoning that if $p<i\leq N$, 
then $\Gamma^\col(\tilde T_i)$ is formed from $\Gamma^\col(\tilde T_{i-1})$ in one of two ways. 
If $(x_i,y_i)=(\tilde x_i,\tilde y_i)$,
then we move box $(x_i,y_i)$ to the end of column $i+1$
and then move $\varphi_{i-1}$ from the end of column $i$ to replace box $(x_i,y_i)$.
If $(x_i,y_i)\neq (\tilde x_i,\tilde y_i)$,
then we move $\varphi_{i-1}$ from the end of column $i$ to the end of column $i+1$ and switch the entries in boxes
$(x_i,y_i)$ and $ (\tilde x_i,\tilde y_i)$.

It remains to compare $\Gamma^\row(\tilde T_{p-1})$ with $\Gamma^\col(\tilde T_{p})$.
We wish to justify the following claims:
\ben
\item[(1)] If $(x_p,y_p)=(\tilde x_p,\tilde y_p)=(p,p)$ and $u_{p-1}+1<u_p$,
then $\Gamma^\col(\tilde T_{p})$ is formed from $\Gamma^\row(\tilde T_{p-1})$
by moving box $(p,p)$ to the end of column $p+1$
and then moving $\varphi_{p-1}$ to replace box $(p,p)$.

\item[(2)] If $(x_p,y_p)=(\tilde x_p,\tilde y_p)=(p,p)$ and $u_{p-1}+1=u_p$,
then $\Gamma^\col(\tilde T_{p})$ is formed from $\Gamma^\row(\tilde T_{p-1})$
by moving $\varphi_{p-1}$ from the end of row $p$ to the end of column $p+1$.

\item[(3)] If $(x_p,y_p)=(p,p)$ and $(\tilde x_p,\tilde y_p)=(p,p+1)$,
then $\varphi_{p-1}$ and box $(p,p+1)$ of $\Gamma^\row(\tilde T_{p-1})$ are both the null element $\emptyset$, while boxes $(p,p)$ and $(p+1,p+1)$ are both present with respective non-null elements $\alpha$ and $\beta$. 
 In this case, 
$\Gamma^\col(\tilde T_{p})$ is formed from $\Gamma^\row(\tilde T_{p-1})$ 
by removing $\varphi_{p-1}$ and placing $\emptyset$ in boxes $(p+1,p+1)$ and $(p+2,p+1)$,
  $\alpha$ in box $(p,p+1)$, and $\beta $ in box $(p,p)$.

  \item[(4)] Together, (2) and (3) imply that if $(x_p,y_p)=(p,p)$ and $(\tilde x_p,\tilde y_p)=(p,p+1)$,
then $p<N$ and $\Gamma^\col(\tilde T_{p+1})$ is formed from $\Gamma^\row(\tilde T_{p-1})$ 
by moving $\varphi_{p-1} = \emptyset$ from the end of row $p$ to the end of column $p+2$
and then swapping the entries in boxes $(p,p)$ and $(p+1,p+1)$;
moreover, both tableaux have $\emptyset$ in position $(p,p+1)$.
\een
Putting together these claims with our observations about 
$\Gamma^\row(\tilde T_{i})$ for $i<p$ and $\Gamma^\col(\tilde T_{i})$ for $i>p$
completely describes how
$\Gamma^\row(\tilde T_{0})$ evolves into $\Gamma^\col(\tilde T_{N})=\Gamma^\row(\tilde T_{N})$
during the bumping process that defines $T\iarrow u$.
Once again, comparing this process with  the definition of $\theta_i$
shows that the desired formulas for $\gamma_{xy}(U,c)$ all hold.

It remains to prove claims (1), (2), and (3). 
The first two claims correspond to the case when $u_{p-1}  < T_{pp}$.
For this situation, define $V$ and $W$ as in the $\lceil u_{p-1} \rceil < T_{pp}$ case of the proof of Proposition~\ref{o-lem2}. 
Since $\row(\tilde T_{p-1}) = \row(V)$,
it follows that $\Gamma^\row(V)$ has the same entry in box $(2i,2j)$ (respectively, box $(2p-1,2p-1)$) as $\Gamma^\row(\tilde T_{p-1})$ does in each box $(i,j) \in T$ (respectively, the unique box not in $T$).
Likewise, as $\col(\tilde T_{p}) = \col(W)$,
it follows that $\Gamma^\col(W)$ has the same entry in each box $(2i,2j)$ (respectively, box $(2p+1,2p+1)$) as $\Gamma^\col(\tilde T_{p})$ does in each box $(i,j) \in T$ (respectively, the unique box not in $T$). Finally, since the row reading word of $V$ (respectively, the column reading word of $W$) can be transformed to its southwest (respectively, northeast) diagonal  reading word 
by a sequence of commutations as described in the proof of Proposition~\ref{o-lem2},
we deduce from part (a) of Lemma~\ref{marked-lem} that $\Gamma^\row(V)= \Gamma^\swdiag(V) $ and $\Gamma^\col(W)=\Gamma^\nediag(W) $.
One can observe these properties for the example in Figure~\ref{running-fig2}, where we have
 \[
 \Gamma^\row(V)= \Gamma^\swdiag(V)  =   \Gamma^\swdiag\(\smalltab{
     \none[\cdot] &  \none[\cdot] &  \none[\cdot] & \none[\cdot]&  \none[\cdot] & \none[\cdot] \\
  \none[\cdot] &  \none[\cdot] &  \none[\cdot] &5&  \none[\cdot] & 6 \\
  \none[\cdot] &  \none[\cdot] & 3&  \none[\cdot] &  \none[\cdot] &  \none[\cdot] \\
  \none[\cdot] & 1 &  \none[\cdot] & 2 &  \none[\cdot] &4 \\
\none[\cdot]&  \none[\cdot] &  \none[\cdot] &  \none[\cdot] &  \none[\cdot] &  \none[\cdot]
  }\)
  =
{\ytableausetup{boxsize = 0.75cm,aligntableaux=center}\scriptsize\begin{ytableau}
     \none[\cdot] &  \none[\cdot] &  \none[\cdot] & \none[\cdot]&  \none[\cdot] & \none[\cdot] \\
  \none[\cdot] &  \none[\cdot] &  \none[\cdot] &\{4,7\}&  \none[\cdot] & \emptyset \\
  \none[\cdot] &  \none[\cdot] & \{2,5\}&  \none[\cdot] &  \none[\cdot] &  \none[\cdot] \\
  \none[\cdot] & \{1,3\} &  \none[\cdot] & \emptyset &  \none[\cdot] & \emptyset \\
\none[\cdot]&  \none[\cdot] &  \none[\cdot] &  \none[\cdot] &  \none[\cdot] &  \none[\cdot]
  \end{ytableau}}
\]
and
\[
\Gamma^\col(W)=\Gamma^\nediag(W) =   \Gamma^\nediag\(\smalltab{
   \none[\cdot] &  \none[\cdot] &  \none[\cdot] & \none[\cdot]& 5 & \none[\cdot] \\
  \none[\cdot] &  \none[\cdot] &  \none[\cdot] &3&  \none[\cdot] & 6 \\
  \none[\cdot] &  \none[\cdot] & \none[\cdot]&  \none[\cdot] &  \none[\cdot] &  \none[\cdot] \\
  \none[\cdot] & 1 &  \none[\cdot] & 2 &  \none[\cdot] &4 \\
\none[\cdot]&  \none[\cdot] &  \none[\cdot] &  \none[\cdot] &  \none[\cdot] &  \none[\cdot]
  }\)
  =
  {\ytableausetup{boxsize = 0.75cm,aligntableaux=center}\scriptsize\begin{ytableau}
     \none[\cdot] &  \none[\cdot] &  \none[\cdot] & \none[\cdot]& \{4,7\} & \none[\cdot] \\
  \none[\cdot] &  \none[\cdot] &  \none[\cdot] &\{2,5\}&  \none[\cdot] & \emptyset \\
  \none[\cdot] &  \none[\cdot] & \none[\cdot]&  \none[\cdot] &  \none[\cdot] &  \none[\cdot] \\
  \none[\cdot] & \{1,3\} &  \none[\cdot] & \emptyset &  \none[\cdot] &\emptyset \\
\none[\cdot]&  \none[\cdot] &  \none[\cdot] &  \none[\cdot] &  \none[\cdot] &  \none[\cdot]
    \end{ytableau}}.
 \]
 
Given the observations in the preceding paragraph, to prove claims (1) and (2), we just need to check that $\Gamma^\nediag(W)$ is formed from $\Gamma^\swdiag(V)$
either by shifting boxes $(2p-1,2p-1)$ and $(2p,2p)$ up one row and one column when 
 $u_{p-1}+1<u_p$, or by moving box $(2p-1,2p-1)$ to $(2p+1,2p+1)$ when $u_{p-1}+1=u_p$.
This is equivalent to showing that   $\Gamma^\nediag(V)=\Gamma^\swdiag(V)$ when $u_{p-1}+1<u_p$
and that 
$\Gamma^\nediag(V)$ is formed from $\Gamma^\swdiag(V)$ by swapping boxes $(2p-1,2p-1)$ and $(2p,2p)$ when $u_{p-1}+1=u_p$.
In the first case, the diagonals of $V$ have no consecutive entries and so
 can be reordered using only commutations, so the identity $\Gamma^\nediag(V)=\Gamma^\swdiag(V)$ follows from part (a) of Lemma~\ref{marked-lem}.
When $u_{p-1}+1=u_p$, we can also reverse all diagonals in $V$ using only commutations
to go from  the southwest diagonal reading word to northeast diagonal reading word,
except for one step that exchanges the consecutive numbers in boxes $(2p-1,2p-1)$ and $(2p,2p)$ when these have been pulled to the start of the relevant word.
By part (c) of Lemma~\ref{marked-lem}, this has the effect of swapping boxes $(2p-1,2p-1)$ and $(2p,2p)$ in $\Gamma^\swdiag(V)$ to form $\Gamma^\nediag(V)$, as desired.
We conclude that our first two claims (1) and (2) both hold.

Suppose instead that we are in the situation of claim (3), so that $u_{p-1}  = T_{pp}$.
It follows from Remark~\ref{iarrow-rmk}(d) that $\varphi_{p-1}$ and box $(p,p+1)$ of $\Gamma^\row(\tilde T_{p-1})$ are both null.
Define $V$ as in the $\lceil u_{p-1} \rceil = T_{pp}$ case of the proof of Proposition~\ref{o-lem2}. 
Since we can transform $\row(\tilde T_{p-1})$ to $ \row(V)$ by a sequence of commutations
followed by one braid relation, it follows from Lemma~\ref{marked-lem}
that
\begin{itemize}
\item  box $(2p+1,2p+1)$ of 
$\Gamma^\row(V)$ has the same entry as the box of $\Gamma^\row(\tilde T_{p-1})$ not in $T$;

\item box $(2p,2p)$ of $\Gamma^\row(V)$ has the same entry as box $(p,p+1)$ of $\Gamma^\row(\tilde T_{p-1})$;

\item box $(2p,2p+2)$ of $\Gamma^\row(V)$ has the same entry as box $(p,p)$ of $\Gamma^\row(\tilde T_{p-1})$;

\item any other box $(2i,2j)$ of $\Gamma^\row(V)$ has the same entry as box $(i,j)$ of $\Gamma^\row(\tilde T_{p-1})$.

\end{itemize}
Alternatively, as $\col(\tilde T_{p}) = \col(V)$,
it follows that $\Gamma^\col(V)$ has the same entry in each box $(2i,2j)$ (respectively, box $(2p+1,2p+1)$) as $\Gamma^\col(\tilde T_{p})$ does in each box $(i,j) \in T$ (respectively, the unique box not in $T$). Finally, since the row reading word of $V$ (respectively, the column reading word of $V$) can be transformed to its southwest (respectively, northeast) diagonal  reading word 
by a sequence of commutations, we have $\Gamma^\row(V)= \Gamma^\swdiag(V) $ and $\Gamma^\col(V)=\Gamma^\nediag(V) $.
One can observe these properties in the example in Figure~\ref{running-fig3}, where we have
\[
 \Gamma^\row(V)= \Gamma^\swdiag(V)  =   \Gamma^\swdiag\(
 \smalltab{
     \none[\cdot] &  \none[\cdot] &  \none[\cdot] & \none[\cdot]&  \none[\cdot] & \none[\cdot] \\
     \none[\cdot] &  \none[\cdot] &  \none[\cdot] & \none[\cdot]&  \none[\cdot] & 7\\
     \none[\cdot] &  \none[\cdot] &  \none[\cdot] & \none[\cdot]& 6& \none[\cdot] \\
  \none[\cdot] &  \none[\cdot] &  \none[\cdot] &5&  \none[\cdot] & 6 \\
  \none[\cdot] &  \none[\cdot] & \none[\cdot]&  \none[\cdot] &  \none[\cdot] &  \none[\cdot] \\
  \none[\cdot] & 1 &  \none[\cdot] & 3 &  \none[\cdot] &4 \\
\none[\cdot]&  \none[\cdot] &  \none[\cdot] &  \none[\cdot] &  \none[\cdot] &  \none[\cdot]
  }\)
  =  {\ytableausetup{boxsize = 0.75cm,aligntableaux=center}\scriptsize\begin{ytableau}
       \none[\cdot] &  \none[\cdot] &  \none[\cdot] & \none[\cdot]&  \none[\cdot] & \none[\cdot] \\
     \none[\cdot] &  \none[\cdot] &  \none[\cdot] & \none[\cdot]&  \none[\cdot] & \{4,8\}\\
     \none[\cdot] &  \none[\cdot] &  \none[\cdot] & \none[\cdot]& \emptyset & \none[\cdot] \\
  \none[\cdot] &  \none[\cdot] &  \none[\cdot] &\emptyset&  \none[\cdot] & \{6,7\} \\
  \none[\cdot] &  \none[\cdot] & \none[\cdot]&  \none[\cdot] &  \none[\cdot] &  \none[\cdot] \\
  \none[\cdot] & \{1,2\} &  \none[\cdot] & \{3,5\} &  \none[\cdot] &\emptyset \\
\none[\cdot]&  \none[\cdot] &  \none[\cdot] &  \none[\cdot] &  \none[\cdot] &  \none[\cdot]
 \end{ytableau}}.
\]
and
\[
 \Gamma^\col(V)= \Gamma^\nediag(V)  =   \Gamma^\nediag\(
 \smalltab{
     \none[\cdot] &  \none[\cdot] &  \none[\cdot] & \none[\cdot]&  \none[\cdot] & \none[\cdot] \\
     \none[\cdot] &  \none[\cdot] &  \none[\cdot] & \none[\cdot]&  \none[\cdot] & 7\\
     \none[\cdot] &  \none[\cdot] &  \none[\cdot] & \none[\cdot]& 6& \none[\cdot] \\
  \none[\cdot] &  \none[\cdot] &  \none[\cdot] &5&  \none[\cdot] & 6 \\
  \none[\cdot] &  \none[\cdot] & \none[\cdot]&  \none[\cdot] &  \none[\cdot] &  \none[\cdot] \\
  \none[\cdot] & 1 &  \none[\cdot] & 3 &  \none[\cdot] &4 \\
\none[\cdot]&  \none[\cdot] &  \none[\cdot] &  \none[\cdot] &  \none[\cdot] &  \none[\cdot]
  }\)
  =  {\ytableausetup{boxsize = 0.75cm,aligntableaux=center}\scriptsize\begin{ytableau}
       \none[\cdot] &  \none[\cdot] &  \none[\cdot] & \none[\cdot]&  \none[\cdot] & \none[\cdot] \\
     \none[\cdot] &  \none[\cdot] &  \none[\cdot] & \none[\cdot]&  \none[\cdot] & \emptyset \\
     \none[\cdot] &  \none[\cdot] &  \none[\cdot] & \none[\cdot]& \emptyset & \none[\cdot] \\
  \none[\cdot] &  \none[\cdot] &  \none[\cdot] &\{4,8\}&  \none[\cdot] & \{6,7\} \\
  \none[\cdot] &  \none[\cdot] & \none[\cdot]&  \none[\cdot] &  \none[\cdot] &  \none[\cdot] \\
  \none[\cdot] & \{1,2\} &  \none[\cdot] & \{3,5\} &  \none[\cdot] &\emptyset \\
\none[\cdot]&  \none[\cdot] &  \none[\cdot] &  \none[\cdot] &  \none[\cdot] &  \none[\cdot]
 \end{ytableau}}.
\]

By the facts just listed, to prove claim (3), it suffices to check that $\Gamma^\nediag(V)$ is formed from $\Gamma^\swdiag(V)$ by swapping boxes $(2p,2p)$ and $(2p+2,2p+2)$.
For this, observe that we can reverse the diagonals of $V$ to go from the southwest diagonal reading word to the northeast diagonal reading word using only commutations, except when we need to reorder the consecutive entries in boxes $(2p,2p)$, $(2p+1,2p+1)$, and $(2p+2,2p+2)$ after these have been brought to the start of the relevant word.
Since this reordering is accomplished by the sequence of swaps 
$(u_p+2)(u_p+1) u_p\cdots \to  (u_p+1)(u_p+2) u_p\cdots \to (u_p+1)u_p(u_p+2)\cdots\to u_p(u_p+1)(u_p+2)\cdots$, it follows from parts (a) and (c) of Lemma~\ref{marked-lem} that exchanging boxes $(2p,2p)$ and $(2p+2,2p+2)$ 
in $\Gamma^\swdiag(V)$ produces $\Gamma^\nediag(V)$, as needed.
The completes the proof of claim (3), which also finishes the proof of the lemma.
\end{proof}

\subsection{A formula to compute primed boxes from marked cycles}\label{cyc-sect2}

Suppose $\hat a$ is a primed involution word with unprimed form $a= \unprime(\hat a)$.
In this section we will develop some notation to express a formula for 
$\PO(\hat a)$ and $\QO(\hat a)$ in terms of $\PO(a)$, $\QO(a)$, and the set of marked cycles $\marked(\hat a)$.

In more detail, 
if $a=a_1a_2\cdots a_n \in \iR(z)$ for some $z \in I_\ZZ$
and $T = \emptyset \iarrow a_1\iarrow \dots \iarrow a_i$ for some $i \in [n]$,
then the entries of $T$ on the main diagonal form a strictly increasing sequence
and the indices of these entries in $\row(T) a_{i+1}a_{i+2}\cdots a_n$
are a sequence of commutations that each contribute one 2-cycle of $z$.
Arranging these sequences into a two-line array gives 
what we   call the \defn{cycle sequence} $\cseq_i(a)$.
The successive values of $\cseq_i(a)$ for $i=1,2,\dots,n$
can only change in a small of number of ways. Our main formula will involve a permutation of $\cyc(z)$
defined by these changes.
 
 As in Section~\ref{cyc-sect1}, suppose
 $T$ is an increasing shifted tableau 
 and $b$ is a word with $\row(T)b \in \iR(z)$.
 If $T$ has exactly $q$ rows, then  the \defn{cycle sequence}
 $\cseq(T,b)$ is the two-line array
\be\cseq(T,b) := \left[\barr{llll} \gamma_{11}(T,b) & \gamma_{22}(T,b) & \dots & \gamma_{qq}(T,b)
\\
T_{11} & T_{22} & \dots & T_{qq}
\earr\right]
.\ee
If $T = \PO(51324)
$
and $ b = 3154$ as in Example~\ref{Tb-ex-eq}
then
\[
 \cseq(T,b) = \left[\barr{ll} \{2,5\} & \{1,6\} \\ 1 & 3 \earr\right] = \cseq_5(513243154)
 . \]
The second row of $\cseq(T,b)$ is strictly increasing
and the elements in the
 first row are distinct 2-cycles of $z$,
since  the index of
$T_{ii}$ in $\row(T)b$ is a commutation for all diagonal positions $(i,i)$ in $T$.
 For involution words $a=a_1a_2\cdots a_n$ and $0\leq i \leq n$, we define
$ \cseq_i(a) := \cseq(T,b)$
where $T = \PO(a_1a_2\cdots a_i)$ and $b=a_{i+1}a_{i+2}\cdots a_n.$

We introduce some auxiliary notation to help compare $ \cseq_i(a) $ with $ \cseq_{i-1}(a)$.
Assume $b$ is nonempty and 
let $u=u_0$ be its first letter. Denote 
the weak and strict bumping paths
resulting from inserting $u$ into $T$ as in \eqref{denote-bp-eq}.
Set $u_i := T_{\tilde x_i \tilde y_i}$ for $i \in [N-1]$ and define 
 $\theta_0 := \gamma_{|T|+1}(\row(T)b)$ and $\theta_i$ for $i \in [N-1]$ by \eqref{theta-def}.
Finally, define the sequence
\be\label{dbump-eq}
\dbump(T,b) := \( (y_i,\tilde y_i, u_{i-1}, \theta_{i-1}) : i=1,2,\dots,p\)\ee
where $p$ is the index of the unique diagonal position in $ \path^\leq(T,u)$
or else $p=N$.

Continuing from Example~\ref{Tb-ex-eq2},
we see that if $T  = \PO(51324) 
$
and $ b = 3154$
then $p=3$ and $\dbump(T,b) = \( (1,1,3,\emptyset), (2,2,4,\emptyset), (3,3,5,\{3,4\})\)$.
We think of $\dbump(T,b) $ as a record of the change between $T\iarrow u$ and $T$, 
and we can use it to compute successive values of $\theta_i$ by the formula
\be\label{succ-theta-eq} \theta_i = \begin{cases}
\gamma_{i,y_i}(T,b)
&\text{if $y_i=\tilde y_i$ and either $i \neq y_i$ or $u_{i - 1} +1 <u_i$} \\ 
\theta_{i-1}&\text{otherwise}
\end{cases}
\quad\text{for $i \in [p-1]$.}\ee
For any involution word $a=a_1a_2\cdots a_n \in \iR(z)$ and $j \in [n]$,
define $\dbump_j(a) := \dbump(T,b)$ where $T = \PO(a_1a_2\cdots a_{j-1})$
and $b=a_ja_{j+1}\cdots a_n$.
The following result shows that $ \cseq_{j}(a)$ 
is completely determined by $ \cseq_{j-1}(a)$ and $\dbump_j(a)$.


\begin{lemma}\label{cseq-lem}
 Let $a$ be an (unprimed) involution word and choose $j \in [\ell(a)]$. 
 Suppose 
 \[
 \cseq_{j-1}(a) = 
  \left[\barr{llll} \gamma_1 & \gamma_2 & \dots & \gamma_q
\\
c_1 & c_2 & \dots & c_q \earr\right]
\quand
\dbump_j(a)= \{ (y_i,\tilde y_i, u_{i-1}, \theta_{i-1})\}_{i \in [p]}
.\]
Exactly one of the following cases applies:
 \ben
  \item[(a)] The sequence $\wpath_j(a)$ ends before reaching the main diagonal if and only if $p<y_p$.
  In this case $i$ appears in $\QO(a)$ in an off-diagonal position and $\cseq_{j}(a) = \cseq_{j-1}(a)$.

  \item[(b)]  The sequence $\wpath_j(a)$ terminates on the main diagonal if and only if $p=y_p=\tilde y_p=q+1$.
  In this case  $i$ appears in $\QO(a)$ in position $(q+1,q+1)$ and
   \[
 \cseq_{j}(a) = 
  \left[\barr{lllll} \gamma_1 & \gamma_2 & \dots & \gamma_q & \theta_{q}
\\
c_1 & c_2 & \dots & c_q & u_{q} \earr\right]
.\]
  
\item[(c)] The sequences $\wpath_j(a)$ and  $\spath_j(a)$ reach (but do not terminate on) the main diagonal in the same row 
if and only if $p=y_p = \tilde y_p \leq q$. In this case $i'$ appears in $\QO(a)$ and we have 
\[u_{p-1}+1\leq   c_p\quand \cseq_{j}(a) = 
  \left[\barr{lllllll} \gamma_1  & \dots & \gamma_{p-1} & \eta & \gamma_{p+1} & \dots & \gamma_q
\\
c_1  & \dots & c_{p-1} & u_{p-1} & c_{p+1} & \dots & c_q \earr\right],
\]
where  $\eta:=\gamma_p$ if $u_{p-1}+1=c_p$ and $\eta:=\theta_{p-1}$ if $u_{p-1}+1<c_p$.

\item[(d)] The sequences $\wpath_j(a)$ and  $\spath_j(a)$ reach the main diagonal in different rows
if and only if
 $p=y_p < \tilde y_p=p+1 \leq q$. In this case $i'$ appears in $\QO(a)$ and we have 
 \[u_{p-1}=c_p \quand \cseq_{j}(a) = 
  \left[\barr{llllllll} \gamma_1  & \dots & \gamma_{p-1} & \gamma_{p+1} & \gamma_p & \gamma_{p+2} & \dots & \gamma_q
\\
c_1  & \dots & c_{p-1} & c_p & c_{p+1} &c_{p+2} & \dots & c_q \earr\right].
\]

 \een
 \end{lemma}
 

 \begin{proof}
 The assertion that exactly one of these cases applies follows from Proposition~\ref{bumping-prop}.
The claims about $u_{p-1}$ in cases (c) and (d) are clear from how $\PO(a_1a_2\cdots a_{j_1}) \iarrow a_j$ is defined.
The description of
 $\cseq_j(a)$ is immediate from the formulas in  Lemma~\ref{gamma-lem}.
\end{proof}


Putting all of this together, we associate a permutation of $\binom{\ZZ}{2} := \{ \{ i,j\} : i,j \in \ZZ,\ i<j\}$
to each involution word. 
Let $a=a_1a_2\dots a_n$ be an (unprimed) involution word for some $z \in I_\ZZ$. 
For each $i \in[n]$, let $\tpi_i(a)$ be the following permutation of $\binom{\ZZ}{2}$ with support in $\cyc(z)$.
If $\cseq_{i-1}(a)$ and $\cseq_{i}(a)$ are equal or have different lengths then  $\tpi_i(a) := 1$.
 Otherwise,
writing 
\[\cseq_{i-1}(a) =   \left[\barr{llll} \gamma_1 & \gamma_2 & \dots & \gamma_q
\\
c_1 & c_2 & \dots & c_q \earr\right]
\quand
\cseq_{i}(a) =   \left[\barr{llll} \eta_1 & \eta_2 & \dots & \eta_q
\\
d_1 & d_2 & \dots & d_q \earr\right],\]
there is either a unique index $j \in [q]$ with $d_j < c_j$,
or a unique index $j \in [q-1]$ with $\gamma_{j+1}=\eta_j \neq \gamma_j=\eta_{j+1}$,
and in both cases
we define $\tpi_i(a)$ to be the transposition of $\binom{\ZZ}{2}$ that swaps  $\eta_j$ and $ \gamma_{j}$ while fixing all other elements.
We then let $\tpi(a) := \tpi_1(a) \tpi_2(a) \cdots \tpi_n(a)$.

\begin{example}\label{cseq-complete-ex}
Suppose $a = 513243154$.
This word belongs to $\iR(z)$ for $z=(1,6)(2,5)(3,4) \in I_\ZZ$.
The successive values of $\PO(a_1a_2\cdots a_i)$ are
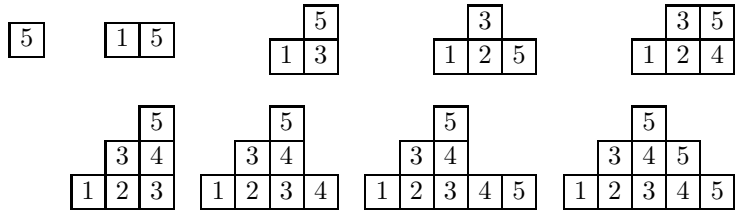
\[ 
\ba
 \smalltab{   5}
&&  \smalltab{  1 & 5}
&& \smalltab{  \none & 5 \\ 1 & 3}
&& \smalltab{  \none & 3 \\ 1 & 2 & 5}
&&\smalltab{ \none & 3 & 5 \\ 1 & 2 & 4}
\\[-8pt]\\&&   \smalltab{ \none& \none & 5 \\ \none & 3 & 4 \\ 1 & 2 & 3}
&&\smalltab{ \none& \none & 5 \\ \none & 3 & 4 \\ 1 & 2 & 3 & 4}
&& \smalltab{ \none& \none & 5 \\ \none & 3 & 4 \\ 1 & 2 & 3 & 4 & 5}
&& \smalltab{ \none& \none & 5 \\ \none & 3 & 4 & 5 \\ 1 & 2 & 3 & 4 & 5}
 \ea
\]
and the successive values of $\gamma_{xy}(T,b)$
for $T = \PO(a_1a_2\cdots a_i)$ and $b=a_{i+1}a_{i+2}\cdots a_9$ are
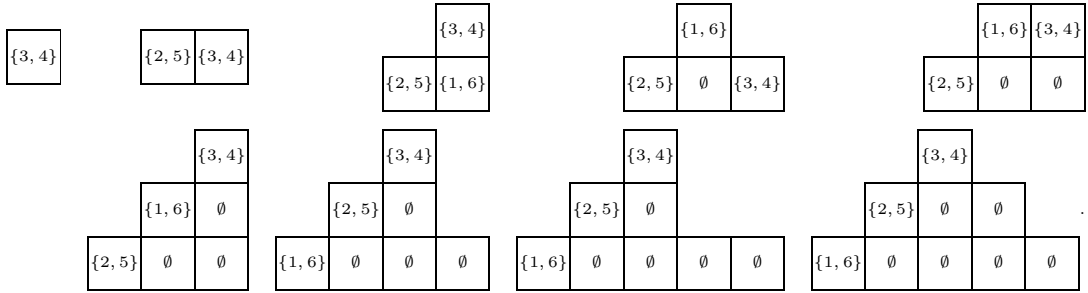
\[
\ytableausetup{boxsize = 0.70cm,aligntableaux=center}
{\tiny
\ba
 \begin{ytableau}   \{3,4\}\end{ytableau}
&&    \begin{ytableau}   \{2,5\} & \{3,4\}\end{ytableau}
&&   \begin{ytableau}   \none & \{3,4\} \\ \{2,5\} & \{1,6\}\end{ytableau}
&&   \begin{ytableau}   \none & \{1,6\} \\ \{2,5\} & \emptyset & \{3,4\}\end{ytableau}
&&   \begin{ytableau} \none & \{1,6\} & \{3,4\} \\ \{2,5\} & \emptyset & \emptyset\end{ytableau}
\\[-8pt]\\
 &&   \begin{ytableau} \none& \none & \{3,4\} \\ \none & \{1,6\} & \emptyset \\ \{2,5\} & \emptyset & \emptyset\end{ytableau}
&&   \begin{ytableau} \none& \none & \{3,4\} \\ \none & \{2,5\} & \emptyset \\ \{1,6\} & \emptyset & \emptyset & \emptyset\end{ytableau}
&&   \begin{ytableau} \none& \none & \{3,4\} \\ \none & \{2,5\} & \emptyset \\ \{1,6\} & \emptyset & \emptyset & \emptyset & \emptyset\end{ytableau}
&&   \begin{ytableau} \none& \none & \{3,4\} \\ \none & \{2,5\} & \emptyset & \emptyset \\ \{1,6\} & \emptyset & \emptyset & \emptyset & \emptyset\end{ytableau}.
 \ea
}
\]
Thus, we have 
\[{\small
\ba \cseq_1(a) &=  \left[\barr{lllllllll} \{3,4\} \\ 5 \earr\right],
&
\cseq_4(a) =  \cseq_5(a) &=  \left[\barr{lllllllll}\{2,5\} & \{1,6\} \\ 1 & 3 \earr\right], 
\\
\cseq_2(a) &=  \left[\barr{lllllllll} \{2,5\} \\ 1 \earr\right],
&
\cseq_6(a) &=  \left[\barr{lllllllll}\{2,5\} & \{1,6\} & \{3,4\} \\ 1 & 3 & 5 \earr\right],
\\
\cseq_3(a) &=  \left[\barr{lllllllll} \{2,5\} & \{3,4\} \\ 1 & 5 \earr\right],
&
\cseq_7(a) =\cseq_8(a) =\cseq_9(a) &=  \left[\barr{lllllllll}\{1,6\} & \{2,5\} & \{3,4\} \\ 1 & 3 & 5 \earr\right],
\ea}
\]
which means that 
$\tpi_1(a)=\tpi_3(a)=\tpi_5(a)=\tpi_6(a)=\tpi_8(a)=\tpi_9(a) = 1$
while
\[
\tpi_2(a) = ( \{2,5\} \leftrightarrow \{3,4\}),
\quad
\tpi_4(a) = ( \{1,6\} \leftrightarrow \{3,4\}),
\quand
\tpi_7(a) = ( \{1,6\} \leftrightarrow \{2,5\}),
\]
so we have
$\tpi(a) =  ( \{1,6\} \leftrightarrow \{3,4\})$.
\end{example}

Suppose $\hat a$ is a primed involution word with $a = \unprime(\hat a)$. 
Recall the definition of the set  $\marked(\hat a)$ from \eqref{marked-def}.
The following result is complementary to Proposition~\ref{unprime-prop}
and gives the second key ingredient in our proof of Theorem~\ref{ck-fkd-thm}.
This proposition reduces the task of locating the (diagonal) primes in 
$\PO(\hat a)$ and $\QO(\hat a)$ to understanding  $\tau(a)$ and $\marked(\hat a)$.

\begin{proposition}\label{tau-prop}
Suppose $\hat a \in \iR^+(z)$ and $a =\unprime(\hat a)$. Let $(i,j) \in \ZZ\times \ZZ$ and $\theta= \gamma_{ij}(\PO(a))$.
 If $i \neq j$ (respectively, $i=j$), then the  entry of $\PO(\hat a)$ (respectively $\QO(\hat a)$) 
 in position $(i,j)$ is primed if and only if 
$\theta \neq \emptyset$ and $ \tpi(a)(\theta) \in \marked(\hat a)$.
\end{proposition}


 \begin{proof}
One can define orthogonal Edelman-Greene insertion by a slightly different 
bumping process, in which an insertion tableau $\bPO(\hat a)$ is built up with diagonal
primes along with a recording tableau $\bQO(\hat a)$ having no diagonal primes, and 
then at the final stage $\PO(\hat a)$ and $\QO(\hat a)$ are formed by moving any diagonal primes in $\bPO(\hat a)$ to $\bQO(\hat a)$.
From this perspective the proposition is just locating the primes in $\bPO(\hat a)$. The following argument is organized around this observation.

Let $n = \ell(\hat a)=\ell(a)$ and
form $\bPO(\hat a)$  from $\PO(\hat a)$ by adding primes to the main diagonal positions that are primed in
$\QO(\hat a)$. 
Note that we have $\bPO(a) = \unprime(\bPO(\hat a))$ by Proposition~\ref{unprime-tab-prop}.
We will show that the entry in  position $(x,y)$ of $\bPO(\hat a)$ is 
primed if and only if  $\theta := \gamma_{xy}(\bPO(a))$ has
$\emptyset\neq \theta \in \cyc(z)$ and $ \tpi(a)(\theta) \in \marked(\hat a)$.
Define 
\[T^j := \bPO(\hat a_1\hat a_2\cdots \hat a_j)\quand b^j := \hat a_{j+1}\hat a_{j+2} \cdots \hat a_n\quad\text{for $0\leq j \leq n$,}\]
and abbreviate by writing $\marked(T^{j},b^{j}):=\marked(\row(T^{j})b^{j})$.
It suffices to check
 that \[ \marked(T^{j},b^{j}) = \left\{ \tpi_j(a)(\theta) : \theta \in  \marked(T^{j-1},b^{j-1})\right\}\text{ for all $j \in [n]$,}\]
 since this will imply that $\marked(\row(\bPO(\hat a))) =\left\{  \theta : \tpi(a)(\theta)\in \marked(\hat a)\right\}$.

Let $\sim$ be the transitive closure of the relation on primed involution words that has 
$\hat w \sim \ck_i(\hat w)$ for all $i \in \ZZ$ such that $\marked(\hat w) = \marked(\ck_i(\hat w))$.
In Lemma~\ref{cseq-lem}, 
if we are in case (a),  case (b), or case (c) with $\eta=\gamma_p$,
then $\tpi_j(a) = 1$ and it follows 
by tracing through the proof of Proposition~\ref{o-lem2} and using Proposition~\ref{marked-prop}
that $\row(T^{j-1}) b^{j-1} \sim \row(T^{j}) b^{j}$ as needed.

If we are in case (c) of Lemma~\ref{cseq-lem} with $\eta \neq \gamma_p$,
then $\tpi_j(a)$ is the transposition of $\cyc(z)$ 
interchanging $\eta \leftrightarrow \gamma_p$, and
it follows similarly that  $ \marked(T^{j},b^{j})$ is formed by 
applying this transposition to all elements of $\marked(T^{j-1},b^{j-1})$.

Finally, suppose we are in case (d) of Lemma~\ref{cseq-lem}, so that 
$\tpi_j(a) = (\gamma_{p} \leftrightarrow \gamma_{p+1})$.
Form $U^j$ from $T^j$ by switching the primes on the entries in positions $(p,p)$ and $(p+1,p+1)$.
Then, again following the proof of Proposition~\ref{o-lem2} and using Proposition~\ref{marked-prop}, one checks that 
$\row(T^{j-1}) b^{j-1} \sim \row(U^{j}) b^{j}$. 
Thus
$\marked(T^{j-1},b^{j-1}) = \marked(U^{j},b^{j}) = \left\{ \tpi_j(a)(\theta) : \theta \in  \marked(T^{j-1},b^{j-1})\right\}$ as desired.
\end{proof}

 As an application, we explain how to deduce Theorem~\ref{ck-fkd-thm} 
 in the case when inserting three consecutive letters in $\hat a$ contributes two diagonal positions to $\PO(\hat a)$.
   
\begin{lemma}\label{diag-box-lem}
Suppose
 $\hat a$ is a primed involution word and $n = \ell(\hat a)$.
Write $\square_j$ for $j\in[n]$ to denote the unique box of $\QO(\hat a)$ containing $j$ or $j'$.
Assume that $i\in [n-2]$ and $\square_i$ and $\square_{i+2}$ are both on the main diagonal.
Then $\PO(\ck_i(\hat a)) = \PO(\hat a)$ and $  \QO(\ck_i(\hat a)) = \fkd_i(\QO(\hat a)).$
\end{lemma}

\begin{proof}
Write $\square_i = (q-1,q-1)$ and $Q = \QO(\hat a)$. Then
we must 
have $\square_{i+1}=(q-1,q)$ and $\square_{i+2} = (q,q)$,
and consequently
 $   \fkd_i(Q) = \fks_i(Q) = \fks_{i+1}(Q)$ 
is formed from $Q$ by swapping $i+1$ and $i'+1$,
and then reversing the primes on the entries in the diagonal boxes $(q-1,q-1)$ and $(q,q)$
if these entries are not both primed or both unprimed.

After possibly invoking Proposition~\ref{square-lem} to interchange $Q$ with $\fkd_i(Q)$, 
we may assume that the entry in position $(q-1,q)$
of $Q$ is $i+1$ rather than $i'+1$.
Let $ \hat b = \ck_i(\hat a)$ and define $a = \unprime(\hat a)$ and $b = \unprime(\hat b)$. 
Then $b = \ck_i(a)$ by Lemma~\ref{unpri-word-lem}.
It is evident from Lemma~\ref{cseq-lem}
that $\tpi_i(a) = \tpi_{i+1}(a) = \tpi_{i+2}(a) = 1$.
Since
we know from Proposition~\ref{unprime-prop} that
 $\QO(b)$ is formed by applying $\fkd_i$ to $\QO(a) = \unprime_{\diag}(Q)$,
which adds a prime to position $(q-1,q)$, 
it is also clear from Lemma~\ref{cseq-lem}
that $\tpi_i(b)  = \tpi_{i+2}(b) = 1$.


To compute $\tpi_{i+1}(b)$,
we
consider the weak bumping paths
$\wpath_i(b)$, $\wpath_{i+1}(b)$, and $\wpath_{i+2}(b)$
 that result from inserting $b_{i}$, $b_{i+1}$, and $b_{i+2}$
successively into $\PO(a_1a_2\cdots a_{i-1}) = \PO(b_1b_2\cdots b_{i-1})$.
In view of Proposition~\ref{bumping-prop},
the first path must terminate at position $(q-1,q-1)$,
the last two positions of the second path must be $(q-1,q-1)$ followed by $(q-1,q)$,
and the last two positions of the third path must be $(q-1,q)$ followed by $(q,q)$.

If the first row of
$\cseq_{i+2}(a)$ is $\left[\barr{cccc} \gamma_1 &\dots&\gamma_q\earr\right]$,
then since $\cseq_{i+2}(a) = \cseq_{i+2}(b)$ by Proposition~\ref{unprime-prop},
we deduce from Lemma~\ref{gamma-lem} that
the first rows of $\cseq_{i-1}(a)=\cseq_{i-1}(b)$,
$\cseq_{i}(b)$, and $\cseq_{i+1}(b) $
are 
$\left[\barr{cccc} \gamma_1 & \dots&\gamma_{q-2}\earr\right]$,
$\left[\barr{ccccc} \gamma_1 & \dots&\gamma_{q-2}& \gamma_q\earr\right]$,
and
$\left[\barr{ccccc} \gamma_1  & \dots & \gamma_{q-2}& \gamma_{q-1}\earr\right]$,
respectively.
Thus $\tpi_{i+1}(b) $ is the permutation of $\cyc(z) $
that swaps $\gamma_{q-1}$ and $\gamma_q$.
Multiplying $\tpi_1(a)\tpi_2(a)\cdots \tpi_{i+2}(a)$
on the right by this permutation gives $\tpi_1(b)\tpi_2(b)\cdots \tpi_{i+2}(b)$
and vice versa.

As we know that $\PO( a) = \PO( b)$ and $  \QO(b) = \fkd_i(\QO(a))   $ by Proposition~\ref{unprime-prop},
it follows 
from Proposition~\ref{tau-prop} that $\PO(\hat a) = \PO(\hat b)$ and $  \QO(\hat b) = \fkd_i(\QO(\hat a))   $.
\end{proof}

\subsection{Constraints on cycle sequences and the $213\leftrightarrow 231$ case of Theorem~\ref{ck-fkd-thm}}\label{proof-sect3}

 The next few sections prove a series of technical results constraining the values of $\cseq_i(a)$ and $\tpi_i(a)$ for an (unprimed) involution word $a$.

In the following lemma, let $\entries(T)\subset \ZZ \sqcup\ZZ'$ denote the set of entries in a shifted tableau $T$.
Also let $\entriesdiag(T)$ denote the subset of entries  appearing on the main diagonal of $T$.

\begin{lemma}\label{bac-pre-lem}
Suppose $a$ and $b$ are (unprimed) involution words for elements of $I_\ZZ$. 
Fix  $0 \leq i \leq \ell(a)-2$ with $a_{i+1}  < a_{i+2}$ and suppose $0\leq j \leq \ell(b)-2$
is an index
such that the following holds:
\ben
\item[(a)] $\cseq_i(a)= \cseq_j(b)$ and $\cseq_{i+2}(a)= \cseq_{j+2}(b)$,
\item[(b)]  $ |\entriesdiag(\QO(a)) \cap \{i+1,i+2\}| = |\entriesdiag(\QO(b)) \cap \{j+1,j+2\}|$, and
\item[(c)] $ |\entries(\QO(a)) \cap \{i+1,i+2\}| =| \entries(\QO(b)) \cap \{j+1,j+2\}|$.
\een
 Then $\tpi_{i+1}(a)\tpi_{i+2}(a) = \tpi_{j+1}(b)\tpi_{j+2}(b)$ as permutations of $\binom{\ZZ}{2}$.
\end{lemma}


\begin{proof}
Let 
$s(a) := |\entriesdiag(\QO(a)) \cap \{i+1,i+2\}|\in \{0,1\}$ be the number of diagonal
entries in $\QO(a)$ equal to $i+1$ or $i+2$ and 
let
$r(a):=2 - |\entries(\QO(a)) \cap \{i+1,i+2\}|\in \{0,1,2\}$
 be the number of (necessarily off-diagonal) entries in $\QO(a)$ equal to $i'+1$ or $i'+2$.
Similarly let $s(b) \in \{0,1\}$ be the number of diagonal
entries in $\QO(b)$ equal to $j+1$ or $j+2$ and 
let
$r(b)\in \{0,1,2\}$
 be the number of entries in $\QO(a)$ equal to $j'+1$ or $j'+2$.

Conditions (b) and (c) imply that
$r(a)=r(b)$ and $s(a)=s(b)$.
The key idea in the proof of this lemma is to observe how this fact
 combined with 
Lemma~\ref{cseq-lem} 
 limits the possible values of  $\cseq_{i+1}(a)$
and 
 $\cseq_{j+1}(b)$
 once $\cseq_{i}(a) = \cseq_{j}(b)$
and $\cseq_{i+2}(a) = \cseq_{j+2}(b)$ are given.
We will then deduce that 
$\tpi_{i+1}(a)\tpi_{i+2}(a) = \tpi_{j+1}(b)\tpi_{j+2}(b)$
from these constraints. 

From now on set $r := r(a)=r(b)$ and $s:= s(a)=s(b)$.
The desired equality
 holds 
when $r=0$ since then $\tpi_{i+1}(a)=\tpi_{i+2}(a) = \tpi_{j+1}(b)=\tpi_{j+2}(b)=1$
by Lemma~\ref{cseq-lem}(a).

Assume $r=1$. Then, by Lemma~\ref{cseq-lem}(a), at least one of 
$\tpi_{i+1}(a)$ or $\tpi_{i+2}(a)$ is trivial, and likewise for $\tpi_{j+1}(b)$ or $\tpi_{j+2}(b)$.
Suppose further that $s=0$. Then 
$
\cseq_{i}(a) = \cseq_{j}(b)
$ and
$
\cseq_{i+2}(a) = \cseq_{j+2}(b)
$
have the same number of columns, so
we have  $
\cseq_{i}(a) = 
\cseq_{i+1}(a)$
or 
$
\cseq_{i+1}(a) = 
\cseq_{i+2}(a)$ (or both), as well as 
$
\cseq_{j}(b) = 
\cseq_{j+1}(b)$
or 
$
\cseq_{j+1}(b) = 
\cseq_{j+2}(b)$ (or both).
Write 
\be\label{write-these-eq}
\cseq_{i}(a)=\cseq_{j}(b) =   \left[\barr{llll} \gamma_1 &  \gamma_2  & \dots & \gamma_q
\\
c_1  & c_2 & \dots & c_q \earr\right]
\ee
and suppose the first row of $\cseq_{i+2}(a) =\cseq_{j+2}(b)$ is $\left[\barr{llll} \eta_1 & \eta_2  & \dots & \eta_q\earr\right]$.
If this is equal to the first row of $\cseq_{i}(a)=\cseq_{j}(b)$,
then  we must be in the ``or both'' case when
\[
\cseq_{i}(a) = 
\cseq_{i+1}(a)= 
\cseq_{i+2}(a)
\quand
\cseq_{j}(b) = 
\cseq_{j+1}(b)= 
\cseq_{j+2}(b),
\]
and then
$\tpi_{i+1}(a)=\tpi_{i+2}(a) = \tpi_{j+1}(b)=\tpi_{j+2}(b)=1$.
Otherwise, it follows by examining cases (c) and (d) in
Lemma~\ref{cseq-lem} 
that
there is either a unique index $p \in [q]$
with $\gamma_p \neq \eta_p$, or a unique  $p \in [q-1]$ with $\gamma_{p+1}=\eta_p \neq \gamma_p=\eta_{p+1}$,
and in either case $\tpi_{i+1}(a)\tpi_{i+2}(a) = \tpi_{j+1}(b)\tpi_{j+2}(b)$
is the permutation of $\binom{\ZZ}{2}$ swapping $\gamma_p$ and $\eta_p$.

Next suppose $r=s=1$.
Consider the weak bumping paths $\wpath_{i+1}(a)$ and $\wpath_{i+2}(a)$ that result from inserting $a_{i+1}$ and $a_{i+2}$
successively into $\PO(a_1a_2\cdots a_{i})$.
Since $a_{i+1} < a_{i+2}$, it follows from Proposition~\ref{bumping-prop2} that
 $\wpath_{i+2}(a)$ terminates at a diagonal position $(q+1,q+1)$
and $\wpath_{i+1}(a)$ contains a unique non-terminal diagonal position $(p,p)$ for some $p \in [q]$.
Denote $\cseq_{i}(a) =\cseq_{i}(b) $ as in \eqref{write-these-eq}.
There are four possibilities for $\cseq_{i+2}(a) = \cseq_{j+2}(b)$, namely:
\be\label{123-cseq-eq-a}
\ba
\left[\barr{llllll} \gamma_1  & \dots & \gamma_p &\dots & \gamma_q &\eta_{q+1}
\\
c_1  &\dots & c_p-1 &\dots & c_q & c_{q+1} \earr\right]
&\quord
\left[\barr{llllll} \gamma_1  & \dots & \eta_p &\dots & \gamma_q &\eta_{q+1}
\\
c_1  &\dots & d_p &\dots & c_q & c_{q+1} \earr\right]
\quord
\\
\left[\barr{lllllll} \gamma_1  & \dots & \gamma_{p+1} & \gamma_p&\dots & \gamma_q &\eta_{q+1}
\\
c_1  &\dots & c_p & c_{p+1} & \dots & c_q & c_{q+1} \earr\right]
&\quord
\left[\barr{llllll} \gamma_1  & \dots & \eta_p &\dots & \gamma_q &\gamma_{p}
\\
c_1  &\dots & d_p &\dots & c_q & c_{q+1} \earr\right],
\ea
\ee
where  $\eta_p,\eta_{q+1} \notin  \{\gamma_1,\gamma_2,\dots,\gamma_{q}\}$ and $d_p < c_p-1$.
 In each case, one can work out the unique possibility for $\cseq_{i+1}(a)$
 by examining cases (b), (c), and (d) in Lemma~\ref{cseq-lem}.

As we pass from $\cseq_{j}(b)$ to $\cseq_{j+1}(b)$ to $\cseq_{j+2}(b)$,
it follows from Lemma~\ref{cseq-lem} that
one step must add an extra column and the other must alter the first $q$ columns either by changing
a single column
or swapping adjacent entries in the first row.
From this observation, we deduce that if
$\cseq_{i+2}(a) = \cseq_{j+2}(b)$ has one of the first three forms in  \eqref{123-cseq-eq-a},  
 then
there are two possibilities for $\cseq_{j+1}(b)$, but in either case 
the factors 
$ \tpi_{j+1}(b)$ and $\tpi_{j+2}(b) $ commute and 
 $\tpi_{i+1}(a)\tpi_{i+2}(a) = \tpi_{j+1}(b)\tpi_{j+2}(b) $ is 
 respectively
 either
 the identity permutation, the transposition
   $(\gamma_p,\eta_{p})$, or the transposition $(\gamma_p,\gamma_{p+1})$.
If $\cseq_{i+2}(a) = \cseq_{j+2}(b)$ has the last form in \eqref{123-cseq-eq-a}
then 
\be\label{mmmust-eq}
\cseq_{i+1}(a) = \cseq_{j+1}(b)=
\left[\barr{llllll} \gamma_1  & \dots & \eta_p &\dots & \gamma_q 
\\
c_1  &\dots & d_p &\dots & c_q  \earr\right]
\ee
so $\tpi_{i+1}(a)=\tpi_{j+1}(b)$ and $\tpi_{i+2}(a) =\tpi_{j+2}(b) $.\footnote{
If $p=q$, then
Lemma~\ref{cseq-lem} with our assumptions that  
$\cseq_{i}(a) = \cseq_{j}(b)$ and $\cseq_{i+2}(a) = \cseq_{j+2}(b)$
does not uniquely determine the first row of $\cseq_{j+1}(b)$.
But considering  the arrays' second rows shows that \eqref{mmmust-eq}
must hold.}

Finally suppose $r=2$ so that $s=0$. 
Then $\cseq_{i}(a) = \cseq_{j}(b)$ and
$\cseq_{i+2}(a) = \cseq_{j+2}(b)$ have the same number of columns
but
$\cseq_{i}(a)\neq \cseq_{i+1}(a)  \neq \cseq_{i+2}(a) $ 
and
$\cseq_{j}(b)\neq \cseq_{j+1}(b)  \neq \cseq_{j+2}(b) $.
Denote $\cseq_{i}(a) = \cseq_{j}(b)$ as in \eqref{write-these-eq}
and 
consider the  weak bumping paths $\wpath_{i+1}(a)$ and $\wpath_{i+2}(a)$ that result from inserting $a_{i+1}$ and $a_{i+2}$
successively into $\PO(a_1a_2\cdots a_{i})$.
Both paths now must contain unique non-terminal diagonal positions $(k,k)$ and $(l,l)$,
and it follows from Proposition~\ref{bumping-prop2} that
 $k<l$ since we assume $a_{i+1} < a_{i+2}$.
We may therefore list the possibilities for 
 $\cseq_{i+2}(a) = \cseq_{j+2}(b)$
 as follows. To start, this array could be
\begin{itemize}
\item[(1)]
$\left[\barr{lllllll} \gamma_1  & \dots & \eta_k & \dots & \eta_{l} &\dots  &\gamma_{q}
\\
c_1  &\dots & d_k & \dots & d_{l}  &\dots & c_q  \earr\right]
$ or
$\left[\barr{llllllll} \gamma_1  & \dots & \eta_k & \dots & \gamma_{l+1} & \gamma_{l} &\dots  &\gamma_{q}
\\
c_1  &\dots & d_k & \dots & c_{l} & c_{l+1} &\dots & c_q  \earr\right]$, 
\end{itemize}
where in these cases for each $p\in\{k,l\}$  either $d_p =  c_p-1$ and $\eta_p=\gamma_p$
or $d_p < c_p-1$ and $\eta_p \notin \{\gamma_1,\gamma_2,\dots,\gamma_q\}$.
When $k+1<l$, the array could also be
\begin{itemize}
\item[(2)] {\footnotesize $\left[\barr{llllllll} \gamma_1  & \dots & \gamma_{k+1} & \gamma_k & \dots & \eta_{l} &\dots  &\gamma_{q}
\\
c_1  &\dots & c_k & c_{k+1}  & \dots & d_l &\dots & c_q  \earr\right]$} or
{\footnotesize
$\left[\barr{lllllllll} \gamma_1  & \dots & \gamma_{k+1} & \gamma_k & \dots & \gamma_{l+1} & \gamma_{l} &\dots  &\gamma_{q}
\\
c_1  &\dots & c_k & c_{k+1}  & \dots & c_l & c_{l+1}  &\dots & c_q  \earr\right]$},
\end{itemize}
where again either $d_l =  c_l-1$ and $\eta_l=\gamma_l$
or $d_l < c_l-1$ and $\eta_l \notin \{\gamma_1,\gamma_2,\dots,\gamma_q\}$.
Finally, if $k+1=l$ then  $\cseq_{i+2}(a) = \cseq_{j+2}(b)$ could also be either
\begin{itemize}
\item[(3)] $\left[\barr{llllll} \gamma_1  & \dots & \gamma_{k+1} & \gamma_{k} &\dots  &\gamma_{q}
\\
c_1  &\dots & c_k & c_{k+1}-1  &\dots & c_q  \earr\right]
$, or
\item[(4)]
$\left[\barr{lllllll} \gamma_1  & \dots & \gamma_{k+1} & \gamma_{k+2} & \gamma_{k} &\dots  &\gamma_{q}
\\
c_1  &\dots & c_k & c_{k+1} &c_{k+2}  &\dots & c_q  \earr\right]
,$ or
\item[(5)] $\left[\barr{llllll} \gamma_1  & \dots & \gamma_{k+1} & \eta_{k} &\dots  &\gamma_{q}
\\
c_1  &\dots & c_k & d_{k+1}  &\dots & c_q  \earr\right]
$,
\end{itemize}
where $d_{k+1} < c_{k+1}-1$ and $\eta_k \notin \{\gamma_1,\gamma_2,\dots,\gamma_q\}$, or
the array could be
\begin{itemize}
\item[(6)] 
$\left[\barr{llllll} \gamma_1  & \dots & \eta_k & \gamma_{k} &\dots  &\gamma_{q}
\\
c_1  &\dots & d_k & d_{k+1}  &\dots & c_q  \earr\right]
,$
\end{itemize}
where  $d_k <  c_k-1$ and $\eta_k \notin \{\gamma_1,\gamma_2,\dots,\gamma_q\}$
and $d_{k+1} < c_{k+1}-1$.
In each case, one can again work out the unique possibility 
for $\cseq_{i+1}(a)$ 
by examining cases (c) and (d) in Lemma~\ref{cseq-lem}.

The values for $\cseq_{j+1}(b)$
 are constrained by Lemma~\ref{cseq-lem}
 and the fact that   $\cseq_{i}(a)=\cseq_{j}(b)\neq \cseq_{j+1}(b)  \neq \cseq_{j+2}(b)  =\cseq_{i+2}(a)$.
In cases (1)-(3) there are two possibilities for $\cseq_{j+1}(b)$ 
but for either one 
$ \tpi_{j+1}(b)$ and $\tpi_{j+2}(b) $ commute and 
 $\tpi_{i+1}(a)\tpi_{i+2}(a) = \tpi_{j+1}(b)\tpi_{j+2}(b) $.
 In case (4), we must have 
 \[
 \cseq_{i+1}(a) = \cseq_{j+1}(b)=
 \left[\barr{lllllll} \gamma_1  & \dots & \gamma_{k+1} & \gamma_{k} & \gamma_{k+2} &\dots  &\gamma_{q}
\\
c_1  &\dots & c_k & c_{k+1} &c_{k+2}  &\dots & c_q  \earr\right].
\]
In case (5), we must have 
\[
\cseq_{i+1}(a) = \cseq_{j+1}(b)=
\left[\barr{llllll} \gamma_1  & \dots & \gamma_{k+1} & \gamma_{k} &\dots  &\gamma_{q}
\\
c_1  &\dots & c_k & c_{k+1}  &\dots & c_q  \earr\right].
\]
In case (6), we must have 
\[
\cseq_{i+1}(a) = \cseq_{j+1}(b)=
\left[\barr{llllll} \gamma_1  & \dots & \eta_k & \gamma_{k+1} &\dots  &\gamma_{q}
\\
c_1  &\dots & d_k & c_{k+1}  &\dots & c_q  \earr\right].
\]
In each situation   $\tpi_{i+1}(a)=\tpi_{j+1}(b)$ and $\tpi_{i+2}(a) =\tpi_{j+2}(b) $,
so
 $\tpi_{i+1}(a)\tpi_{i+2}(a) = \tpi_{j+1}(b)\tpi_{j+2}(b) $.
\end{proof}

The action of $\ck_i$ comes in three different forms: either 
$\ck_i$ transforms a ``$213$-pattern'' to a ``$231$-pattern'',
a ``$121$-pattern'' to a ``$212$-pattern'',
or a ``$132$-pattern'' to a ``$312$-pattern''.
 We can use the lemmas in this section
to derive the following result. This lemma, combined with Proposition~\ref{tau-prop}, will be used to prove Theorem~\ref{ck-fkd-thm}
 when $\ck_i$ acts a $213 \leftrightarrow 231$ transformation.

\begin{lemma}\label{bac-lem}
Suppose $a=a_1a_2\cdots a_n$ is an (unprimed) involution word for an element of $I_\ZZ$.
 Assume   $i\in[n-2]$ and  $ a_{i+1} <  a_{i}  <  a_{i+2}$. Then 
 $\tpi(\ck_i(a)) = \tpi(a)$.
\end{lemma}

 
\begin{proof}
Let $b := \ck_i(a) = a_1 \cdots a_i a_{i+2} a_{i+1}\cdots a_n$.
We wish to prove that $\tpi(a) = \tpi(b)$.
Write $\square_j$ for $j\in[n]$ to denote the box of $\QO(a)$ containing $j$ or $j'$.
We first check that $\square_i$ and $\square_{i+2}$ are not both on the main diagonal.
Arguing by contradiction, we observe that
these positions could only both be on the  diagonal if the weak bumping paths $\wpath_i(a)$, $\wpath_{i+1}(a)$, and $\wpath_{i+2}(a)$ that result from inserting $a_{i}$, $a_{i+1}$, and $a_{i+2}$
successively into $\PO(a_1a_2\cdots a_{i-1})$ respectively terminate at $(q-1,q-1)$, $(q-1,q)$, and $(q,q)$ for some $q>0$.
Assume this is the case,
so that we have $\wpath_{i}(a) =\rwpath_{i}(a)$ and $\wpath_{i+1}(a) =\rwpath_{i+2}(a)$.

Since $a_i > a_{i+1}$,
Proposition~\ref{bumping-prop2} implies
 that the positions in $\rwpath_{i+1}(a)$ 
are all weakly to the left of the corresponding positions in $\rwpath_i(a)$. The second to last position in $\wpath_{i+1}(a)$ must therefore be $(q-1,q-1)$, so the entry in position $(q-1,q)$ of $\PO(a_1a_2\cdots a_{i+1})$
is the same as the entry in position $(q-1,q-1)$ of $\PO(a_1a_2\cdots a_{i})$. Since $a_{i+1} < a_i < a_{i+2}$, it is  easy to check 
that the first $q-1$ positions in $\wpath_{i+2}(a)$
are strictly to the right of the corresponding positions in $\wpath_i(a)$, and that if $\wpath_{i+2}(a)$ reaches row $q$ then its position in that row 
must be strictly to the right of $(q-1,q)$. But this makes it impossible for $\wpath_{i+2}(a)$ to terminate at $(q,q)$.

Thus $\square_i$ and $\square_{i+2}$ are not both on the main diagonal.
By Proposition~\ref{unprime-prop}
$\PO( a_1a_2 \cdots  a_{j}) = \PO( b_1 b_2\cdots b_{j})$
for all $j \in [n]\setminus\{i+1\}$
along with $\QO(b) = \fkd_i(\QO(a))$, so $\tpi_j(a) = \tpi_j(b)$ for all $j \in [n] \setminus \{i + 1, i+2\}$.
It remains
to show that $\tpi_{i+1}(a)\tpi_{i+2}(a) = \tpi_{i+1}(b)\tpi_{i+2}(b)$.
Evidently $\cseq_{i}(a) = \cseq_{i}(b)$
and $\cseq_{i+2}(a) = \cseq_{i+2}(b)$ and $a_{i+1} < a_{i+2}$.
Since $\QO(b) = \fkd_i(\QO(a))$
and $\square_i$ and $\square_{i+2}$ are not both on the main diagonal,
it follows from Proposition~\ref{square-lem}
that conditions (b) and (c) in Lemma~\ref{bac-pre-lem} also  hold, so that result implies that
  $\tpi_{i+1}(a)\tpi_{i+2}(a) = \tpi_{i+1}(b)\tpi_{i+2}(b)$.
\end{proof}

\subsection{Constrains from intersecting and non-intersecting bumping paths}\label{inter-sect}

This section contains two technical lemmas that
constrain how $\cseq_i(a)$ and $\QO(a)$ can change when adjacent letters are swapped and 
the successive bumping paths associated to these letters either intersect or remain disjoint.

\begin{lemma}\label{bump-a-lem}
Let $a,b,c$ be unprimed words with $n:=\ell(a)$. Suppose $X,Y\in\ZZ$
are such that 
\ben
\item[(a)] $XYb$ and $YXc$ are reduced words for the same permutation in $S_\ZZ$, and
\item[(b)] $aXYb$ and $aYXc$ are involution words (necessarily for the same element  in $I_\ZZ$).
\een
Let $T := \PO(a)$. 
If $\rspath(T,X) \cap \rspath(T,Y)$ is nonempty then its first position is also in $\rwpath(T,X) \cap \rwpath(T,Y)$.
If $\rspath(T,X) \cap \rspath(T,Y)$
has an off-diagonal position
then
\bei
\item $\cseq_{n+1}(aXYb)=\cseq_{n+1}(aYXc)$;
\item $n+1$ is on the diagonal in $\QO(aXYb)$
if and only if $n+1$ is on the diagonal in $\QO(aYXc)$;
\item $n'+1$ is in $\QO(aXYb)$
if and only if
$n'+1$ is in $\QO(aYXc)$.
\eei
\end{lemma}

\begin{proof}
Suppose $\rspath(T,X) \cap \rspath(T,Y)$ is nonempty and the first position in this intersection
is $(j,k)$.
To show that $(j,k) $ also belongs to $\rwpath(T,X) \cap \rwpath(T,Y)$,
 write $X_0:=X<Y_0:=Y$
and let $X_i$ and $Y_i$ be the entries of $T$ in the $i$th positions of 
$\spath(T,X) $ and $ \spath(T,Y)$ respectively.
Then $X_{j-1} < Y_{j-1}$ and the smallest entry of $T$ in row $j$ that
is greater than both of these numbers is $X_j=Y_j$ by definition.
This means that row $j$ of $T$ cannot contain any entry $w$ with $X_{j-1} < w \leq Y_{j-1}$,
 so by Remark~\ref{iarrow-rmk}, row $j$ of $T$ also cannot contain $X_{j-1}$.
Hence $(j,k)\in \rwpath(T,X) \cap \rwpath(T,Y)$ as desired.

It is clear from Definition~\ref{iarrow-def}
that $\rspath(T,X)$ and $\rspath(T,Y)$ coincide after their first $j-1$ positions, 
and it follows by our claim that $\rwpath(T,X)$ and $\rwpath(T,Y)$ also coincide after their first $j-1$ positions.
If $j\neq k$, then all of these paths continue after row $j$,
and
we have $\gamma_{xy}(T,XYb) = \gamma_{xy}(T,YXc)$ for all positions $(x,y)$
since $XYb$ and $YXc$ are reduced words for the same permutation.
Given these observations,
the result follows from Lemma~\ref{cseq-lem}.
\end{proof}

The next lemma gives us precise control over 
cycle sequences and diagonal entries when swapping adjacent letters in an involution word
that are ``far apart'' and have disjoint bumping paths.

\begin{lemma}\label{bump-b-lem}
Suppose $a,b$ are unprimed words and $X,Y\in\ZZ$
are such that $X+1<Y$ and $aXYb$ is an involution word for an element of $I_\ZZ$.
Let $T = \PO(a)$ and $n=\ell(a)$,
and assume $\rwpath(T,X)$ and $\rwpath(T,Y)$ are disjoint.
Then $\cseq_{n+2}(aXYb)=\cseq_{n+2}(aYXb)$, and for each
$\epsilon \in\{0,1\}$, the number
$n+1+\epsilon$ is on the main diagonal in $\QO(aXYb)$
if and only if $n+2-\epsilon$ is on the main diagonal in $\QO(aYXb)$,
while $n'+1+\epsilon$ is in $\QO(aXYb)$
if and only if
$n'+2-\epsilon$ is in $\QO(aYXb)$.
\end{lemma}


\begin{proof}
Again write $X_0:=X<Y_0:=Y$
and let $X_i$ and $Y_i$ be the entries of $T$ in the $i$th positions of 
$\spath(T,X) $ and $ \spath(T,Y)$ respectively.
Suppose $\rwpath(T,X) $ and $ \rwpath(T,Y)$ are disjoint.
Lemma~\ref{bump-a-lem} with $b=c$ implies that
$\rspath(T,X) $ and $ \rspath(T,Y)$ must also be disjoint.
We argue that since $X+1<Y$, it must further hold that $\rspath(T,X) $ and $\rwpath(T,Y) $
are disjoint. To see this, note that if $X_i=Y_i-1$ in some row $i> 0$ of $T$ occupied by both
$\rspath(T,X) $ and $ \rspath(T,Y)$, then this row of $T$ must also contain $X_i-1$
and we must have $X_{i-1} = X_i-1 $ and $Y_{i-1} = X_i$,
since otherwise $\rwpath(T,X) $ and $ \rwpath(T,Y)$ would intersect in the position of $X_i$ in row $i$.
But this means that if $X_i= Y_i-1$ for any row $i>0$ then we also have
$X_0 =X_0-1$, which is a contradiction since $X_0=X$ and $Y_0=Y$.


From these properties, we deduce that in any given row occupied by 
all four paths, the position in 
$\rwpath(T,X) $  is weakly to the left of the position in $ \rspath(T,X)$,
which is strictly to the left of the position in $ \rwpath(T,Y)$,
which finally is weakly to the left of the position in $ \rspath(T,Y)$.
It follows that if $(i,i) \in \rwpath(T,X) \cap \rspath(T,X)$
then any diagonal position $(j,j) \in \rwpath(T,Y)$ must have $i<j$,
while if $(i,i) \in \rwpath(T,X) $ and $(i,i+1)\in \rspath(T,X)$
then any diagonal position $(j,j) \in \rwpath(T,Y)$ must have $i+1<j$.


In addition, $T_{xy}$ and $\gamma_{xy}(T)$
only differ from $(T\iarrow w)_{xy}$
and $\gamma_{xy}(T\iarrow w)$
 at positions $(x,y) \in \wpath(T,w)\cup \spath(T,w)$ by Lemma~\ref{gamma-lem}.
Since $\gamma_{n+1}(aXYb) = \gamma_{n+2}(aYXb)$ 
 and
 $\gamma_{n+1}(aYXb) = \gamma_{n+2}(aXYb)$ 
 as $X+1< Y$,
 it follows in view of Proposition~\ref{bumping-prop} that 
 \be\label{dddbbb-eq} \dbump_{n+1}(aXYb) = \dbump_{n+2}(aYXb).\ee
To prove the lemma, it suffices to show that
 $\dbump_{n+1}(aYXb)$ and $\dbump_{n+2}(aXYb)$
 end with the same tuple, or that $\rwpath(T,Y)$ and $\rwpath(T\iarrow u,Y)$ both never reach the main diagonal.
 In the former situation 
Lemma~\ref{cseq-lem} implies the desired result.
In the latter situation
Lemma~\ref{cseq-lem} implies  
\[\cseq_n(aXYb)=\cseq_n(aYXb)=\cseq_{n+1}(aYXb),\]
which means that
$\cseq_{n+1}(aXYb)=\cseq_{n+2}(aYXb)$ in view of \eqref{dddbbb-eq},
along with \[\cseq_{n+1}(aXYb) =\cseq_{n+2}(aXYb) ,\]
so
$\cseq_{n+2}(aXYb)=\cseq_{n+2}(aYXb)$ holds.
The other assertions about the locations of $n+1$, $n+2$, $n'+1$, and $n'+2$
in $\QO(aXYb)$ and $\QO(aYXb)$ are easy to deduce from Lemma~\ref{cseq-lem}.

To this end, recall the definitions of $ \cwpath(T,X) $ and $ \cspath(T,X) $ from \eqref{cpath-eq}.
If the positions in $\cwpath(T,X) \cup \cspath(T,X)$
are disjoint from $\rwpath(T,Y) \cup \rspath(T,Y)$,
then the latter union is disjoint from
 $\wpath(T,X) \cup \spath(T,X)$,
and so the stronger property 
$\dbump_{n+1}(aYXb) = \dbump_{n+2}(aXYb)$
holds in view of Lemma~\ref{gamma-lem}.

Instead suppose
$\cwpath(T,X)\cup \cspath(T,X)$ and $\rwpath(T,Y)\cup \rspath(T,Y)$ are not disjoint.
For each $i>0$, let $\cspath(T,X,i)$ be the set of positions in $\cspath(T,X)$ in row $i$,
and let 
\[ \cwpath(T,X,i) :=\{ (i-1,j) \in \cwpath(T,X) :
(i,j) \in \cspath(T,X)\}.\]
Then each position in $\cwpath(T,X)\cup \cspath(T,X)$ belongs to $\cwpath(T,X,i)\cup \cspath(T,X,i)$
for a unique value of $i$,
and every position in $\cwpath(T,X,i)\cup \cspath(T,X,i)$ occurs in a column strictly to the left of
every position in $\cwpath(T,X,i+1)\cup \cspath(T,X,i+1)$ by Proposition~\ref{bumping-prop}.

Let $i$ be minimal such that 
$\cwpath(T,X,i)\cup \cspath(T,X,i)$ and $\rwpath(T,Y)\cup \rspath(T,Y)$ intersect.
Assume the leftmost position in $\cwpath(T,X,i) \cup \cspath(T,X,i)$ is in column $j+1$ while
\[ |\cwpath(T,X,i)| = l
\quand |\cspath(T,X,i)| = k + l\] for some integers $k,l \geq0$ with $k+l>0$.
If $i=1$ then we must have $l=0$ and $j+k-1$ must be the length of the first row of $T$.
If $i>1$ then we must have $Y_{j+k+t}= Y_{j+k}+t$ for $t \in[l]$.
Finally, all positions in $\cwpath(T,X,i) \cup \cspath(T,X,i)$ must be occupied in $T$,
except that when $l=0$
the single position $(i,j+k) $ may be outside the domain of $ T$.

First assume all positions in $\cwpath(T,X,i) \cup \cspath(T,X,i)$ are occupied in $T$.
Then we must have 
 $i>1$, so the entries of $T$ in positions $\{i-1,i\}\times \{j+1,j+2,\dots,j+k+l\}$ are
\[{\scriptsize
\ytableausetup{boxsize = 1.7cm,aligntableaux=center}
\begin{ytableau}
 X_{j+1} & X_{j+2} & \cdots & X_{j+k} &  X_{j+k}+1 &  X_{j+k}+2 & \cdots &  X_{j+k}+l \\
 T_{i-1,j+1}  &   T_{i-1,j+2}& \cdots & T_{i-1,j+k}  &   X_{j+k}&  X_{j+k}+1 & \cdots &  X_{j+k}+l-1 
\end{ytableau}}
\]
while the corresponding entries of $T\iarrow u$ are\footnote{Most of the boxes labeled by question marks in $T\iarrow X$ contain the same entries
as the corresponding positions of $T$. Such an entry could be different
if its position belongs to $\rwpath(T,X)\cap \rspath(T,X)$.
A given row has at most one such position,
which must be strictly to the left of any terms of $\rwpath(T,Y)$ in the same row.}
\[{\scriptsize
\ytableausetup{boxsize = 1.7cm,aligntableaux=center}
\begin{ytableau}
 X_{j} & X_{j+1} & \cdots & X_{j+k-1} &  X_{j+k}+1 &  X_{j+k}+2 & \cdots &  X_{j+k}+l \\
 ?  &   ?& \cdots & ?  &   X_{j+k}&  X_{j+k}+1 & \cdots &  X_{j+k}+l-1 
\end{ytableau}}.
\]
In this case
one of the following holds:
\ben
\item[(1)] $i=j$ and $T_{ii} = X_j$, 
\item[(2)] $i=j+1$ and $k=0$ and $T_{i-1,i-1} + 1 = T_{i-1,i} =T_{ii}-1 = X_j$,
or 
\item[(3)] $i < j$ and $X_j$ appears in column $j$ of $T$ above row $i$.
 \een
Position $(i-1,j+k+l+1)$ in $T$ must be
unoccupied or contain an entry greater than $X_{j+k}+l$, so
position $(i,j+k+l+1)$ in $T$ 
is unoccupied or contains an entry greater than $X_{j+k}+l+1$.
This implies that 
neither $(i-1,j+k+l)$ nor $(i,j+k+l)$ can belong to $\rwpath(T,Y) \setminus \rspath(T,Y)$.
Therefore
if $(x,y)$ is in the intersection of $\rwpath(T,Y) $ and 
$\cwpath(T,X,i)\cup \cspath(T,X,i)$
then $(x,y)$ or $(x,y+1)$ must be in the intersection of 
$\rspath(T,Y) $ and 
$\cwpath(T,X,i)\cup \cspath(T,X,i)$.
Furthermore, if $(i-1,y) \in \rspath(T,Y) \cap \cwpath(T,X,i)$
then
$(i,y) \in  \rspath(T,Y) \cap \cspath(T,X,i)$.


So we may assume that $(i,j+\delta) \in  \rspath(T,Y) \cap \cspath(T,X,i)$
for some $ \delta \in [k+l]$.
If $ k < \delta \leq l$ then we also have 
 $(i-1,j+\delta) \in  \rspath(T,Y) \cap \cwpath(T,X,i)$.
 In view of the minimality of $i$,
 apart from these one or two positions there are no other elements in the intersection
 of $\rspath(T,Y)$ and $\cwpath(T,X) \cup \cspath(T,X)$,
 since $\rspath(T,Y)$
 contains at most one position in each row,   
 and since all positions  
 of $\rspath(T,Y)$ above row $i$ contain entries of $T$ that are greater than $X_{j+\delta}$
 while all positions $\cwpath(T,X) \cup \cspath(T,X)$ above row $i$ contain entries of $T$
 that are at most $X_j$.
To proceed, we divide our analysis into six subcases:
 \ben
\item[(a)] If $k+1 < \delta\leq l$ then  Lemma~\ref{gamma-lem}
implies 
$\dbump_{n+1}(aYXb) = \dbump_{n+2}(aXYb)$ which suffices.
%

\item[(b)] Suppose $k>0$ and $\delta = k+1$, so that $l>0$ while $(i-1,j+k+1)$ and $(i,j+k+1)$ are both in $\rspath(T,Y)$.
We cannot have $T_{i-1,j+k} = X_{j+k} - 1$,
since then $(i-1,j+k)$ would be in $\rwpath(T,Y)$ and not $\rspath(T,X)$,
meaning that $(i-1,j+k)$ would have to belong to $\cwpath(T,X,i)$.
Therefore $(i-1,j+k+1)$ is also in $\rwpath(T,Y)$.
This means that terms $i$ and $i+1$ of $\dbump_{n+1}(aYXb)$ are 
\[ (j+k,j+k+1,X_{j+k}, \theta)\quand  (y,\tilde y , X_{j+k}+1, \theta)\]
for the 2-cycle $\theta := \gamma_{i-1,j+k+1}(T,XYb) = \gamma_{i-1,j+k+1}(T,YXb)$ and some columns 
$y\leq \tilde y\leq j + k +1$.
By Lemma~\ref{gamma-lem}, terms $i$ and $i+1$ of $\dbump_{n+2}(aXYb)$ are 
\[ (j+k+1,j+k+1,X_{j+k}, \eta)\quand  (y,\tilde y , X_{j+k}+1, \theta)\] 
for $\eta := \gamma_{i,j+k+1}(T,XYb) = \gamma_{i,j+k+1}(T,YXb)$ and
the same values of $\theta$, $y$, $\tilde y$.
Thus $\dbump_{n+1}(aYXb)$ and $\dbump_{n+2}(aXYb)$ only differ in their $i$th terms, 
so their final terms coincide as needed.


\item[(c)]
Suppose $k=0$ and $\delta=k+1=1$, so that again $l>0$. Then cases (1) and (2) 
would each lead to a contradiction of our assumption that $\rwpath(T,X) \cap \rwpath(T,Y)$ is empty:
case (1) would imply that this intersection contains $(i,i)$ while case (2)
could imply that it contains $(i-1,i-1)$.
Therefore we are in case (3) so 
position $(i,j)$ in $T$ contains an entry that is at most
$X_j-1$ while position $(i+1,j)$   in $T$ contains an entry that is at most
$X_j$.
It follows that terms $i$ and $i+1$ of $\dbump_{n+1}(aYXb)$ have the form 
\[ (j+1,j+1,X_{j}, \theta)\quand  (j+1,j+1 , X_{j}+1, \eta)\]
while terms $i$ and $i+1$ of $\dbump_{n+2}(aXYb)$ have the form
\[ (j+1,j+1,X_{j}, \eta)\quand  (j+1,j+1 , X_{j}+1, \theta)\] 
for 
$\theta := \gamma_{i-1,j+1}(T,XYb) = \gamma_{i-1,j+1}(T,YXb)$
and
$\eta := \gamma_{i,j+1}(T,XYb) = \gamma_{i,j+1}(T,YXb)$.
As in the previous paragraph,
it follows that $\dbump_{n+1}(aYXb)$ and $\dbump_{n+2}(aXYb)$ do not differ outside these two terms,
so either both sequences end in the same tuple
in view of \eqref{succ-theta-eq}
 or
 $\rwpath(T,Y)$ and $\rwpath(T\iarrow X,Y)$ never reach the main diagonal 
since $(i+1,j+1)$ is not a diagonal position.
This is again sufficient to conclude that the lemma holds.

\item[(d)] 
The case
$\delta = k>1$ cannot occur, as in this event, it would follow in view of Proposition~\ref{bumping-prop}
 that
  $(i-1,j+k)$ and $(i,j+k)$ are both in $\rspath(T,Y)$ with $X_{j+k-1} \leq T_{i-1,j+k}< X_{j+k}$, 
  which contradicts the fact that  $T_{i-1,j+k} < X_{j+k-1}$
as $(i-1,j+k) \notin \rspath(T,X)$.
%

\item[(e)] 
If $k>0$ and $1 < \delta < k$,
then it follows from Lemma~\ref{gamma-lem} that 
$\dbump_{n+1}(aYXb)$ and $\dbump_{n+2}(aXYb)$
differ only in their $i$th term, 
where if this term of $\dbump_{n+1}(aYXb)$
is $(y,\tilde y,d,\eta)$ then the corresponding term of 
$\dbump_{n+2}(aXYb)$ is $(1+y,1+\tilde y,d,\eta)$.
Both sequences then have more than $i$ terms so 
they  end with the same tuple as needed.


\item[(f)]
Next suppose $k>0$ and $\delta=1$. If $X_j < Y_{i-1}$ then the
argument in subcase (e) still applies.

Assume $Y_{i-1} \leq X_j$. Then we cannot be in cases (1) or (2) without contradicting 
$\rwpath(T,X) \cap \rwpath(T,Y)=\varnothing$,
so $X_j$ appears in column $j$ of $T$ above row $i$ and
position $(i+1,j)$   in $T$ contains an entry that is at most
$X_j$. 
The entry in position $(i,j)$ of $T$ cannot be greater than
$Y_{i-1}$ since $(i,j+1) \in \rspath(T,Y)$,
and this entry must also not be equal to $Y_{i-1}$ since then we would have $ X_{j+1} = Y_{i-1} + 1$
which can only hold if $X_j = Y_{i-1}$, in which case column $j$ of $T$ would have two equal entries,
contradicting the fact that all columns of $T$ are strictly increasing.
Thus position $(i,j)$   in $T$ contains an entry that is less than $Y_{j-1}$.

It follows that 
$\dbump_{n+1}(aYXb)$ and $\dbump_{n+2}(aXYb)$ only differ in terms $i$ and $i+1$:
while these terms in $\dbump_{n+1}(aYXb)$ must have the form
$ (j+1,j+1,Y_{i-1},\theta) $ and $ (j+1,j+1,X_{j+1},\eta)$
 for some 2-cycles $\theta$ and $\eta$, the corresponding terms of $\dbump_{n+2}(aXYb)$ are  
\[ (j+1,j+2,Y_{i-1}, \theta)\quand (j+1,j+1,X_{j+1},\theta)\]
when $Y_{i-1} =X_j$, or 
\[ (j+1,j+1,Y_{i-1}, \theta)\quand (j+1,j+1,X_{j},\phi)\]
when $Y_{i-1} < X_j$, where we may have $\phi \neq \eta$.
As in our earlier cases, we conclude  
that either both sequences end in the same tuple in view of \eqref{succ-theta-eq},
or we observe that
$(i+1,j+1)$ is not a diagonal position 
so  $\rwpath(T,Y)$ and $\rwpath(T\iarrow X,Y)$ never reach the main diagonal.
\een
This completes our argument if all positions in $\cwpath(T,X,i) \cup \cspath(T,X,i)$ are occupied in $T$.

When this does not occur, we must have $l=0$ and $(i,j+k)\notin T$.
In this case row $i$ of $T$ is 
\[{\scriptsize
\ytableausetup{boxsize = 1.1cm,aligntableaux=center}
\begin{ytableau}
T_{i1} &T_{i2}& \cdots & T_{ij} &
 X_{j+1} & X_{j+2} & \cdots & X_{j+k-1}& \none \end{ytableau}} 
\]
while row $i$ of $T\iarrow X$ is
\[{\scriptsize
\ytableausetup{boxsize = 1.1cm,aligntableaux=center}
\begin{ytableau}
T_{i1} &T_{i2} & \cdots & T_{ij} &
 X_j & X_{j+1}  & \cdots & X_{j+k-2} & X_{j+k-1} \end{ytableau}}.
\]
Here, cases (1) or (3) from above must apply.
We cannot have $(i, j+k) \in \rwpath(T,Y) \setminus \rspath(T,Y)$ if $(i,j+k) \notin T$,
so  again $(i,j+\delta) \in  \rspath(T,Y) \cap \cspath(T,X,i)$
for some $ \delta \in [k]$.
By the minimality of $i$, this position is the unique element in both 
 $ \rspath(T,Y)$ and
$\cwpath(T,X)\cup \cspath(T,X)$,
since  $ \rspath(T,Y)$ contains at most one position in each row, and 
since all positions of $ \rspath(T,Y)$ above row $i$ contain entries greater than $X_{j+\delta}$ while
 all positions of $\rwpath(T,Y)\cup \rspath(T,Y)$ above row $i$ contain entries that are at most $X_j$.
 We are left with two further subcases:

\ben
\item[(g)]
If  $X_j < Y_{i-1}$, then 
it follows from Lemma~\ref{gamma-lem} as in subcase (e) that
$\dbump_{n+1}(aYXb)$ and $\dbump_{n+2}(aXYb)$
differ only in their $i$th term, 
where if this term of $\dbump_{n+1}(aYXb)$
is $(y,\tilde y,d,\eta)$ then the corresponding term of 
$\dbump_{n+2}(aXYb)$ is $(1+y,1+\tilde y,d,\eta)$.
In this event, both sequences have more than $i$ terms unless $y=\tilde y = j + k$.
Since $(j,j+k)$ is not a diagonal position, we conclude that the lemma holds holds either way.

 
\item[(h)] 
Assume $Y_{i-1} \leq X_j$. Then we cannot be in case (1) without contradicting 
$\rwpath(T,X) \cap \rwpath(T,Y)=\varnothing$,
so $i<j$ and $X_j$ appears in column $j$ of $T$ above row $i$.
If $\delta < k$ then we can repeat the argument given in subcase (f)
to deduce our result.
If $\delta=k$
then we must have $k=1$ and $Y_{i-1} < X_j$.
In this situation, 
$\dbump_{n+1}(aYXb)$ has only $i$ terms
and ends with a term  of the form $(j+1,j+1,Y_{i-1},\theta)$
for some 2-cycle $\theta$,
and   $\dbump_{n+2}(aXYb)$ is formed from $\dbump_{n+1}(aYXb)$ by appending the tuple $(j+1,j+1,X_j,\phi)$
for some 2-cycle $\phi$.
Since neither $(i,j+1)$ nor $(i+1,j+1)$ is a diagonal position,
this shows that $\rwpath(T,Y)$ and $\rwpath(T\iarrow X,Y)$ never reach the main diagonal so the lemma again holds.
\een
This completes our proof of the lemma.
\end{proof}

\subsection{The $121\leftrightarrow 212$ and $132\leftrightarrow 312$ cases of Theorem~\ref{ck-fkd-thm}}\label{tech-sect}

In this section we prove one final lemma to help prove Theorem~\ref{ck-fkd-thm}
in the case when $\ck_i$ acts by transforming a ``$121$-pattern'' to a ``$212$-pattern''
or a ``$132$-pattern'' to a ``$312$-pattern''. This is our most technical result; it is the main application of the lemmas 
in the previous section.

\begin{lemma}\label{acb-lem} Suppose $a=a_1a_2\cdots a_n$ is an (unprimed) involution word for an element of $ I_\ZZ$.
Write $\square_j$ for $j\in[n]$ to denote the box of $\QO(a)$ containing $j$ or $j'$.
Suppose  $i\in [n-2]$ is such that $ a_{i} \leq  a_{i+2} <  a_{i+1}$,
but $\square_i$ and $\square_{i+1}$ are not both  
 on the main diagonal.
 Then  $\tpi(\ck_i(a)) = \tpi(a)$.
\end{lemma}


\begin{proof}
Define $
b = \ck_i(a) $. Our goal is to show that that $\tpi(a)=\tpi(b)$.
We have
either 
$
a_i < a_{i+2} < a_{i+1}
$
and
$
b =  a_1 \cdots a_{i+1}a_{i} a_{i+2}\cdots a_n
$,
or 
$
a_i = a_{i+2} < a_{i+1}
$
and
$ 
b =  a_1 \cdots a_{i+1}a_{i} a_{i+1}\cdots a_n
.$
In either case,
Proposition~\ref{unprime-prop} implies
that
$\PO( a_1a_2\cdots a_j ) = \PO( b_1b_2\cdots b_j)$ for $j \in [n]\setminus\{i,i+1\}$
 so we have $\cseq_j(a)= \cseq_j(b)$ for $j \in [n]\setminus\{i,i+1\}$.
Thus $\tpi_j(a) = \tpi_j(b)$ for $j \in [n] \setminus\{i,i+1,i+2\}$ and 
it is enough to show that 
$
\tpi_{i}(a)\tpi_{i+1}(a)\tpi_{i+2}(a) =\tpi_{i}(b) \tpi_{i+1}(b)\tpi_{i+2}(b).
$

Let $s(a)$ be the number of diagonal entries in $\QO(a)$ equal to $i$, $i+1$, or $i+2$.
We must have $s(a) \in \{0,1\}$ since $i$ and $i+2$ are not both on the main diagonal.
Let $r(a) \in \{0,1,2\}$ be the number of (off-diagonal) entries in $\QO(a)$ equal to $i'$, $i'+1$, or $i'+2$.
Since $\QO(b) = \fkd_i(\QO(a))$ by Proposition~\ref{unprime-prop},
we deduce from Proposition~\ref{square-lem} that $s(a) = s(b)$ and $r(a) = r(b)$.

\begin{claim}
If $\rspath_i(a)$ and $\rspath_i(b)$ intersect off the main diagonal then $\tau(a)=\tau(b)$.
\end{claim}

\begin{proof}[Proof of the claim]
In case, Lemma~\ref{bump-a-lem}
implies that $\cseq_i(a) = \cseq_i(b)$ so $\tpi_i(a) = \tpi_i(b)$.
As $s(a) = s(b)$ and $r(a) = r(b)$,
we can use
 Lemma~\ref{bac-pre-lem}
to deduce that $\tpi_{i+1}(a)\tpi_{i+2}(a) = \tpi_{i+1}(b)\tpi_{i+2}(b)$.
\end{proof}
 
 \begin{claim}
If $a_i < a_{i+2} $ and the paths $\rwpath_i(a)$ and $\rwpath_i(b)$ are disjoint then $\tau(a)=\tau(b)$.
\end{claim}

\begin{proof}[Proof of the claim]
In this case Lemma~\ref{bump-b-lem} implies that
 $\cseq_{i+1}(a) =\cseq_{i+1}(b)$ so   $\tpi_{i+2}(a) = \tpi_{i+2}(b)$.
As $s(a) = s(b)$ and $r(a) = r(b)$,
 Lemma~\ref{bac-pre-lem}
again implies $\tpi_{i}(a)\tpi_{i+1}(a) = \tpi_{i}(b)\tpi_{i+1}(b)$.
 \end{proof}
 
Thus, we may assume that 
$\rspath_i(a)$ and $\rspath_i(b)$ intersect in at most one position, which is on the main diagonal,
and that if $a_i < a_{i+2} $ then $\rwpath_i(a)$ and $\rwpath_i(b)$ intersect in at least one position.
For the next part of our argument, we will assume that if $a_i < a_{i+2} $ 
then the first position in the (nonempty) intersection of $\rwpath_i(a)$ and $\rwpath_i(b)$ is off the main diagonal.

We define an index $j$ and a number $u$ in the following way.
If $a_i=a_{i+2} <a_{i+1}$, then we set $j:=0$ and $u:=a_i$.
If instead
$a_i < a_{i+2} < a_{i+1}$, then let $j>0$ be the row index of the 
the first position in the intersection of $\rwpath_i(a)$ and $\rwpath_i(b)$.
This position 
cannot belong to $\rspath_i(a) \cap \rspath_i(b)$, so
it must be occupied in $T$, and we define $u$ to be its entry.

Define $k$ to be the row index of
 the last position in $\rwpath_i(a)$.
Then $j<k$ and 
the following observations are consequences of our assumption that 
$\rspath_i(a)$ and $\rspath_i(b)$
do not intersect off the main diagonal:
\ben
 \item[(A1)] Suppose $t \in \{1,2,\dots,k-j-1\}$ or $t = 0  < j$.
 Then row $j+t$ of $T$ contains both $u+t$ and $u+t+1$,
and the positions of $u+t$ and $u+t+1$ in row $j+t$ of $T$  are in $\rspath_i(a) \cap \rwpath_i(b)$
and  $\rspath_i(b)$, respectively.
 %
Moreover, if 
 row $j+t$ of $T$  contains $u+t-1$ then its position is in $\rwpath_i(a)$, and 
 otherwise the position of $u+t$ in row $j+t$ of $T$ is in $\rwpath_i(a)$.
It follows that if $j>0$ then row $j$ of $T$ does not contain $u-1$,
since $\rwpath_i(a)$ and $\rwpath_i(b)$ share a position in this row.
 

\item[(A2)] The position $(k,k)$ is in $\rwpath_i(a)$, since otherwise the last position in $\rwpath_i(a)$
would be an off-diagonal element of $\rspath_i(a)\cap \rspath_i(b)$.
If occupied, the entry of position $(k,k)$ in $T$  must be at least $ u + k - j-1$.


\een

Suppose $x,y \in \ZZ$ are such that $\row(T)xy$ is an involution word.
 The tableau $T \iarrow x$ differs from $T$ only in the positions
 that belong to  $\spath(T,x)$, which contain successively increasing entries until the last position which is not in $T$.
 
 If we know only the first $k-1$ positions of $\wpath(T,x)$ and $\spath(T,x)$,
 but we know that the entry of $T$ in the $k$th term of $\spath(T,x)$ is bounded below by some number $N$ when this position is present in $T$,
 then we can compute the subtableau of $T \iarrow x$ formed by omitting all entries greater than $N$.
 In this event, we can then also compute the initial subsequences of
 $\wpath(T \iarrow x,y)$ and $\spath(T \iarrow x, y)$ that consist of positions with entries of $T \iarrow x$ that are bounded above by $N$.
These observations let us deduce 
the following additional properties:
 \ben
 \item[(A3)] The first $k-1$ terms of  $\rspath_i(a)$ and $\rspath_{i+1}(b)$ coincide, 
 as do the  first $k-1$ terms of
$\rspath_{i}(b)$ and $\rspath_{i+1}(a)$, as do the  first $k-1$ terms of $\rwpath_i(a)$ and $\rwpath_{i+1}(b)$.
 

 \item[(A4)]
  The first $k-1$ terms of $\rwpath_{i}(b)$ and $\rwpath_{i+1}(a)$ 
  are the same except in the rows $j+t$  
  where $T$ does not contain  $u+t-1$, for $t \in \{1,2,\dots, k-j-1\}$ or $t=0<j$.
In these rows,
$\rwpath_{i+1}(a) $ contains the position of $u+t+1$ in $T$,
rather than the position of $u+t$  which is in $\rwpath_{i}(b)$.


   \item[(A5)] The first $j$ terms of $\spath_{i+2}(a)$ and $\spath_{i+2}(b)$ coincide, as do
   the first $j$ terms of $\wpath_{i+2}(a)$ and $\wpath_{i+2}(b)$.
If $j>0$ then term $j$ of all four paths is the position of $u+1$ in row $j$ of $T$.


   \item[(A6)] If $t \in [ k-j-1]$, then the $(j+t)$th terms of 
$\wpath_{i+2}(a), $  $ \spath_{i+2}(a), $ $ \wpath_{i+2}(b), $ and $ \spath_{i+2}(b)$
are  
 either the respective positions in row $j+t$ of $T$ of  
$u+t-1$, $ u+t$, $ u+t$, and $ u+t+1$
  when row $j+t$ of $T$ contains the entry $u+t-1$,
or   the respective positions of  
$u+t$, $ u+t+1$, $ u+t+1$, and $ u+t+1$
when the same  row does not contain $u+t-1$.
  

\een
Combining the preceding observations, we arrive at the following key claim:
\ben
\item[(A7)] Let $v=u+k-j-1$ and assume $k>1$. 
Then
the entries of 
the shifted tableaux 
$T, $ $ T \iarrow a_i, $ and $ T\iarrow a_i \iarrow a_{i+1}$
in the $(k-1)$th positions of 
$ \rspath_i(a), $ $ \rspath_{i+1}(a), $ and $ \rspath_{i+2}(a)$
are $v$, $v+1$, and $v$, respectively.
Likewise, the entries of 
the shifted tableaux 
$T$, $ T \iarrow b_i$, and $ T\iarrow b_i \iarrow b_{i+1}$
in the $(k-1)$th positions of 
$\rspath_i(b), $ $ \rspath_{i+1}(b), $ and $ \rspath_{i+2}(b)$
are $v+1$, $v$, and $v+1$, respectively.

\een
This last property still makes sense when $j=0$ and $k=1$ if we 
define the entries in the ``$0$th position'' of $\rspath_{m}(a)$ and $\rspath_{m}(b)$ to be $a_m$ and $b_m$, respectively.


We need just one other observation.
Let $U$ be the shifted tableau formed from $T$ by omitting the first $k-1$ rows.
Using Proposition~\ref{o-lem2} and property (A7), 
one can check that $a$ is equivalent under $\simICK$ to a word that begins
with $\row(U) \hs v\hs (v+1)\hs v $. If $U$ were empty or if all entries in $U$ were greater than $v+2$ then this word 
is an involution equivalent 
under $\iisim$ to $v\hs (v+1)\hs v\hs \row(U)$ which is impossible by Proposition~\ref{mat3-prop}.
Thus: 
 \ben
 \item[(A8)] The entry of $T$ in position $(k,k)$ is occupied by $v$, $v+1$, or $v+2$.
 \een
 
We can now reason precisely about the 
 possibilities for $\tpi_i(a)$, $\tpi_{i+1}(a)$, $\tpi_{i+2}(a)$, $\tpi_i(b)$, $\tpi_{i+1}(b)$, and $\tpi_{i+2}(b)$.
Below, we will refer to the entries of the shifted tableaux arranged in the  diagram
\def\AA{\PO(a_1\cdots a_{i-1})}
\def\AAA{\PO(a_1\cdots a_{i})}
\def\AAAA{\PO(a_1\cdots a_{i+1})}
\def\AAAAA{\PO(a_1\cdots a_{i+2})}
\def\BBB{\PO(b_1\cdots b_{i})}
\def\BBBB{\PO(b_1\cdots b_{i+1})}
\be\label{6-eq}{\small
\ytableausetup{boxsize = 0.8cm,aligntableaux=center}
\begin{tikzcd}[ampersand replacement=\&]
 \AA   \arrow[rd, "b_{i}"'] \arrow[r, "a_i"]  \& 				\AAA \arrow[rr, "a_{i+1}"] \&\& \AAAA \arrow[r, "a_{i+2}"] \& \AAAAA  \\
    \& 				\BBB \arrow[rr, "b_{i+1}"'] \&\& \BBBB \arrow[ru, "b_{i+2}"'] \& 
\end{tikzcd}}
\ee
where in this picture, an arrow $\xrightarrow{\ u\ }$ connects two tableaux if inserting $u$ into the first tableau according 
to Definition~\ref{iarrow-def} gives the second.
We also write 
\be\label{acb-cseq-eq}
 \cseq_{i-1}(a)= \cseq_{i-1}(b) = 
  \left[\barr{llll} \gamma_1 & \gamma_2 & \dots & \gamma_q
\\
c_1 & c_2 & \dots & c_q \earr\right].\ee

\begin{claim}
Assume that 
$\rspath_i(a)$ and $\rspath_i(b)$ intersect in at most one position, which is on the main diagonal,
and that if $a_i < a_{i+2} $ then 
the intersection of $\rwpath_i(a)$ and $\rwpath_i(b)$ is nonempty and its first position
is off the main diagonal. Then $\tau(a) =\tau(b)$.
\end{claim}

\begin{proof}[Proof of the claim]
As in earlier claims, it suffices to show  $\tpi_i(a)\tpi_{i+1}(a)\tpi_{i+2}(a) = \tpi_i(b)\tpi_{i+1}(b)\tpi_{i+2}(b)$.
As noted above, there are three possibilities for  the entry of $T$ in position $(k,k)$. 
First suppose the entry of $T$ in position $(k,k)$ is $v$.
Then, in view of Remark~\ref{iarrow-rmk},  the entries of  $T$ in positions $\{k,k+1,k+2\}\times\{k,k+1,k+2\}$ must be
$T_{k+i,k+j} = v+ i + j$ for all $0 \leq i \leq j \leq 2$.
Using Lemma~\ref{cseq-lem} and property (A7), one checks  
that the entries in these positions are the same for all six tableaux in \eqref{6-eq},
and that 
$\tpi_i(a) = (\gamma_k,\gamma_{k+1}) = \tpi_{i+2}(b)$ and
$ \tpi_{i+1}(a) = (\gamma_{k},\gamma_{k+2}) = \tpi_{i+1}(b)$ and
$ \tpi_{i+2}(a) = (\gamma_{k+1},\gamma_{k+2}) = \tpi_i(b)$.
Thus
 $\tpi_i(a)\tpi_{i+1}(a)\tpi_{i+2}(a) = \tpi_i(b)\tpi_{i+1}(b)\tpi_{i+2}(b) = (\gamma_k,\gamma_{k+2})$.


Suppose next that the entry of $T$ in position $(k,k)$ is $v + 1$.
Then, again in view of Remark~\ref{iarrow-rmk},  the entries of  $T$ in positions $\{k,k+1\}\times\{k,k+1\}$ must be
$T_{k+i,k+j} = v+ i + j$ for all $0 \leq i \leq j \leq 1$.
Assume $k>1$. Then row $k-1$ of $T$ contains   $v$ and $v+1$ in off-diagonal positions, so 
the entry in position $(k-1,k+1)$ of $T$ is at most $v+1$.
If equality holds, then the entries of the six tableaux in \eqref{6-eq}
in positions $\{k-1,k,k+1\}\times \{k,k+1\}$
must be 
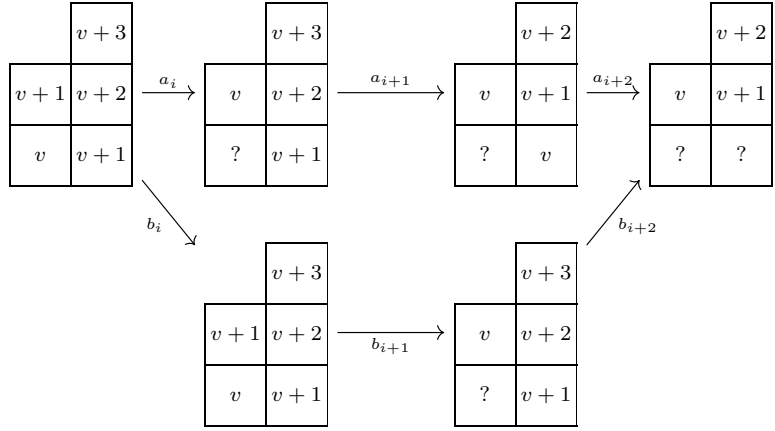
\def\AA{ {\begin{ytableau} \none & v+3  \\  v+1 & v + 2  \\ v & v +1  \end{ytableau}}}
\def\AAA{ {\begin{ytableau} \none & v+3  \\ v & v + 2  \\  ? & v +1 \end{ytableau}}}
\def\AAAA{ {\begin{ytableau} \none & v+2  \\ v & v + 1 \\ ? & v   \end{ytableau}}}
\def\AAAAA{ {\begin{ytableau} \none & v+2  \\ v & v + 1  \\ ? & ?  \end{ytableau}}}
\def\BBB{ {\begin{ytableau} \none & v+3  \\ v+1 & v + 2  \\ v & v +1 \end{ytableau}}}
\def\BBBB{ {\begin{ytableau} \none & v+3  \\ v & v + 2  \\ ? & v +1 \end{ytableau}}}
\[
\ytableausetup{boxsize = 0.8cm,aligntableaux=center}
\scriptsize
\begin{tikzcd}[ampersand replacement=\&]
 \AA   \arrow[rd, "b_{i}"'] \arrow[r, "a_i"]  \& 				\AAA \arrow[rr, "a_{i+1}"] \&\& \AAAA \arrow[r, "a_{i+2}"] \& \AAAAA  \\
    \& 				\BBB \arrow[rr, "b_{i+1}"'] \&\& \BBBB \arrow[ru, "b_{i+2}"'] \& 
\end{tikzcd}.
\]
On the other hand,
if the entry in position $(k-1,k+1)$ of $T$ is less than $v+1$ then
position $(k-1,k+2)$ of $T$ must have an entry less than $v+2$.
When this happens or when $k=1$,
the entries in the six tableaux in \eqref{6-eq} in positions $\{k,k+1\}\times\{k,k+1,k+2\}$
must instead be
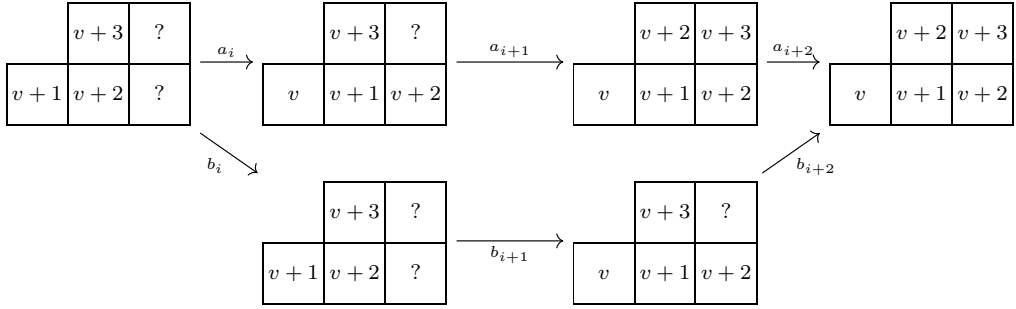
\def\AA{{\begin{ytableau} \none & v+3 & ?  \\  v+1 & v + 2 &?   \end{ytableau}}}
\def\AAA{{\begin{ytableau} \none & v+3  & ? \\ v & v + 1 & v+2  \end{ytableau}}}
\def\AAAA{{\begin{ytableau} \none & v+2 & v+3  \\ v & v + 1 & v+2  \end{ytableau}}}
\def\AAAAA{\AAAA}
\def\BBB{{\begin{ytableau} \none & v+3 & ? \\ v+1 & v + 2 &? \end{ytableau}}}
\def\BBBB{{\begin{ytableau} \none & v+3 & ?  \\ v & v + 1 & v+2  \end{ytableau}}}
\[
\ytableausetup{boxsize = 0.8cm,aligntableaux=center}
\scriptsize
\begin{tikzcd}[ampersand replacement=\&]
 \AA   \arrow[rd, "b_{i}"'] \arrow[r, "a_i"]  \& 				\AAA \arrow[rr, "a_{i+1}"] \&\& \AAAA \arrow[r, "a_{i+2}"] \& \AAAAA  \\
    \& 				\BBB \arrow[rr, "b_{i+1}"'] \&\& \BBBB \arrow[ru, "b_{i+2}"'] \& 
\end{tikzcd}
\]
where $\boxed{?}$ denotes a position that may be unoccupied.
In both cases, it follows using Lemmas~\ref{gamma-lem} 
that the values of $\gamma_{xy}$ applied to the  six tableaux in \eqref{6-eq} in positions $\{k,k+1\}\times\{k,k+1\}$
are
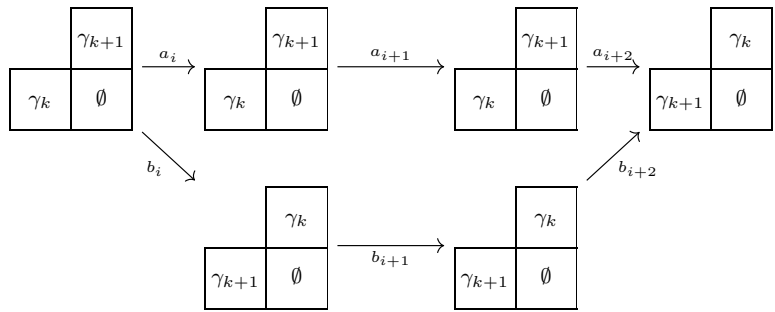
\def\AA{{\begin{ytableau} \none & \gamma_{k+1}   \\   \gamma_{k} & \emptyset     \end{ytableau}}}
\def\AAA{{\begin{ytableau}\none & \gamma_{k+1}   \\   \gamma_{k} & \emptyset   \end{ytableau}}}
\def\AAAA{{\begin{ytableau} \none & \gamma_{k+1}   \\   \gamma_{k} & \emptyset   \end{ytableau}}}
\def\AAAAA{{\begin{ytableau} \none & \gamma_{k}   \\   \gamma_{k+1} & \emptyset   \end{ytableau}}}
\def\BBB{{\begin{ytableau} \none & \gamma_{k}   \\   \gamma_{k+1} & \emptyset  \end{ytableau}}}
\def\BBBB{{\begin{ytableau} \none & \gamma_{k}   \\   \gamma_{k+1} & \emptyset   \end{ytableau}}}
\[
\ytableausetup{boxsize = 0.8cm,aligntableaux=center}
\scriptsize
\begin{tikzcd}[ampersand replacement=\&]
 \AA   \arrow[rd, "b_{i}"'] \arrow[r, "a_i"]  \& 				\AAA \arrow[rr, "a_{i+1}"] \&\& \AAAA \arrow[r, "a_{i+2}"] \& \AAAAA  \\
    \& 				\BBB \arrow[rr, "b_{i+1}"'] \&\& \BBBB \arrow[ru, "b_{i+2}"'] \& 
\end{tikzcd}.
\]
Thus, it follows by Lemma~\ref{cseq-lem} 
 that 
$
 \tpi_i(a) = \tpi_{i+1}(a) = \tpi_{i+1}(b) = \tpi_{i+2}(b) = 1
 $
and
$
\tpi_{i+2}(a) = \tpi_{i}(b) = (\gamma_k,\gamma_{k+1}),
$
so  $\tpi_i(a)\tpi_{i+1}(a)\tpi_{i+2}(a) = \tpi_i(b)\tpi_{i+1}(b)\tpi_{i+2}(b)=(\gamma_k,\gamma_{k+1})$ as needed.
  

Finally, 
suppose the entry of $T$ in position $(k,k)$ is $v + 2$.
If $k>1$ then row $k-1$ of $T$ contains $v$ and $v+1$ off the main diagonal,
so the entry in position $(k-1,k+1)$ of $T$ must be less than $v+2$.
There are two subcases depending on the entry in position $(k-1,k+2)$ of $T$.
If $k>1$ and this position contains a number less than $v+2$, or if $k=1$,
then
the entries in the six tableaux in \eqref{6-eq} in positions $\{k,k+1\}\times\{k,k+1,k+2\}$
are
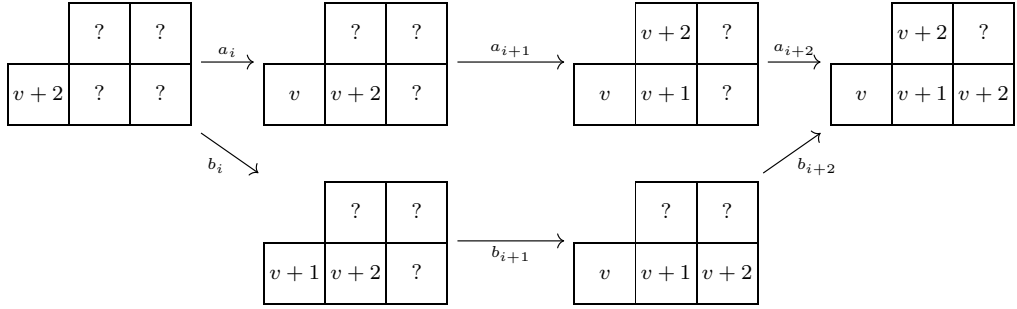
\def\AA{{\begin{ytableau} \none & ? & ?  \\  v+2 & ?  & ? \end{ytableau}}}
\def\AAA{{\begin{ytableau} \none & ? & ? \\  v & v+2& ?  \end{ytableau}}}
\def\AAAA{{\begin{ytableau} \none & v+2 & ? \\  v & v+1& ?  \end{ytableau}}}
\def\AAAAA{{\begin{ytableau} \none & v+2 & ? \\  v & v+1& v+2  \end{ytableau}}}
\def\BBB{{\begin{ytableau} \none & ?& ?  \\  v+1 & v+2& ? \end{ytableau}}}
\def\BBBB{{\begin{ytableau}  \none & ?& ?  \\  v & v+1& v+2  \end{ytableau}}}
\[
\ytableausetup{boxsize = 0.8cm,aligntableaux=center}
\scriptsize
\begin{tikzcd}[ampersand replacement=\&]
 \AA   \arrow[rd, "b_{i}"'] \arrow[r, "a_i"]  \& 				\AAA \arrow[rr, "a_{i+1}"] \&\& \AAAA \arrow[r, "a_{i+2}"] \& \AAAAA  \\
    \& 				\BBB \arrow[rr, "b_{i+1}"'] \&\& \BBBB \arrow[ru, "b_{i+2}"'] \& 
\end{tikzcd}.
\]
If $k>1$ and position $(k-1,k+2)$ of $T$ is unoccupied or contains a number greater than or equal to $v+2$,  
then positions $(k-1,k)$ and $(k-1,k+1)$ of $T$ must contain the numbers $v$ and $v+1$.
In this case 
the entries in the six tableaux in \eqref{6-eq} in positions $\{k-1,k,k+1\}\times\{k,k+1\}$
are
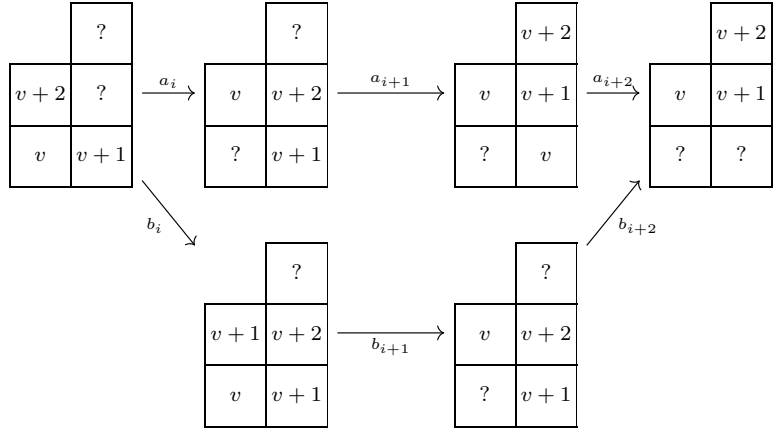
\def\AA{{\begin{ytableau} \none & ?   \\  v+2 & ?    \\ v & v +1  \end{ytableau}}}
\def\AAA{{\begin{ytableau} \none & ?  \\  v & v+2 \\ ? & v +1  \end{ytableau}}}
\def\AAAA{{\begin{ytableau} \none & v+2  \\  v & v+1 \\ ? & v     \end{ytableau}}}
\def\AAAAA{{\begin{ytableau} \none & v+2  \\  v & v+1 \\ ? & ?    \end{ytableau}}}
\def\BBB{{\begin{ytableau} \none & ?  \\  v+1 & v+2 \\ v & v +1   \end{ytableau}}}
\def\BBBB{{\begin{ytableau}  \none & ?  \\  v & v+2   \\ ? & v +1   \end{ytableau}}}
\[
\ytableausetup{boxsize = 0.8cm,aligntableaux=center}
\scriptsize
\begin{tikzcd}[ampersand replacement=\&]
 \AA   \arrow[rd, "b_{i}"'] \arrow[r, "a_i"]  \& 				\AAA \arrow[rr, "a_{i+1}"] \&\& \AAAA \arrow[r, "a_{i+2}"] \& \AAAAA  \\
    \& 				\BBB \arrow[rr, "b_{i+1}"'] \&\& \BBBB \arrow[ru, "b_{i+2}"'] \& 
\end{tikzcd}.
\]
Write $\eta_k$ and $\eta_{k+1}$ for the entries in the first row of $\cseq_{i+2}(a)$
in columns $k$ and $k+1$.
The following assertions apply equally to both of the cases above.
First, since $\cseq_{i-1}(a)=\cseq_{i-1}(b)$ and
$\cseq_{i+2}(a)=\cseq_{i+2}(b)$, one can check using
   Lemmas~\ref{gamma-lem} and \ref{cseq-lem}  that  $\gamma_k = \eta_k$.
If  $\cseq_{i-1}(a)$ has only $k$ columns, then
it follows similarly that
 the values of $\gamma_{xy}$ applied to the  six tableaux in \eqref{6-eq} in positions $\{k,k+1\}\times\{k,k+1\}$
are
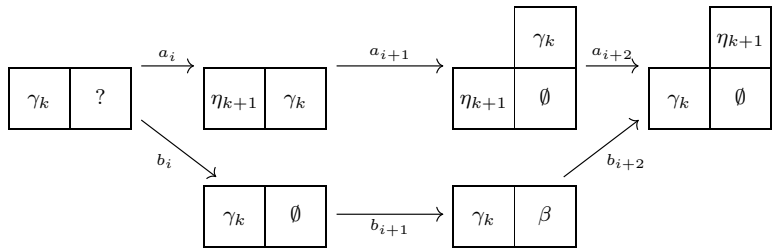
\def\AA{{\begin{ytableau} \none & \none  \\   \gamma_{k} & ?     \end{ytableau}}}
\def\AAA{{\begin{ytableau}  \none & \none  \\  \eta_{k+1}  & \gamma_{k}    \end{ytableau}}}
\def\AAAA{{\begin{ytableau}    \none & \gamma_{k}  \\  \eta_{k+1}  & \emptyset  \end{ytableau}}}
\def\AAAAA{{\begin{ytableau}  \none & \eta_{k+1}   \\  \gamma_{k}  & \emptyset  \end{ytableau}}}
\def\BBB{{\begin{ytableau}     \gamma_{k} & \emptyset  \end{ytableau}}}
\def\BBBB{{\begin{ytableau}     \gamma_{k} & \beta   \end{ytableau}}}
\[
\ytableausetup{boxsize = 0.8cm,aligntableaux=center}
\scriptsize
\begin{tikzcd}[ampersand replacement=\&]
 \AA   \arrow[rd, "b_{i}"'] \arrow[r, "a_i"]  \& 				\AAA \arrow[rr, "a_{i+1}"] \&\& \AAAA \arrow[r, "a_{i+2}"] \& \AAAAA  \\
    \& 				\BBB \arrow[rr, "b_{i+1}"'] \&\& \BBBB \arrow[ru, "b_{i+2}"'] \& 
\end{tikzcd}
\]
where we set $\beta:=\emptyset$ in the first subcase above and $\beta:=\eta_{k+1}$ in the second.
Thus  
$ \tpi_i(a) = \tpi_{i+2}(a)=(\gamma_k,\eta_{k+1})$
and
$ \tpi_{i+1}(a)=\tpi_{i}(b) =\tpi_{i+1}(b) = \tpi_{i+2}(b)=1,$
 giving
  $\tpi_i(a)\tpi_{i+1}(a)\tpi_{i+2}(a) = \tpi_i(b)\tpi_{i+1}(b)\tpi_{i+2}(b) = 1$ as desired.
If $\cseq_{i-1}(a)$ has at least $k+1$ columns, then
it follows likewise that
the values of $\gamma_{xy}$ applied to the  six tableaux in \eqref{6-eq} in positions $\{k,k+1\}\times\{k,k+1\}$
are
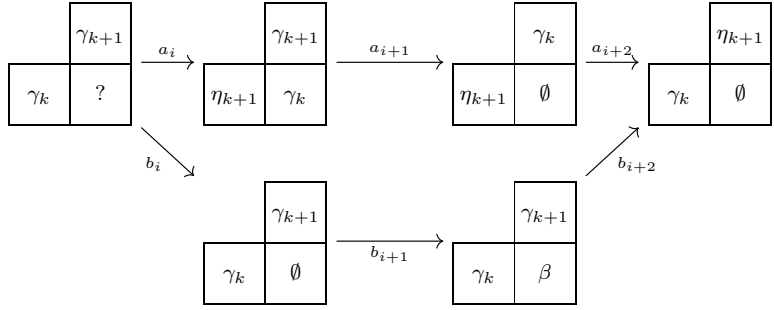
\def\AA{{\begin{ytableau} \none & \gamma_{k+1}  \\   \gamma_{k} & ?     \end{ytableau}}}
\def\AAA{{\begin{ytableau}   \none & \gamma_{k+1}  \\   \eta_{k+1} & \gamma_{k}   \end{ytableau}}}
\def\AAAA{{\begin{ytableau} \none & \gamma_{k}  \\   \eta_{k+1} & \emptyset \end{ytableau}}}
\def\AAAAA{{\begin{ytableau}  \none & \eta_{k+1}   \\  \gamma_{k}  & \emptyset  \end{ytableau}}}
\def\BBB{{\begin{ytableau} \none & \gamma_{k+1}  \\   \gamma_{k} & \emptyset  \end{ytableau}}}
\def\BBBB{{\begin{ytableau} \none & \gamma_{k+1}  \\   \gamma_{k} & \beta   \end{ytableau}}}
\[
\ytableausetup{boxsize = 0.8cm,aligntableaux=center}
\scriptsize
\begin{tikzcd}[ampersand replacement=\&]
 \AA   \arrow[rd, "b_{i}"'] \arrow[r, "a_i"]  \& 				\AAA \arrow[rr, "a_{i+1}"] \&\& \AAAA \arrow[r, "a_{i+2}"] \& \AAAAA  \\
    \& 				\BBB \arrow[rr, "b_{i+1}"'] \&\& \BBBB \arrow[ru, "b_{i+2}"'] \& 
\end{tikzcd}
\]
where $\beta$ has the same definition as before.
Thus Lemma~\ref{cseq-lem} gives
 $
 \tpi_i(a) = \tpi_{i+2}(a)=(\gamma_k,\eta_{k+1})
$
and
$
\tpi_{i+1}(a) = (\gamma_k,\gamma_{k+1})
$
while 
$ \tpi_{i}(b) =\tpi_{i+1}(b) = 1 $ and $ \tpi_{i+2}(b)= (\gamma_{k+1},\eta_{k+1})$,
so
 $\tpi_i(a)\tpi_{i+1}(a)\tpi_{i+2}(a) = \tpi_i(b)\tpi_{i+1}(b)\tpi_{i+2}(b) = (\gamma_{k+1},\eta_{k+1})$
 as needed.
This completes our proof of the claim.
\end{proof}

It remains to consider the case when 
$a_i < a_{i+2}$ and
$\rspath_i(a)$ and $\rspath_i(b)$
 do not intersect off the main diagonal, but $\rwpath_i(a)$ and $\rwpath_i(b)$ intersect in a unique position which is on the main diagonal.
Suppose this position is $(k,k)$.
This position must be occupied in $T$, 
since otherwise one can check using Remark~\ref{iarrow-rmk} that
both $i$ and $i+2$ would be on the main diagonal of $\QO(a)$.
The reasoning we used to justify  (A3)  
lets us similarly derive the following claims:
\ben
\item[(B1)] The first $k-1$ terms of  $\spath_i(a)$ and $\spath_{i+1}(b)$ coincide,
as do the first $k-1$ terms of $\spath_{i+1}(a)$ and $\spath_{i}(b)$.
Each of the first $k-1$ terms of the first two paths 
is strictly to the right of the main diagonal
and strictly to the left of the corresponding term in the second two paths.
 The same statements hold for the 
 corresponding weak bumping paths.
 
  
 \item[(B2)] The first $k-1$ terms of $\spath_{i+2}(a)$ and $\spath_{i+2}(b)$ coincide.
 Each of the first $k-1$ terms of these paths 
is strictly to the right of the corresponding term in $\spath_i(a)$ or $\spath_{i+1}(b)$,
 and weakly to the left of corresponding term in  $\spath_{i+1}(a)$ or $\spath_{i}(b)$.
 The same statements hold for the 
  corresponding weak bumping paths. 
  \een
If $k=1$ then let $u:=a_i=b_{i+1}<v:= a_{i+2}=b_{i+2}<w:= a_{i+1} = b_i.$ If $k>1$ 
then define $u$, $v$, and $w$ to be the entries of 
$ T$, $T\iarrow a_{i}\iarrow a_{i+1}$, and $T\iarrow a_{i}$, respectively,
in position $k-1$ of $\spath_i(a)$, $\spath_{i+2}(a)$, and $\spath_{i+1}(a)$ respectively.
It follows from (B1) and (B2) that:
\begin{itemize}
\item[(B3)] Assume $k>1$. Then $u$ is also the entry of $T\iarrow b_{i}$ in position $k-1$ of $\spath_{i+1}(b)$.
Likewise, $v$ is also the entry of $T\iarrow b_{i}\iarrow b_{i+1}$  in position $k-1$ of $\spath_{i+2}(b)$.
In turn, $w$ is also the entry of $T$ in position $k-1$ of $\spath_{i}(b)$, and $u<v<w$.

\item[(B4)] The entry of $T$ in position $(k,k)$ is at least $w$ since $(k,k) \in \rwpath_i(b)$.
\end{itemize}
This leaves us with three possibilities $\tpi_i(a)$, $\tpi_{i+1}(a)$, $\tpi_{i+2}(a)$, $\tpi_i(b)$, $\tpi_{i+1}(b)$, and $\tpi_{i+2}(b)$,
as we discuss in the proof of our final claim.

\begin{claim}
Assume $a_i < a_{i+2} $ and
$\rspath_i(a)$ and $\rspath_i(b)$
 have no main diagonal intersection, but $\rwpath_i(a)$ and $\rwpath_i(b)$ intersect in a unique diagonal position $(k,k)$.
Then $\tau(a) =\tau(b)$.
\end{claim}

\begin{proof}[Proof of the claim]
Denote $\cseq_{i-1}(a)= \cseq_{i-1}(b) $ as in \eqref{acb-cseq-eq} above.
Again write $\eta_k$ and $\eta_{k+1}$ for the entries in the first row of $\cseq_{i+2}(a)$
in columns $k$ and $k+1$.

First suppose the entry in position $(k,k)$ of $T$
is $w$. 
Then, in view of Remark~\ref{iarrow-rmk}, the entries of $T$  in positions $\{k,k+1\}\times\{k,k+1\}$ must be
$T_{k+i,k+j} = w+ i + j$ for all $0 \leq i \leq j \leq 1$.
If $k>1$, then row $k-1$ of $T$ contains both $u$ and $w$ in positions off the main diagonal,
so the entry in position $(k-1,k+1)$ of $T$ is at most $w$.
If $k>1$ and this entry is equal to $w$, then the entries of the six tableaux in \eqref{6-eq}
in positions $\{k-1,k,k+1\}\times\{k,k+1\}$
are
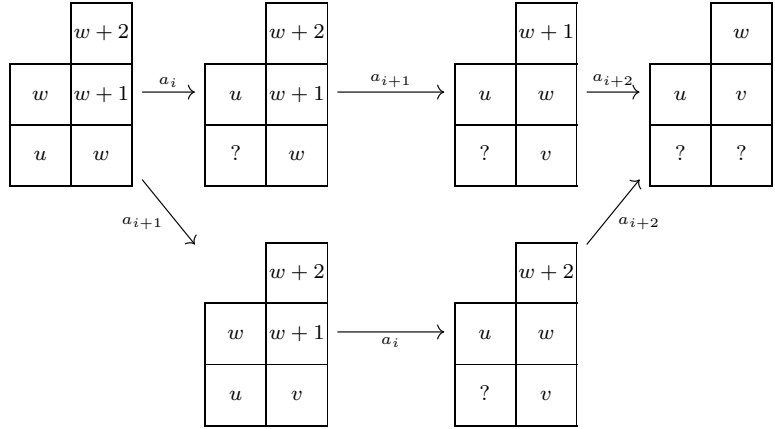
\def\AA{{\begin{ytableau} \none & w+2  \\  w & w + 1  \\ u & w  \end{ytableau}}}
\def\AAA{{\begin{ytableau} \none & w+2  \\  u & w + 1  \\ ? & w \end{ytableau}}}
\def\AAAA{{\begin{ytableau}  \none & w+1  \\  u & w   \\ ? & v \end{ytableau}}}
\def\AAAAA{{\begin{ytableau}  \none & w  \\  u & v   \\ ? & ?  \end{ytableau}}}
\def\BBB{{\begin{ytableau} \none & w+2  \\  w & w + 1  \\ u & v  \end{ytableau}}}
\def\BBBB{{\begin{ytableau}\none & w+2  \\  u & w  \\ ? & v \end{ytableau}}}
\[
\ytableausetup{boxsize = 0.8cm,aligntableaux=center}
\scriptsize
\begin{tikzcd}[ampersand replacement=\&]
 \AA   \arrow[rd, "a_{i+1}"'] \arrow[r, "a_i"]  \& 				\AAA \arrow[rr, "a_{i+1}"] \&\& \AAAA \arrow[r, "a_{i+2}"] \& \AAAAA  \\
    \& 				\BBB \arrow[rr, "a_{i}"'] \&\& \BBBB \arrow[ru, "a_{i+2}"'] \& 
\end{tikzcd}.
\]
Alternatively, if $k>1$ and the entry in position $(k-1,k+1)$ of $T$ is less than $w$,
then the entry of $T$ in position $(k-1,k+2)$ must be occupied by a number less than $w+1$.
In this case, or if $k=1$,
 the entries of the six tableaux in \eqref{6-eq}
in positions $\{k,k+1\}\times\{k,k+1,k+2\}$
are
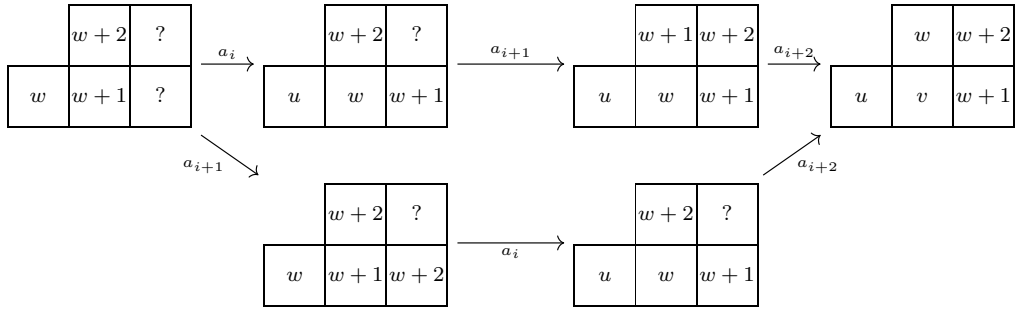
\def\AA{{\begin{ytableau} \none & w+2 &?  \\  w & w + 1 &?    \end{ytableau}}}
\def\AAA{{\begin{ytableau} \none & w+2 & ?  \\  u & w & w +1    \end{ytableau}}}
\def\AAAA{{\begin{ytableau}  \none & w+1 & w+2  \\  u & w & w+1    \end{ytableau}}}
\def\AAAAA{{\begin{ytableau}   \none & w & w+2  \\  u & v & w+1     \end{ytableau}}}
\def\BBB{{\begin{ytableau}\none & w+2 &?  \\  w & w + 1 &w+2  \end{ytableau}}}
\def\BBBB{{\begin{ytableau}\none & w+2 &?  \\  u & w  &w+1\end{ytableau}}}
\[
\ytableausetup{boxsize = 0.8cm,aligntableaux=center}
\scriptsize
\begin{tikzcd}[ampersand replacement=\&]
 \AA   \arrow[rd, "a_{i+1}"'] \arrow[r, "a_i"]  \& 				\AAA \arrow[rr, "a_{i+1}"] \&\& \AAAA \arrow[r, "a_{i+2}"] \& \AAAAA  \\
    \& 				\BBB \arrow[rr, "a_{i}"'] \&\& \BBBB \arrow[ru, "a_{i+2}"'] \& 
\end{tikzcd}.
\]
In both situations, it follows by Lemma~\ref{gamma-lem} 
that
the values of $\gamma_{xy}$ applied to the  six tableaux in \eqref{6-eq} in positions $\{k,k+1\}\times\{k,k+1\}$
are
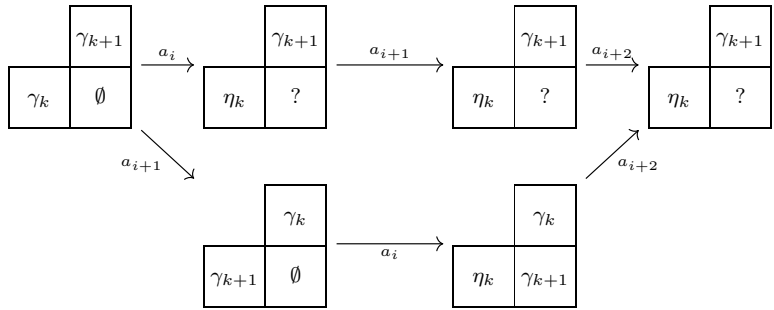
\def\AA{{\begin{ytableau} \none & \gamma_{k+1}  \\   \gamma_{k} & \emptyset     \end{ytableau}}}
\def\AAA{{\begin{ytableau}   \none & \gamma_{k+1}  \\   \eta_k  & ?    \end{ytableau}}}
\def\AAAA{{\begin{ytableau}  \none & \gamma_{k+1}  \\   \eta_k  & ? \end{ytableau}}}
\def\AAAAA{{\begin{ytableau}   \none & \gamma_{k+1}  \\   \eta_k  & ? \end{ytableau}}}
\def\BBB{{\begin{ytableau}  \none & \gamma_{k}  \\   \gamma_{k+1} & \emptyset  \end{ytableau}}}
\def\BBBB{{\begin{ytableau}  \none & \gamma_{k}  \\   \eta_{k} & \gamma_{k+1} \end{ytableau}}}
\[
\ytableausetup{boxsize = 0.8cm,aligntableaux=center}
\scriptsize
\begin{tikzcd}[ampersand replacement=\&]
 \AA   \arrow[rd, "a_{i+1}"'] \arrow[r, "a_i"]  \& 				\AAA \arrow[rr, "a_{i+1}"] \&\& \AAAA \arrow[r, "a_{i+2}"] \& \AAAAA  \\
    \& 				\BBB \arrow[rr, "a_{i}"'] \&\& \BBBB \arrow[ru, "a_{i+2}"'] \& 
\end{tikzcd}
\]
so by Lemma~\ref{cseq-lem}
we have
$ \tpi_i(a) = (\gamma_k,\eta_k) $ and $ \tpi_{i+1}(a) = \tpi_{i+2}(a) =1$
while
$\tpi_i(b) = \tpi_{i+2}(b) = (\gamma_k,\gamma_{k+1}) $ and $ \tpi_{i+1}(b)= (\eta_k,\gamma_{k+1})$,
so
$\tpi_i(a)\tpi_{i+1}(a)\tpi_{i+2}(a) = \tpi_i(b)\tpi_{i+1}(b)\tpi_{i+2}(b) = (\gamma_k,\eta_k)$
as desired.


Suppose next that the entry in position $(k,k)$ of $T$
is $ w+1$. 
If $k>1$ then the entry in position $(k-1,k+1)$ of $T$ is at most $w$, so 
the entries of the six tableaux in \eqref{6-eq}
in positions $\{k,k+1\}\times\{k,k+1\}$
are
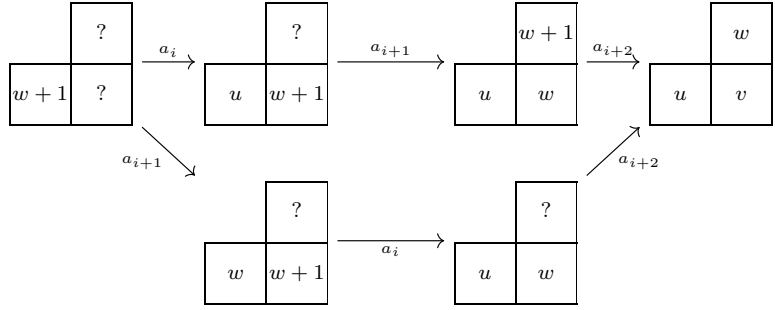
\def\AA{{\begin{ytableau} \none & ?  \\  w+1 & ?  \end{ytableau}}}
\def\AAA{{\begin{ytableau} \none & ?  \\  u & w+1\end{ytableau}}}
\def\AAAA{{\begin{ytableau}  \none & w+1  \\  u & w  \end{ytableau}}}
\def\AAAAA{{\begin{ytableau}   \none & w  \\  u & v  \end{ytableau}}}
\def\BBB{{\begin{ytableau}  \none & ?  \\  w & w+1  \end{ytableau}}}
\def\BBBB{{\begin{ytableau} \none & ?  \\  u & w  \end{ytableau}}}
\[
\ytableausetup{boxsize = 0.8cm,aligntableaux=center}
\scriptsize
\begin{tikzcd}[ampersand replacement=\&]
 \AA   \arrow[rd, "a_{i+1}"'] \arrow[r, "a_i"]  \& 				\AAA \arrow[rr, "a_{i+1}"] \&\& \AAAA \arrow[r, "a_{i+2}"] \& \AAAAA  \\
    \& 				\BBB \arrow[rr, "a_{i}"'] \&\& \BBBB \arrow[ru, "a_{i+2}"'] \& 
\end{tikzcd}.
\]
First assume  the array $\cseq_{i-1}(a)$ has only $k$ columns.
Then
it follows by Lemma~\ref{gamma-lem} 
that
the values of $\gamma_{xy}$ applied to the  six tableaux in \eqref{6-eq} in positions $\{k,k+1\}\times\{k,k+1\}$
are
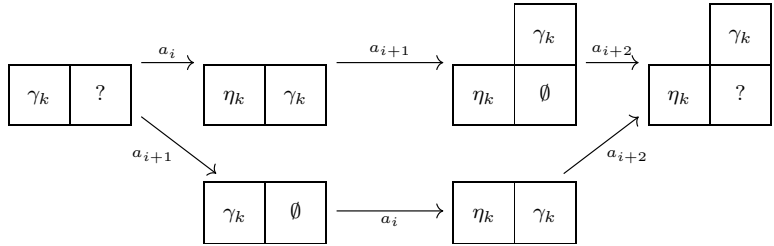
\def\AA{{\begin{ytableau} \none & \none  \\   \gamma_{k} & ?     \end{ytableau}}}
\def\AAA{{\begin{ytableau}   \none &\none  \\   \eta_k  & \gamma_{k}    \end{ytableau}}}
\def\AAAA{{\begin{ytableau}  \none & \gamma_{k}  \\   \eta_k  & \emptyset \end{ytableau}}}
\def\AAAAA{{\begin{ytableau}  \none & \gamma_{k}  \\   \eta_k  & ?   \end{ytableau}}}
\def\BBB{{\begin{ytableau}    \gamma_{k} & \emptyset  \end{ytableau}}}
\def\BBBB{{\begin{ytableau}    \eta_{k} & \gamma_k \end{ytableau}}}
\[
\ytableausetup{boxsize = 0.8cm,aligntableaux=center}
\scriptsize
\begin{tikzcd}[ampersand replacement=\&]
 \AA   \arrow[rd, "a_{i+1}"'] \arrow[r, "a_i"]  \& 				\AAA \arrow[rr, "a_{i+1}"] \&\& \AAAA \arrow[r, "a_{i+2}"] \& \AAAAA  \\
    \& 				\BBB \arrow[rr, "a_{i}"'] \&\& \BBBB \arrow[ru, "a_{i+2}"'] \& 
\end{tikzcd}
\]
so by Lemma~\ref{cseq-lem} we have
 $
 \tpi_i(a) = \tpi_{i+1}(b)=(\gamma_k,\eta_{k})
$
and
$
 \tpi_{i+1}(a)=\tpi_{i+2}(a)=\tpi_{i}(b) = \tpi_{i+2}(b)=1,
 $
 so
  $\tpi_i(a)\tpi_{i+1}(a)\tpi_{i+2}(a) = \tpi_i(b)\tpi_{i+1}(b)\tpi_{i+2}(b) = (\gamma_k,\eta_{k})$ as needed.
 If $\cseq_{i-1}(a)$ has at least $k+1$ columns, then
it follows likewise that
the values of $\gamma_{xy}$ applied to the  six tableaux in \eqref{6-eq} in positions $\{k,k+1\}\times\{k,k+1\}$
are
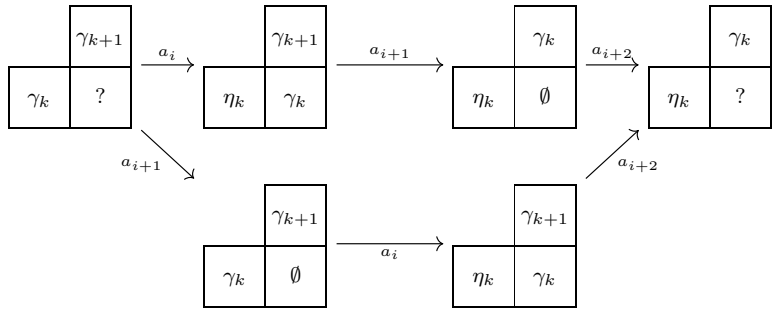
\def\AA{{\begin{ytableau} \none & \gamma_{k+1}  \\   \gamma_{k} & ?     \end{ytableau}}}
\def\AAA{{\begin{ytableau}   \none & \gamma_{k+1}  \\   \eta_k  & \gamma_{k}    \end{ytableau}}}
\def\AAAA{{\begin{ytableau}  \none & \gamma_{k}  \\   \eta_k  & \emptyset \end{ytableau}}}
\def\AAAAA{{\begin{ytableau}  \none & \gamma_{k}  \\   \eta_k  & ?   \end{ytableau}}}
\def\BBB{{\begin{ytableau}  \none &  \gamma_{k+1}  \\   \gamma_{k} & \emptyset  \end{ytableau}}}
\def\BBBB{{\begin{ytableau} \none &  \gamma_{k+1}  \\   \eta_{k} & \gamma_k \end{ytableau}}}
\[
\ytableausetup{boxsize = 0.8cm,aligntableaux=center}
\scriptsize
\begin{tikzcd}[ampersand replacement=\&]
 \AA   \arrow[rd, "a_{i+1}"'] \arrow[r, "a_i"]  \& 				\AAA \arrow[rr, "a_{i+1}"] \&\& \AAAA \arrow[r, "a_{i+2}"] \& \AAAAA  \\
    \& 				\BBB \arrow[rr, "a_{i}"'] \&\& \BBBB \arrow[ru, "a_{i+2}"'] \& 
\end{tikzcd}
\]
so by Lemma~\ref{cseq-lem} we have
 $
 \tpi_i(a) = \tpi_{i+1}(b)=(\gamma_k,\eta_{k})
$
and
$
 \tpi_{i+1}(a)=\tpi_{i+2}(b)=(\gamma_k,\gamma_{k+1})
 $
 and
 $
 \tpi_{i+2}(a)=\tpi_{i}(b) =1,
 $
 so
  $\tpi_i(a)\tpi_{i+1}(a)\tpi_{i+2}(a) = \tpi_i(b)\tpi_{i+1}(b)\tpi_{i+2}(b) = (\gamma_k,\gamma_{k+1},\eta_{k})$
  as desired.

  Finally suppose that the entry in position $(k,k)$ of $T$
is $x > w+1$. 
If $k>1$ then the entry in position $(k-1,k+1)$ of $T$ is at most $w$, so 
the entries of the six tableaux in \eqref{6-eq}
in positions $\{k,k+1\}\times\{k,k+1\}$
are
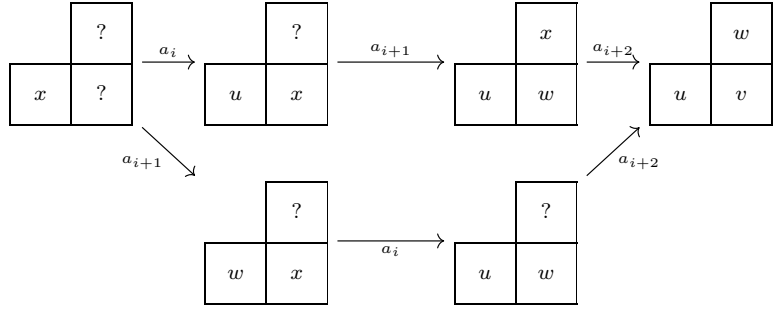
\def\AA{{\begin{ytableau} \none & ?  \\  x & ?  \end{ytableau}}}
\def\AAA{{\begin{ytableau} \none & ?  \\  u & x\end{ytableau}}}
\def\AAAA{{\begin{ytableau}  \none & x  \\  u & w  \end{ytableau}}}
\def\AAAAA{{\begin{ytableau}   \none & w  \\  u & v  \end{ytableau}}}
\def\BBB{{\begin{ytableau}  \none & ?  \\  w & x  \end{ytableau}}}
\def\BBBB{{\begin{ytableau} \none & ?  \\  u & w  \end{ytableau}}}
\[
\ytableausetup{boxsize = 0.8cm,aligntableaux=center}
\scriptsize
\begin{tikzcd}[ampersand replacement=\&]
 \AA   \arrow[rd, "a_{i+1}"'] \arrow[r, "a_i"]  \& 				\AAA \arrow[rr, "a_{i+1}"] \&\& \AAAA \arrow[r, "a_{i+2}"] \& \AAAAA  \\
    \& 				\BBB \arrow[rr, "a_{i}"'] \&\& \BBBB \arrow[ru, "a_{i+2}"'] \& 
\end{tikzcd}.
\]
 If  the array $\cseq_{i-1}(a)$ has only $k$ columns,
the values of $\gamma_{xy}$ applied to the  six tableaux in \eqref{6-eq} in positions $\{k,k+1\}\times\{k,k+1\}$
are
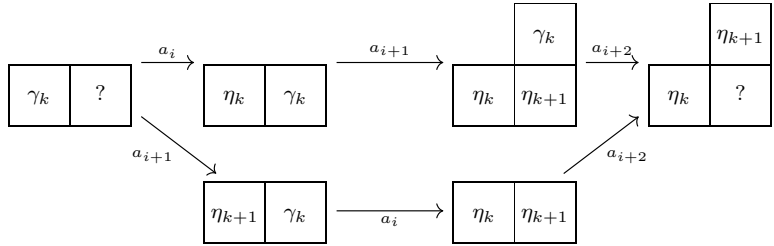
\def\AA{{\begin{ytableau} \none & \none  \\   \gamma_{k} & ?     \end{ytableau}}}
\def\AAA{{\begin{ytableau}   \none &\none  \\   \eta_k  & \gamma_{k}    \end{ytableau}}}
\def\AAAA{{\begin{ytableau}  \none & \gamma_{k}  \\   \eta_k  & \eta_{k+1} \end{ytableau}}}
\def\AAAAA{{\begin{ytableau}  \none & \eta_{k+1}  \\   \eta_k  & ?   \end{ytableau}}}
\def\BBB{{\begin{ytableau}    \eta_{k+1} & \gamma_k  \end{ytableau}}}
\def\BBBB{{\begin{ytableau}    \eta_{k} & \eta_{k+1} \end{ytableau}}}
\[
\ytableausetup{boxsize = 0.8cm,aligntableaux=center}
\scriptsize
\begin{tikzcd}[ampersand replacement=\&]
 \AA   \arrow[rd, "a_{i+1}"'] \arrow[r, "a_i"]  \& 				\AAA \arrow[rr, "a_{i+1}"] \&\& \AAAA \arrow[r, "a_{i+2}"] \& \AAAAA  \\
    \& 				\BBB \arrow[rr, "a_{i}"'] \&\& \BBBB \arrow[ru, "a_{i+2}"'] \& 
\end{tikzcd}
\]
so by Lemma~\ref{cseq-lem} we have
 $
 \tpi_i(a) =(\gamma_k,\eta_{k})
 $
 and
 $
 \tpi_{i+1}(a)=1
 $
 and
 $ 
 \tpi_{i+2}(a) = (\gamma_k,\eta_{k+1})
$
while
$\tpi_i(b)= (\gamma_k,\eta_{k+1})$
and
$
 \tpi_{i+1}(b)=( \eta_{k}, \eta_{k+1})
$
and
$
    \tpi_{i+2}(b)=1,
 $
 so we have
  $\tpi_i(a)\tpi_{i+1}(a)\tpi_{i+2}(a) = \tpi_i(b)\tpi_{i+1}(b)\tpi_{i+2}(b) = (\gamma_k, \eta_{k+1},\eta_k)$ as needed.
If  $\cseq_{i-1}(a)$ has at least $k+1$ columns, then
the values of $\gamma_{xy}$ applied to the  six tableaux in \eqref{6-eq} in positions $\{k,k+1\}\times\{k,k+1\}$
are
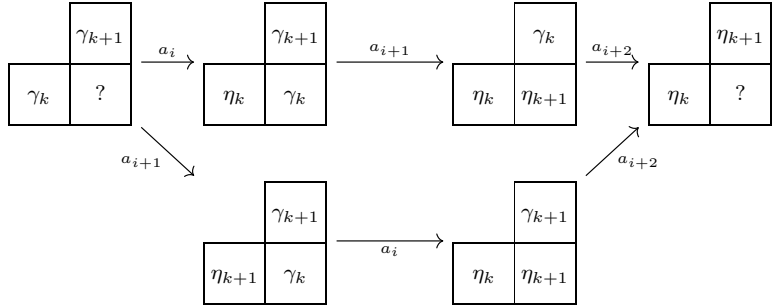
\def\AA{{\begin{ytableau} \none & \gamma_{k+1}  \\   \gamma_{k} & ?     \end{ytableau}}}
\def\AAA{{\begin{ytableau}   \none &\gamma_{k+1}  \\   \eta_k  & \gamma_{k}    \end{ytableau}}}
\def\AAAA{{\begin{ytableau}  \none & \gamma_{k}  \\   \eta_k  & \eta_{k+1} \end{ytableau}}}
\def\AAAAA{{\begin{ytableau}  \none & \eta_{k+1}  \\   \eta_k  & ?   \end{ytableau}}}
\def\BBB{{\begin{ytableau}  \none & \gamma_{k+1}  \\  \eta_{k+1} & \gamma_k  \end{ytableau}}}
\def\BBBB{{\begin{ytableau} \none & \gamma_{k+1}  \\   \eta_{k} & \eta_{k+1} \end{ytableau}}}
\[
\ytableausetup{boxsize = 0.8cm,aligntableaux=center}
\scriptsize
\begin{tikzcd}[ampersand replacement=\&]
 \AA   \arrow[rd, "a_{i+1}"'] \arrow[r, "a_i"]  \& 				\AAA \arrow[rr, "a_{i+1}"] \&\& \AAAA \arrow[r, "a_{i+2}"] \& \AAAAA  \\
    \& 				\BBB \arrow[rr, "a_{i}"'] \&\& \BBBB \arrow[ru, "a_{i+2}"'] \& 
\end{tikzcd}
\]
so by Lemma~\ref{cseq-lem} we have
 $
 \tpi_i(a) =(\gamma_k,\eta_{k})
 $
 and
 $
 \tpi_{i+1}(a)=(\gamma_k,\gamma_{k+1})
 $
 and
 $ 
 \tpi_{i+2}(a) = (\gamma_k,\eta_{k+1})
$
while
$\tpi_i(b)= (\gamma_k,\eta_{k+1})
$
and
$
 \tpi_{i+1}(b)=( \eta_{k}, \eta_{k+1})
$
and
$
    \tpi_{i+2}(b)=(\gamma_{k+1},\eta_{k+1}),
$
 so
  $\tpi_i(a)\tpi_{i+1}(a)\tpi_{i+2}(a) = \tpi_i(b)\tpi_{i+1}(b)\tpi_{i+2}(b) = (\gamma_k, \eta_{k+1},\gamma_{k+1},\eta_k)$
  as desired.
  This completes our proof of the claim.
  \end{proof}
  Combining our successive claims also completes the proof of the lemma.
\end{proof}

\subsection{Proofs of Theorems~\ref{o-thm1} and \ref{ck-fkd-thm}}\label{proof-sect4}

Combining all of the results above now lets us
fill in the proofs to 
Theorems~\ref{o-thm1} and \ref{ck-fkd-thm}.

 \begin{proof}[Proof of Theorem~\ref{o-thm1}]
 Remark~\ref{iarrow-rmk} and Proposition~\ref{o-lem2} imply that 
if $\hat a \in \iR^+(z)$ for some $z \in I_\ZZ$,
then $\PO(\hat a)$ is an increasing shifted tableau with no primes on the main diagonal whose row reading word is in $\iR^+(z)$. 
In this case it follows by definition that $\QO(\hat a)$ is a standard shifted tableau of the same shape.
 
Let $(P,Q)$ be an arbitrary pair of shifted tableaux of the same shape,
such that $Q$ is standard and $P$ increasing with no primed on the main diagonal and $\row(P) \in \iR^+(z)$.
The unprimed form \cite[Thm. 5.19]{HMP4} of  Theorem~\ref{o-thm1}  asserts that there is a unique
unprimed word $a \in \iR(z)$ with $\PO(a) = \unprime(P)$ and $\QO(a) = \unprime_{\diag}(Q)$.
Since we have $\gamma_{ii}(P) \in \cyc(z)$ for all diagonal positions $(i,i)$ in $P$,
Proposition~\ref{tau-prop} implies that there is a unique way to assign primes to the commutations in $a$
to obtain a primed word $\hat a \in \iR^+(z)$ with $\PO(\hat a) = P$ and $\QO(\hat a) = Q$.
We conclude that $\hat a \mapsto (\PO(\hat a),\QO(\hat a))$ is a bijection from $\iR^+(z)$ to the desired image.
\end{proof}


 \begin{proof}[Proof of Theorem~\ref{ck-fkd-thm}]
Suppose $\hat a$ is a primed involution word with $n = \ell(\hat a)$ and $a = \unprime(\hat a)$. 
Choose $i \in \ZZ$ with $i+2 \in [n]$ and let $ \hat b = \ck_i(\hat a)$.
We wish to show that
$\PO(\hat a) = \PO(\hat b)$ and   $\QO(\hat b) = \fkd_i(\QO(\hat a))$. 
This holds if $i \leq 0$ by Propositions~\ref{first-toggle-obs} and \ref{second-toggle-obs}.
Assume $i\in[n-2]$ and let $b = \unprime(\hat b)$. Then $b = \ck_i(a)$ by Lemma~\ref{unpri-word-lem}
and we have 
\be\label{final-eq1}\ba
\unprime(\PO(\hat a)) &=\PO(a) =  \PO(b) = \unprime(\PO(\hat b))\ea
\ee
by  Proposition~\ref{unprime-tab-prop} for the first and last equalities and Proposition~\ref{unprime-prop} for the second equality.
Likewise, we have
\be\label{final-eq2}\ba 
\unprime_{\diag}(\fkd_i(\QO(\hat a))) 
&=\fkd_i(\unprime_{\diag}(\QO(\hat a))) 
=\fkd_i(\QO(a)) 
\\&=\QO(b)
= \unprime_{\diag}(\QO(\hat b))
\ea
\ee
by \eqref{up-fkd-eq} for the first equality, Proposition~\ref{unprime-tab-prop} for the second and last equalities,
and
Proposition~\ref{unprime-prop} for the third equality.

As usual write $\square_j$ for the box of $\QO(\hat a)$
containing $j$ or $j'$.
If $\square_i$ and $\square_{i+2}$ are both on the main diagonal, then 
we have $\PO(\hat a) = \PO(\hat b)$  by Lemma~\ref{diag-box-lem}.
Otherwise, we have $\tau(a) = \tau(b)$ by Lemmas~\ref{bac-lem} and \ref{acb-lem},
so $\PO(\hat a) = \PO(\hat b)$ follows from Proposition~\ref{tau-prop} and \eqref{final-eq1}.

It follows from the definitions of $\fkd_i$ and $\QO$ that   $\fkd_i(\QO(\hat a))$ and
$\QO(\hat b)$ each only differ from $\QO(\hat a)$ in their entries in positions
$\square_i$, $\square_{i+1}$,  and $\square_{i+2}$. In view of \eqref{final-eq2},
the only possible difference between $\fkd_i(\QO(\hat a))$ and $\QO(\hat b)$ is whether there are primes in 
whichever of $\square_i$, $\square_{i+1}$, or $\square_{i+2}$ are also on the main diagonal.

If all three of $\square_i$, $\square_{i+1}$, and $\square_{i+2}$ are off the diagonal 
then necessarily $\fkd_i(\QO(\hat a)) = \QO(\hat b)$.
If exactly two of these positions are on the main diagonal then the same conclusion holds by 
Lemma~\ref{diag-box-lem}.
We cannot have all three of $\square_i$, $\square_{i+1}$, and $\square_{i+2}$ on the main diagonal,
and if exactly one of these positions is on the main diagonal then
we just need
to show that its entry is primed in $\fkd_i(\QO(\hat a))$ if and only if it is primed in $\QO(\hat b)$,
or equivalently that 
$\primesdiag(\fkd_i(\QO(\hat a))) = \primesdiag(\QO(\hat b))$.
This holds since \eqref{up-fkd-eq}
asserts that $\primesdiag(\fkd_i\QO(\hat a)) = \primesdiag((\QO(\hat a)))$,
and by definition 
$\primes(\PO(\hat a)) + \primesdiag(\QO(\hat a))=\primes(\hat a)
=\primes(\hat b) = \primes(\PO(\hat b)) + \primesdiag(\QO(\hat b)).$
But $\PO(\hat a) = \PO(\hat b)$,
so $\primesdiag(\fkd_i(\QO(\hat a))) 
= \primesdiag(\QO(\hat b))$.
\end{proof}

\section{Other insertion algorithms}\label{other-sect}

In this final section, we discuss some apparently novel ``primed'' variations
of \defn{Sagan-Worley insertion} (see \cite[\S8]{Sag87}
or \cite[\S6.1]{Worley}) 
and \defn{shifted mixed insertion algorithm} (see \cite[Def. 6.7]{HaimanMixed}).
 The domains of these maps are similar to various  \defn{super-RSK correspondences} (see, e.g., \cite{LNS,Muth,ShimWhite}).
Sections~\ref{modif-sect}, \ref{ok-sect}, and \ref{bij-sect} focus on Sagan-Worley insertion, while
Sections~\ref{mixed-sect} and \ref{relat-sect}  discuss shifted mixed insertion.
This section is mostly independent of the earlier parts of this paper, with the exception of Proposition~\ref{sw-prop} and Corollary~\ref{o-cor3}.

\subsection{Modifying Sagan-Worley insertion}\label{modif-sect}

This section presents the definitions of two versions of the \defn{Sagan-Worley insertion algorithm},
which sends primed {\biwords} to pairs of shifted tableaux.  
A \defn{\biword} is a two-line array of positive integers
\be\label{biword-eq}
\phi = \left[\barr{llll} i_1 & i_2 & \dots & i_n \\ a_1 & a_2 & \dots & a_n\earr\right]
\ee
 where the entries in the top row are weakly increasing
 and
 such that if $i_j = i_{j+1}$ then $a_j \leq  a_{j+1}$.
 We call the top row $i_1i_2\cdots i_n$ of $\phi$ its \defn{index}
 and we call the bottom row $a_1a_2\cdots a_n$ its \defn{value}.
A \defn{primed \biword} is a two-line array 
  satisfying the same conditions,
except its value may have entries $0< a_j \in \ZZ\sqcup\ZZ'$ 
 if no column  $\left[\barr{l} i \\ a\earr\right]$  with  $a\in\ZZ'$ is repeated.
Thus
 \[
 \left[\barr{llllll} 1 & 1 & 1 & 2 & 2 & 3 \\ 4 & 4 & 5 & 5 & 6 & 1\earr\right]
 \quand
 \left[\barr{llllll} 1 & 1 & 1 & 2 & 2 & 3 \\ 4' & 4 & 5' & 5' & 6 & 1\earr\right]
 \]
 are primed {\biwords} while the following are not:
  \[
 \left[\barr{llllll} 1 & 1 & 1 & 2 & 2 & 3 \\ 4 & 4' & 5 & 5 & 6 & 1\earr\right]
 \quand
 \left[\barr{llllll} 1 & 1 & 1 & 2 & 2 & 3 \\ 4' & 4' & 5 & 5 & 6 & 1\earr\right].
 \]
 When given as an input to an insertion algorithm, the
 index of a (primed) {\biword} will give the labels of the recording tableau.
 The condition ``if $i_j = i_{j+1}$ then $a_j \leq  a_{j+1}$'' is designed to ensure that this tableau will be semistandard.
 
We identify a (primed) word $a=a_1a_2\cdots a_n$ with the (primed) {\biword} whose value is $a$ and whose index is   $1,2,3,\dots,n$.
If we never have $a_i=a_{i+1} \in \ZZ'$, then we can form a primed {\biword} $\phi$ with value $a$ from each increasing factorization in  $\Incr_N(a)$
by placing $i$ above all letters in the $i$th factor.
The increasing factorization 
\[a=(45',\emptyset, 2'37')\quad\text{corresponds to}\quad
  \phi=\left[\barr{lllll} 1 & 1 & 3 & 3 &3 \\ 4 & 5' & 2' & 3 & 7'\earr\right]
\] in this way.
This gives a bijection from $\Incr_N(a)$ (when $a$ has no adjacent equal primed letters)
to primed {\biwords} with value $a$ and whose index does not exceed $N$ (when $N$ is finite). 

\begin{definition}
\label{sw-def}
Suppose $\phi$ is a primed {\biword} written in the form \eqref{biword-eq}.
We construct a sequence of increasing shifted tableaux with no primed entries on the main diagonal  
$\emptyset = P_0,P_1,\dots,P_n$ 
 in which
 $P_j$ is formed from $P_{j-1}$ as follows:
\begin{itemize}
\item[(1)] 
On each iteration, an entry $u \in \ZZ\sqcup \ZZ'$ is inserted into a row or column of a shifted tableau.
The process begins 
with $a_j$ inserted into the first row of $P_{j-1}$.
\item[(2)]
If inserting into a row when $u\in \ZZ$, or into a column when $u \in \ZZ'$,
locate the first entry $v$ in the row or column such that $u < v$;
otherwise, locate the first entry $v$ such that $u\leq v$.
When such an entry exists, we say that $u$ ``bumps'' $v$ from its position.

\item[(3)]
If no such $v$ exists then $u$ is added to the end of the row or column to form $P_j$.
If $u$ is primed and the added position is on the main diagonal, then we change its value to $\lceil u \rceil$
and say that  the insertion process ends in column insertion.
Otherwise, we say that the process ends in column (respectively, row) insertion if we are inserting into a column (respectively, row).

\item[(4)]
If $v$ is not on the main diagonal, then 
 replace $v$ by $u$ and  insert $v$ into either the next
 row (if we were inserting into a row) or next column (if we were inserting into a column). 

\item[(5)]
Assume $v$ is on the main diagonal.\footnote{
In this setting the diagonal entry $v $ will always be unprimed and therefore equal to $\lceil v\rceil$, but we do not draw attention to this property
as it will not hold in a modified version of this algorithm described below.
}
If $\lceil u \rceil = \lceil v\rceil $ then continue by inserting $\lceil v\rceil $ into the next column.
If $\lceil u \rceil \neq \lceil v\rceil $ then replace $v$ by $\tilde u$ and insert $\tilde v$ into the next column,
where $\tilde u$ and $\tilde v$ are given by switching the primes 
of $u$ and $v$.

 \end{itemize}
Now define $\oPSW(\phi):=P_n$ and let $\oQSW(\phi)$ be the shifted tableau with the same shape 
whose entry in the unique box of $P_j$ that is not in $P_{j-1}$
is either $i_j$ (when adding $a_j$ to $P_{j-1}$
ends in row insertion) or  $i_j'$ (when adding $a_j$ to $P_{j-1}$ ends in column insertion).

\end{definition}
This slightly modifies the original definition of Sagan-Worley insertion 
from \cite[\S8]{Sag87}
or \cite[\S6.1]{Worley}. 
The latter map, which we will denote by
$\phi \mapsto (\spPSW(\phi), \spQSW(\phi)),$ is given
by repeating Definition~\ref{sw-def}
with two changes: 
\begin{itemize}
\item first, in step (3) we do not remove the prime from a newly added diagonal entry
and we say that the insertion process ends in column insertion only if the last step inserts into a column;

\item second, in step (5) when  $\lceil u \rceil \neq \lceil v\rceil $, 
we redefine $\tilde u$ and $\tilde v$ to be $\tilde u := u$ and $\tilde v :=v$. 
\end{itemize}
It is convenient to think of these maps 
as ``orthogonal'' and ``symplectic'' versions of the same algorithm. Proposition~\ref{tech-sw-lem} will make the basis for this parallelism more precise. 
Primes may occur on the main diagonal of $\spPSW(\phi)$ or $\oQSW(\phi) $ but not on the main diagonal  of $\spQSW(\phi)$ or $\oPSW(\phi)$.

\begin{example}\label{sw-ex2}
Suppose $\phi = \left[\barr{lllll} 1 & 1 & 2 & 2 &2 \\ 4 & 5' & 2' & 3 & 7'\earr\right]$. Then in the notation of Definition~\ref{sw-def}
\[
P_1 = \ytab{4},\quad
P_2 =  \ytab{4 & 5'},\quad
P_3 =\ytab{2 & 4' & 5'},\quad
P_4 = \ytab{\none & 4 \\ 2 &  3 & 5'},\quad
P_5 = \ytab{\none & 4 \\ 2 &  3 & 5' & 7'},
\]
so we have
\[
\oPSW(\phi) =\ytab{\none & 4 \\ 2 & 3 & 5' & 7'} 
\quand
\oQSW(\phi)=\ytab{\none & 2' \\ 1 & 1 & 2' & 2}.
\] 
On the other hand, one can check that
\[
\spPSW(\phi) =\ytab{\none & 4 \\ 2' & 3 & 5' & 7'} 
\quand
\spQSW(\phi)=\ytab{\none & 2 \\ 1 & 1 & 2' & 2}.
\]
Similarly, if
 $\phi = \left[\barr{llllllllll} 1 & 1 &1 & 1 & 3 & 3 &3& 5 & 5 &5 \\ 4 & 4 & 5' & 5& 2'&2 & 3& 3 & 7'& 7\earr\right]$ then
\[
\ba\oPSW(\phi) &=\ytab{\none & 4  & 4 & 5'\\ 2 & 2 & 3 & 3 & 5 & 7' & 7} \\
\oQSW(\phi)&=\ytab{\none &3' & 3 & 5 \\ 1 & 1 &1 & 1& 3' &5 & 5}\ea
\qquand
\ba\spPSW(\phi) &=\ytab{\none & 4  & 4 & 5'\\ 2' & 2 & 3 & 3 & 5 & 7' & 7} \\
\spQSW(\phi)&=\ytab{\none &3 & 3 & 5 \\ 1 & 1 &1 & 1& 3' &5 & 5}.\ea
\]
Finally, comparing with Example~\ref{po-ex}, 
if $c= 4  1'  3  5 4'  2$
then
\[
\ba \oPSW(c) &=\ytab{\none & 3 & 4\\ 1 & 2 & 4' & 5} \\ 
\oQSW(c)&=\ytab{\none & 3' & 5 \\ 1 & 2' &4&6'}\ea
\qquand
\ba\spPSW(c) &=\ytab{\none & 3 & 4\\ 1' & 2 & 4' & 5} \\ 
\spQSW(c)&=\ytab{\none & 3 & 5 \\ 1 & 2' &4&6'}.\ea
\]
\end{example}

The following example illustrates some more differences between these two algorithms.

\begin{example}\label{new-ex}
For $x,y \in \ZZ\sqcup \ZZ'$ 
 identify the word $xy$ with  $ \left[\barr{lllll} 1 & 2 \\ x & y\earr\right]$.
If $x \in \ZZ$ then 
\[ 
\ba 
\oPSW(xx) &= \ytab{ x & x},\\
\oQSW(xx) &= \ytab{ 1 & 2},
\ea
\quad
\ba 
\oPSW(xx') &= \ytab{ x & x},\\
\oQSW(xx') &= \ytab{ 1 & 2'},
\ea
\quad
\ba 
\oPSW(x'x') &= \ytab{ x & x},\\
\oQSW(x'x') &= \ytab{ 1' & 2'},
\ea
\quad
\ba 
\oPSW(x'x) &= \ytab{ x & x},\\
\oQSW(x'x) &= \ytab{ 1' & 2},
\ea\]
while 
\[ 
\ba 
\spPSW(xx) &= \ytab{ x & x},\\
\spQSW(xx) &= \ytab{ 1 & 2},
\ea
\quad
\ba 
\spPSW(xx') &= \ytab{ x & x},\\
\spQSW(xx') &= \ytab{ 1 & 2'},
\ea
\quad
\ba 
\spPSW(x'x') &= \ytab{ x' & x},\\
\spQSW(x'x') &= \ytab{ 1 & 2'},
\ea
\quad
\ba 
\spPSW(x'x) &= \ytab{ x' & x},\\
\spQSW(x'x) &= \ytab{ 1 & 2}.
\ea\]
Alternatively, if $ x,y \in \ZZ$ and $x<y$ then
\[ 
\ba 
\oPSW(yx) &= \ytab{  x  & y},\\
\oQSW(yx) &= \ytab{ 1 & 2'},
\ea
\quad
\ba 
\oPSW(yx') &= \ytab{ x & y'},\\
\oQSW(yx') &= \ytab{ 1 & 2'},
\ea
\quad
\ba 
\oPSW(y'x') &= \ytab{ x & y'},\\
\oQSW(y'x') &= \ytab{ 1' & 2'},
\ea
\quad
\ba 
\oPSW(y'x) &= \ytab{ x & y},\\
\oQSW(y'x) &= \ytab{ 1' & 2'},
\ea\]
while 
\[ 
\ba 
\spPSW(yx) &= \ytab{ x & y},\\
\spQSW(yx) &= \ytab{ 1 & 2'},
\ea
\quad
\ba 
\spPSW(yx') &= \ytab{ x' & y},\\
\spQSW(yx') &= \ytab{ 1 & 2'},
\ea
\quad
\ba 
\spPSW(y'x') &= \ytab{ x' & y'},\\
\spQSW(y'x') &= \ytab{ 1 & 2'},
\ea
\quad
\ba 
\spPSW(y'x) &= \ytab{ x & y'},\\
\spQSW(y'x) &= \ytab{ 1 & 2'}.
\ea\]
\end{example}

We can derive some nontrivial properties of 
Sagan-Worley insertion by observing that its bumping mechanics are identical to shifted Edelman-Greene insertion 
applied to  $\simICK$-equivalence classes of primed involution words involving no braid relations.
One can try to convert a primed word 
to an element of such a class by  ``doubling'' every letter, so that distinct adjacent letters always differ by more than one.
This is our motivation for the following definition.

Given a primed word $a=a_1a_2\cdots a_n$, form $\double(a) $ by applying the map with $i\mapsto 2i$ and $i' \mapsto (2i)'  $ for $i \in \ZZ$ to the letters 
of $a$.
If $\phi$ is a primed {\biword} then define $\double(\phi)$ by applying $\double$ to its value.
For a shifted tableau $T$, construct $\double(T)$
by applying $\double$  
to all of its entries.

A primed word $a$ is a \defn{partial signed permutation} if $\unprime(a)$ has all distinct letters.\footnote{
This terminology is motivated by the fact that if $\unprime(a)$ is a permutation of $1,2,3,\dots,n$ then $a$ is the one-line representation of 
a signed permutation, that is, an element of the hyperoctahedral group.
}
Define a primed {\biword} to be \defn{value-strict} if its value is a partial signed permutation.

\begin{proposition}\label{sw-prop}
Suppose $\phi$ is a primed {\biword} that is value-strict. Then the value of $\double(\phi)$ is a primed involution word,
and it holds that
\[\double\circ\oPSW(\phi) =  \PO \circ \double(\phi)
\quand \oQSW(\phi)  = \QO\circ\double(\phi).\]
\end{proposition}

\begin{proof}
Let $\phi$ be as in \eqref{biword-eq}.
The first claim holds since $\unprime(\double(a_1a_2\cdots a_n))$ is an involution word
where every index  is a commutation.
 This ensures that $ \PO \circ \double(\phi)$ and $\QO \circ \double(\phi)$ are defined,
and that the first tableau coincides with  $ \oPSW \circ \double(\phi) =\double\circ \oPSW(\phi)$
while the second coincides with  $\oQSW \circ \double(\phi) =\oQSW(\phi) $.
\end{proof}

\begin{example} We compute $\QO\circ\double(\phi)$ by viewing    $\double(\phi)$
as an element of $\Incr_\infty(\iR^+(z))$ for some $z \in I_\ZZ$.
If $
\phi =  \left[\barr{lllll} 1 & 1 & 3 & 3 &3 \\ 4 & 5' & 2' & 3 & 7'\earr\right] \leftrightarrow (45',\emptyset, 2'37')
$
then $\double(\phi) \leftrightarrow (8\ 10', \emptyset, 4'\ 6\ 14')$
so
\[
\PO \circ \double(\phi) = \ytab{ \none & 8 \\ 4  & 6 &  10' & 14'}
\quand
\QO\circ \double(\phi)=\ytab{\none & 3' \\ 1 & 1 & 3' & 3}.
\]
\end{example}

 \subsection{Bijective properties}\label{bij-sect}

In this section we derive a formula analogous to Proposition~\ref{tau-prop}
which relates our two versions of Sagan-Worley insertion.
Then we use this result to show that orthogonal Sagan-Worley insertion defines a bijective mapping.

Let $a=a_1a_2\cdots a_n$ be a primed word,
so that
$\oPSW(a) := \oPSW\( \left[\barr{llll} 1 & 2 & \dots & n \\ a_1 & a_2 & \dots & a_n\earr\right]\)$
via our identification of primed words with primed {\biwords}.
For each $j \in [n]$, consider 
the shifted tableaux
 $\oPSW(a_1a_2\cdots a_{j-1})$ and $\oPSW(a_1a_2\cdots a_{j})$.
If these tableaux have different numbers of rows or the same entries in all diagonal positions, then define $\tpi^\SW_j(a)$ to be the 
identity permutation of $\ZZ$. Otherwise, there is a unique diagonal position with different entries in the two tableaux, and   
we let $\tpi^\SW_j(a)$ be the transposition interchanging these.
If $a=45'2'37'$ as in Example~\ref{sw-ex2},
then $\tpi^\SW_3(a) = (2,4)$ and $\tpi^\SW_j(a) = 1$ for $j \in \{1,2,4,5\}$.
Let $\tpi^\SW(a) := \tpi^\SW_1(a)\tpi^\SW_2(a)\cdots \tpi^\SW_n(a)$.
For a primed {\biword} $\phi$ whose value is $a_1a_2\cdots a_n$ define
$ \tpi^\SW(\phi):=\tpi^\SW(a_1a_2\cdots a_n)$.

Let $T$ be a semistandard shifted tableau.
A position $(i,j)$ in $T$ is \defn{free} if 
$\lceil T_{ij} \rceil \neq \lceil T_{xy}\rceil $ whenever $x > i$ or $y < j$,
which in French notation means that $(x,y)$ lies strictly above or strictly to the left of $(i,j)$.
Every diagonal position in $T$ is free.
Adding or removing primes from free positions does not change whether
$T$ is semistandard.
If $(i-1,j-1)$ and $(i,j)$ are both positions in $T$, then we must have $\lceil T_{i-1,j-1}\rceil < \lceil T_{ij}\rceil$.
It follows that if $u\in \ZZ$ is the unprimed form of the entry of $T$ in some position $(i,j)$,
then $(i,j)$ is free if and only if it contributes the first letter equal to $u$ or $u'$ in the reading word $\row(T)$.
Consequently, if $u$ and $v$ are the entries in distinct free positions in $T$, then $\lceil u \rceil \neq \lceil v \rceil$.
Let $\unprime_{\free}(T)$ be the tableau formed from $T$ by removing the primes from   all free positions.
 This is called the \defn{canonical form} of $T$ in \cite[Def. 2.6]{GLP}.

We say that  $u \in \ZZ$ is \defn{initially primed} (respectively, \defn{initially unprimed})
in a primed word if $u'$ (respectively, $u$) appears in the word
and is before any other letters equal to $u$ (respectively $u'$).
Form $\unprime_{\init}(a)$ from a primed word $a$ by unpriming the first
 appearance of $u'$ for each initially primed letter $u \in \ZZ$.
 This is called the \defn{canonical form} of $a$ in \cite[Def. 2.1]{GLP}.
The previous paragraph implies that
  $ \unprime_{\init}(\row(T))=\row(\unprime_\free(T)) $ for any semistandard shifted tableau $T$.

\begin{proposition}
\label{tech-sw-lem}
 Suppose $\phi$ is a primed {\biword} written as in \eqref{biword-eq}.
\ben
\item[(a)] The shifted tableaux $\oPSW(\phi)$ and $\spPSW(\phi)$ are semistandard with the same free positions, and
\be\label{tech-sw-eq} \unprime_{\free}(\oPSW(\phi)) = \unprime_{\free}(\spPSW(\phi))\quand  \unprime_{\diag}(\oQSW(\phi)) = \spQSW(\phi).\ee
\item[(b)] Let $(i,j)$ be a free position in $\spPSW(\phi)$  and 
let $ u \in \ZZ$ be this position's value with its prime removed.
The entry of $\spPSW(\phi)$ in position $(i,j)$  is primed if and only if $u$ is initially primed in the value of $\phi$.
If $i\neq j$ (respectively, $i=j$), then the entry of $\oPSW(\phi)$ (respectively, $\oQSW(\phi)$)
in position $(i,j)$ is primed if and only if $\tpi^\SW(\phi)(u)$ is initially primed in the value of  $\phi$.
\een
\end{proposition}

\begin{proof}
It is known that $\spPSW(\phi)$ is always a semistandard shifted tableau \cite[Thm. 8.1]{Sag87}.
Suppose during the insertion process that defines $\spPSW(\phi)$, 
 a free position $(x,y)$ with entry $v$ is bumped by a number $u$. The sequence of insertions leading to this 
 point starts with some number inserted into a semistandard shifted tableau.
 It follows that we can only have $\lceil u \rceil = \lceil v \rceil$ if $u$ bumps $v$ when inserted into a row,
 since otherwise $u$ would have been bumped on the previous iteration from a position contributing an 
 earlier letter in the row reading word,
 contradicting our assumption that the position of $v$ is free.
 From this observation, it also follows that $u$ would still bump the position $(x,y)$ if we toggled the prime on its entry $v$:
 this is clear if $\lceil u \rceil < \lceil v \rceil$ or if $v$ is primed, 
 and it holds if $\lceil u \rceil = v \in \ZZ$  as then we must be inserting into a row with $u=v'$.
Another relevant property is that
the position which $v$ subsequently bumps on the next iteration (or the new position added to the tableau if $v$ is placed at the end of a row or column) only depends on $\lceil v\rceil $. 
This position is also free unless $v$ is on the main diagonal with $\lceil u\rceil =\lceil v\rceil $,
in which case the free entry is unchanged (as is illustrated in Example~\ref{new-ex}).
Finally, 
if  $T=\spPSW(a_1a_2\cdots a_{j-1})$ has no entries equal to $\lceil a_j\rceil$ or $\lceil a_j\rceil'$,
then when $a_j$ is inserted into $T$ it is placed into the first row and is automatically free.

Given these observations, it follows by induction on the number of columns of $\phi$ that  
$\spPSW(\phi)$ contains $u'$ in a free position for some $u \in \ZZ$ if and only if $u$ is initially primed in the value of $\phi$.
Moreover, we see in this way that the tableau $\oPSW(\phi)$ is formed from $\spPSW(\phi)$ by toggling the primes on certain free positions,
and that the identities \eqref{tech-sw-eq} hold. We already know that $\spPSW(\phi)$ is semistandard, so
$\oPSW(\phi)$ is also a semistandard shifted tableau.

For the last part of the result,
consider a semistandard shifted tableau $T$ and let $\square_u$ for $u \in \ZZ$ denote the free position of $T$ containing $u$ or $u'$,
if this exists.
If $\square_u$ and $\square_v$ are both defined, then 
let $(u,v)\in  S_\ZZ$ act on $T$ by reversing the primes on the entries in these positions if they are not both primed or both unprimed,
and otherwise leaves $T$ unchanged.
This operation extends to an action of the group of permutations of the entries of $\unprime(T)$.

Let $a=a_1a_2\cdots a_n$ be the value of $\phi$.
Form $\bbPSW(a)$ from $\oPSW(a)$ by adding primes to all diagonal positions that are primed in $\oQSW(a)$.
Then $\bbPSW(a)$ is constructed by the same insertion process as the one that defines 
$\spPSW(a)$, except that whenever an inserted number $u$ is about to bump a diagonal entry $v$ with $\lceil u \rceil < \lceil v \rceil$
and $\{u,v\} \not \subset \ZZ$ and $\{u,v\} \not \subset \ZZ'$,
we reverse the primes on $u$ and $v$.  In the exceptional case $\tpi^\SW_j(a)$ is the transposition exchanging $\lceil u\rceil$ and $\lceil v \rceil$,
and outside this case $\tpi^\SW_j(a) = 1$. Thus, with respect to the action defined in the previous paragraph,
it follows that $\tpi^\SW(a) : \bbPSW(a) \mapsto \spPSW(a)$.
This implies the rest of the desired result.
\end{proof}

 \begin{remark}\label{osp-remark}
Orthogonal and symplectic Sagan-Worley insertion restrict to the same map on all (unprimed) {\biwords}.
Proposition~\ref{tech-sw-lem} shows that  we also have
$\oPSW(a) = \spPSW(a)$ for all primed words that have $a = \unprime_{\init}(a)$.
Therefore both $a \mapsto \oPSW(a)$ and $a \mapsto \spPSW(a)$
descend to the same map from ``equivalence classes'' of words to ``equivalence classes''
of shifted tableaux in the sense of \cite[Defs. 2.1 and 2.6]{GLP}.
 \end{remark}
 
We may represent a primed {\biword} $\phi$ as the matrix  $A$
whose entry in position $(i,j)$ is the number of columns 
equal to $\left[\barr{l} i \\ j'\earr\right]$ or $\left[\barr{l} i \\ j\earr\right]$,
and where this number is circled 
 if the column $\left[\barr{l} i \\ j'\earr\right]$ appears.
This gives a bijection between primed {\biwords}  and $\NN$-valued matrices with finitely many nonzero entries,
in which nonzero entries be optionally circled.
Following \cite{Sag87}, we call the latter \defn{circled matrices}. 
For example,
\be\label{A-eq}
\phi=\left[\barr{lllllll} 1 & 1 & 1 &2 & 2 & 2 & 3 \\ 2' & 2 & 2 & 1 & 1 & 2' & 1\earr\right]
\quad\text{has associated circled matrix}
\quad
A=\left[\barr{ll} 0 & \circled{3} \\ 2 & \circled{1} \\ 1 & 0\earr\right].
\ee
This circled matrix $A$ has all entries $A_{ij} \in \{0,1,2,3\}$; that is, the circles have no effect on the value $A_{ij}$.
A primed {\biword} is value-strict if and only if its associated circled matrix has all entries in $\{0,1\}$
and at most nonzero entry in each column.

\begin{theorem}\label{circled-thm}
The map $\phi \mapsto (\oPSW(\phi),\oQSW(\phi))$
is a bijection from primed {\biwords} 
to pairs $(P,Q)$
of semistandard shifted tableaux of the same shape, where $P$ has no primes on the main diagonal
and where the number of times that $j$ or $j'$ (for any $j \in \ZZ$) appear in $P$
and in $Q$. Moreover, if  $A = [ A_{ij}]$ is the circled matrix of $\phi$ then 
each row sum $\sum_i  A_{ij}$ (respectively, column sum $\sum_k  A_{jk}$)
is the number of times that $j$ or $j'$ appear in $\oPSW(\phi)$ (respectively, in $\oQSW(\phi)$). 
\end{theorem}

\begin{remark}\label{circled-thm-rmk}
Theorem~\ref{circled-thm} remains true when the relevant map 
is replaced by $\phi \mapsto (\spPSW(\phi),\spQSW(\phi))$
if one requires $Q$ instead of $P$ to have no diagonal primes (see \cite[Thm. 8.1]{Sag87} or \cite[Thm. 6.1.1]{Worley}).
\end{remark}

\begin{proof}
Let $\phi$ be a primed {\biword}.
Toggling whether a given number in the value of $\phi$ is initially primed or not 
has no effect on $\tpi^\SW(\phi)$ by Proposition~\ref{tech-sw-lem}.
The result is therefore clear from the same result and \cite[Thm. 8.1]{Sag87}  or \cite[Thm. 6.1.1]{Worley}.
\end{proof}

If $\phi$ and $A$ are as in \eqref{A-eq} then $A$ has row sums $1$, $2$ and column sums $3$, $3$, $1$, while
\[\oPSW(\phi) = \ytab{ \none & 2 & 2 \\ 1 & 1  &1 & 2 &2 }\quand \oQSW(\phi) = \ytab{\none & 2 & 2 \\ 1' & 1 & 1 & 2' & 3'}.\]

\subsection{Orthogonal Knuth operators}\label{ok-sect}

There is a conjectural analogue of Theorem~\ref{ck-fkd-thm} for Sagan-Worley insertion, which we describe in this section.
 Let $\shk$ denote the operator that acts on 1- and 2-letter primed words 
by interchanging  
\[
X \leftrightarrow X',
\quad
XX \leftrightarrow XX',\quad
X'X' \leftrightarrow X'X,\]
\[
XY \leftrightarrow YX,\quad 
X'Y \leftrightarrow Y'X,\quad 
XY' \leftrightarrow YX',\quand 
X'Y' \leftrightarrow Y'X',
\]
for all distinct $X,Y \in \ZZ$. Let $\shk$ act on 3-letter primed words as the involution interchanging
\[
ACB \leftrightarrow CAB\quand 
YXZ \leftrightarrow YZX
\]
for all $A,B,C,X,Y,Z \in \ZZ\sqcup \ZZ'$ with $\lceil A \rceil \leq B \leq \lceil C \rceil -\frac{1}{2}$ and $X +\frac{1}{2}\leq  \lceil Y\rceil  \leq Z$,
while fixing 
any 3-letter words not of these forms.  
For a primed word $a=a_1a_2\cdots a_n$ and $i \in[n-2]$, define
\[ 
\ba
\shk_{-1}(a) &:=\shk(a_1) a_2a_3\cdots a_n, \\
\shk_0(a) &:= \shk(a_1a_2)a_3\cdots a_n, \\
\shk_i(a) &:= a_1\cdots a_{i-1} \shk(a_ia_{i+1}a_{i+2}) a_{i+3}\cdots a_n,
\ea
\]
while setting $\shk_i(a) := a$ for $ i \in \ZZ$ with $i+2 \notin[\ell(a)]$.
These \defn{orthogonal Knuth operators} coincide with $\ck_i$ on partial signed permutations.

\begin{conjecture}\label{sw-conj}
If $i \in \ZZ$ 
 then 
$
\oPSW(\shk_i(a))=\oPSW(a)$ and $ \oQSW(\shk_i(a))=\fkd_i(\oQSW(a)). $

\end{conjecture}

It is trivial to verify  these identities when $i \in \{-1,0\}$.
As with Theorem~\ref{ck-fkd-thm}, the difficulty lies in the case when $1\leq i \in \ell(a)-2$. 
Let  $\simSW$  denote the transitive closure of the  relation
on primed words with $a \sim \shk_i(a)$ 
for all $i\in \ZZ$. 

\begin{proposition} 
If $a$ is a primed word then $a \simSW \row(\oPSW(a))$.
\end{proposition}

\begin{proof} 
One can mimick the proof of Proposition~\ref{o-lem2}, using the
relation $\simSW$ in place and $\simICK$, after rewriting Definition~\ref{sw-def}
in a form similar to Definitions~\ref{iarrow-def} and \ref{o'-eg-def}.
We omit the details.
\end{proof}

Thus, Conjecture~\ref{sw-conj} would imply the following:

\begin{conjecture}\label{sw-conj2}
Two primed words satisfy $a\simSW b$ if and only if $\oPSW(a) = \oPSW(b)$.
\end{conjecture}

A version of this property for the original ``symplectic'' form of Sagan-Worley insertion is already known. 
Modify the definition of $\shk_i$ by setting 
\[\op_{-1}(a) := a
\quand
\op_0(a):=a_2a_1a_3a_4\cdots a_n\text{ if $\lceil a_1 \rceil \neq \lceil a_2 \rceil$ and $n:=\ell(a) \geq 2$},\]
while defining $\op_i(a) := \shk_i(a)$ in all other cases.
Write   $\sim$  for the transitive closure of the  relation
 with $a \sim \op_i(a)$ 
for all $i \in \ZZ$.
Notice that if $X\in\ZZ$ then 
$XX \sim XX' \not\sim X'X' \sim X'X$ while
$XX \simSW XX' \simSW X'X' \simSW X'X.$

Worley \cite[Thm. 6.2.2]{Worley} shows that two
primed words satisfy $a\sim b$ if and only if $\spPSW(a) = \spPSW(b)$,
so in particular $
\spPSW(\op_i(a))=\spPSW(a)$ for all $i$.
 We do not know of a reference for this analogue of the second identity in Conjecture~\ref{sw-conj}:
 
 \begin{conjecture}\label{sw-conj3}
If $i > 0$ and $a$ is any primed word then 
$ \spQSW(\op_i(a))=\fkd_i(\spQSW(a)).$ 
\end{conjecture}

The case $i=-1$ is excluded from this conjecture because 
$ \spQSW(\op_{-1}(a))\neq \fkd_{-1}(\spQSW(a))$ whenever $a$ is nonempty, as then
 $\op_{-1}(a)=a$
but $ \spQSW(a)\neq \fkd_{-1}(\spQSW(a)).$
The case $i=0$ is excluded because one can check directly that $\spQSW(\op_0(a))=\fkd_0(\spQSW(a))$  for all primed words $a$.


\begin{proposition}\label{ree-prop}
 If $i>0$ and $\oQSW(\shk_i(a))=\fkd_i(\oQSW(a))$  
then $\spQSW(\op_i(a))=\fkd_i(\spQSW(a))$.
\end{proposition}

\begin{proof} In this case
 $\spQSW(\op_i(a)) =\spQSW(\shk_i(a)) = \unprime_{\diag}(\oQSW(\shk_i(a)))$ by Proposition~\ref{tech-sw-lem},
and this is equal to $ \unprime_{\diag}(\fkd_i(\oQSW(a)))=\fkd_i(\spQSW(a)) $ 
via \eqref{up-fkd-eq} and the same lemma.
\end{proof}

If Conjecture~\ref{sw-conj3} were known, 
then one could derive Conjectures~\ref{sw-conj} and \ref{sw-conj2} 
by (a simplified version of) the strategy we used in Section~\ref{proofs-sect} to prove Theorem~\ref{ck-fkd-thm}.

In more detail, suppose $a$ is a primed word, $i \in [\ell(a)-2]$, and $b:= \shk_i(a)$.
The numbers that are initially primed in $a$ are the same as 
in $b$, so we have
  $\unprime_{\init}(b) = \shk_i(\unprime_{\init}(a))$ and 
$
  \unprime_{\free}(\oPSW(a)) 
  =
  \spPSW(\unprime_{\init}(a))
  =
  \spPSW(\unprime_{\init}(b))
  =
 \unprime_{\free}(\oPSW(b)) 
  $ by Proposition~\ref{tech-sw-lem} and \cite[Thm. 6.2.2]{Worley}.
To prove that $\oPSW(a) = \oPSW(b)$ it suffices by Proposition~\ref{tech-sw-lem} to show that $\tpi^\SW(a) = \tpi^\SW(b)$.
 This can be achieved by proving appropriate versions of the lemmas in Sections~\ref{proof-sect3} and \ref{proof-sect4}.
Then one can deduce $\oQSW(b)=\fkd_i(\oQSW(a))$
 from $\spQSW(b)=\fkd_i(\spQSW(a))$ by an argument similar to the proof of Theorem~\ref{ck-fkd-thm} in Section~\ref{proof-sect4}.

For partial signed permutations,  all of these conjectural results follow from
 Section~\ref{ock-sect}:

 \begin{corollary}\label{o-cor3}
Suppose $a$ and $b$ are partial signed permutations.
Then $a\simSW b$
      if and only if $\oPSW(a) =\oPSW(b)$. Moreover, 
  $\oQSW(\shk_i(a)) = \fkd_i(\oQSW(a))$ for all $i $.
       \end{corollary}

 \begin{proof}
 This follows from Proposition~\ref{sw-prop} given Theorem~\ref{ck-fkd-thm} and 
 Corollary~\ref{o-cor2}, since the operators $\shk_i$ and $\ck_i$
 coincide on partial signed permutations, as do the relations $\simSW$ and $\simICK$.
 \end{proof}
 
 Our two forms of Sagan-Worley insertion do not coincide on  partial signed permutations.
 However, because of Proposition~\ref{ree-prop}, the previous corollary implies the following:
 
 \begin{corollary}
 If $a$ is a partial signed permutation then $\spQSW(\op_i(a)) = \fkd_i(\spQSW(a))$ for all $i$.
 \end{corollary}


\subsection{Extending shifted mixed insertion}\label{mixed-sect}

We now discuss two similar ``primed'' extensions of Haiman's \defn{shifted mixed insertion algorithm} \cite[Def. 6.7]{HaimanMixed}.
These algorithms will turn out to be closely related to the forms of Sagan-Worley insertion analyzed above.
Define a primed {\biword} to be \defn{index-strict} if its index is strictly increasing.
A primed {\biword} is index-strict if and only if its associated circled matrix has all entries in $\{0,1\}$
and at most nonzero entry in each row.

\begin{definition}\label{hm-def}
Suppose $\phi$ is an index-strict primed {\biword} written as in \eqref{biword-eq}.
We construct  a sequence of shifted tableaux
$\emptyset = U_0,U_1,\dots,U_n=U$ whose entries are pairs $(\epsilon,u)$ where $\epsilon \in \{\pm\}$ and $u \in \ZZ\sqcup \ZZ'$.
These tableaux become weakly increasing with no primes on the main diagonal if every entry $(\epsilon,u)$ is replaced by $u$.
The tableau
 $U_j$   is formed from $U_{j-1}$ as follows:
\ben
\item[(1)] Define $\alpha \in \{\pm\} \times \ZZ$ to be 
  $(+,\lceil a_j \rceil)$ if $a_j \in \ZZ$ or $(-,\lceil a_j \rceil)$ if $a_j \in \ZZ'$.
%
Insert this pair into the first row of $U_{j-1}$ according to the following procedure.

\item[(2)] At each stage, a pair $\beta_1 = (\epsilon_1,u_1)$ with $u_1 \in \ZZ\sqcup \ZZ'$ is inserted into a row (when $u_1 \in \ZZ$) or 
a column (when $u_1 \in \ZZ'$).
If every pair $\beta_2=(\epsilon_2,u_2)$ in that row or column has $u_1\geq u_2$ then $\beta_1$ is added to the end;
in this case, the added box can only be on the main diagonal if $u_1 \in\ZZ$.\footnote{
If $u_1 \in \ZZ'$ then the previous iteration must have bumped a position in the preceding column,
so as our tableaux $U_i$ are weakly increasing (when ignoring signs), $\beta_1$
must be strictly bounded by some $\beta_2$ is the current column.}
Otherwise let $\beta_2= (\epsilon_2,u_2)$ be 
the leftmost pair in the row or column with $u_1< u_2$. 

\item[(3)] If  $\beta_2$ is on the main diagonal, then it will always holds that $u_2 \in \ZZ$, and we proceed by
replacing $\beta_2$ with $\beta_1$ and inserting  $(\epsilon_2, u_2')$ into the column
to the right of $\beta_2$.

\item[(4)] If  $\beta_2$ is not on the main diagonal, then replace $\beta_2$  with $(\epsilon_2,u_1)$
and 
 insert $(\epsilon_1,u_2)$ into
 either the row after $\beta_2$ when $u_2\in \ZZ$ or the column to the right of $\beta_2$ when $u_2\in \ZZ'$.
\een
Form $\PHM(\phi)$  from $U$ by replacing 
each main diagonal entry $(\epsilon,x)$ with $\epsilon={-}$ by $x'$, and all other entries $(\epsilon,x)$  by $x$.
Let  $\QHM(\phi)$ be the shifted tableau with the same shape whose entry 
 in the box of $U_j$ that is not in $U_{j-1}$ is either $i_j$ or $i_j'$, with a primed number occurring
 precisely when this box is off the main diagonal and its entry in $U_j$ has the form $(\epsilon,x)$ with  $\epsilon = {-}$.
 \end{definition}

Unlike earlier algorithms, here successive insertions do not always occur
in consecutive rows and columns; also, the orientation of insertion can switch multiple times from rows to columns
and from columns back to rows.
As our notation suggests, Definition~\ref{hm-def} has a ``symplectic'' variant.

\begin{definition}\label{hm-def2}
Given an index-strict primed {\biword} $\phi$ written as in \eqref{biword-eq}, define shifted tableaux
$\emptyset = U_0,U_1,\dots,U_n=U$ by repeating Definition~\ref{hm-def},
 but modifying step (3) so that the entry $\beta_2$ is replaced by $(\epsilon_2,u_1)$ while  $(\epsilon_1, u_2')$ is inserted into the next column.
Then:
\begin{itemize}
\item Form $\spPHM(\phi)$  from $U$ by replacing all entries $(\epsilon,x)$  by $x$.
\item Let  $\spQHM(\phi)$ be the shifted tableau with the same shape whose entry 
 in the box of $U_j$ that is not in $U_{j-1}$ is either $i_j$ or $i_j'$, with a primed number occurring
 precisely when the entry of $U_j$ in this box  has the form $(\epsilon,x)$ with  $\epsilon = {-}$.
 \end{itemize}
 \end{definition}
 
 \begin{remark}\label{hm-remark}
 When the index of $\phi$ consists of the numbers $1,2,3,\dots,n$ and the value of $\phi$ has no primed entries,
both $\phi \mapsto (\PHM(\phi),\QHM(\phi))$ and $\phi \mapsto (\spPHM(\phi),\spQHM(\phi))$ 
 reduce to \defn{shifted mixed insertion}   \cite[Def. 6.7]{HaimanMixed}.
 Neither extension seems to have appeared in the literature.
  We refer to these maps as \defn{orthogonal} and \defn{symplectic mixed insertion}.
  More generally, the two algorithms restrict to the same map on all index-strict (unprimed) {\biwords}.
  \end{remark}
  
\begin{example}\label{sw-ex1}
Suppose our index-strict primed {\biword} is $\phi= \left[\barr{lllll} 2 & 3& 4& 5 & 7 \\ 2' & 2 & 1 & 1' & 2' \earr\right].$ Then,
writing $\pm x$ in place of $(\pm, x)$, 
the sequence of shifted tableaux $U_j$ in 
Definition~\ref{hm-def} are
\[ 
\ytableausetup{boxsize = 0.70cm,aligntableaux=center}
U_1 = \begin{ytableau}-2\end{ytableau},
\quad
U_2= \begin{ytableau}-2 & +2\end{ytableau},
\quad
U_3= \begin{ytableau} \none & -2 \\ +1 & +2'\end{ytableau},
\quad
U_4= \begin{ytableau} \none & -2 \\ +1 & +1 & -2'\end{ytableau},
\quad
U_5= \begin{ytableau} \none & -2 \\ +1 & +1 & -2' & -2\end{ytableau},
\]
so 
$
\PHM(\phi) = 
 \smalltab{ \none & 2' \\ 1 & 1 & 2' & 2 }
$ and $
 \QHM(\phi) = 
 \smalltab{  \none & 4 \\ 2 & 3 & 5' & 7' }.
$ The tableaux $U_j$ in 
Definition~\ref{hm-def2} are
\[ 
\ytableausetup{boxsize = 0.70cm,aligntableaux=center}
U_1 = \begin{ytableau}-2\end{ytableau},
\quad
U_2= \begin{ytableau}-2 & +2\end{ytableau},
\quad
U_3= \begin{ytableau} \none & +2 \\ -1 & +2'\end{ytableau},
\quad
U_4= \begin{ytableau} \none & +2 \\ -1 & +1 & -2'\end{ytableau},
\quad
U_5= \begin{ytableau} \none & +2 \\ -1 & +1 & -2' & -2\end{ytableau},
\]
so we have 
$
\spPHM(\phi) = 
 \smalltab{ \none & 2 \\ 1 & 1 & 2' & 2 }
$ and $
 \spQHM(\phi) = 
 \smalltab{  \none & 4 \\ 2' & 3 & 5' & 7' }.
$
\end{example}

\subsection{Relating  shifted mixed insertion to Sagan-Worley insertion}\label{relat-sect}

The original forms of shifted mixed insertion and Sagan-Worley insertion
take permutations as inputs. Inverting these inputs exchanges the outputs of the two algorithms by \cite[Thm. 6.10]{HaimanMixed}.
 In this final section we show that this property extends to our primed forms of both insertion algorithms,
 with inversion 
replaced by a transpose operation $\phi \mapsto \phi^\top$ on primed {\biwords}. 

The relevant transpose operation is given as follows.
Starting from a primed {\biword} $\phi$,
first move any primes from the value to the entries directly above them,
then interchange the two rows and reorder the columns to be lexicographically increasing, and call the result $\phi^\top$. If 
\be\label{transpose-eq}
\phi= \left[\barr{lllll} 2 & 3 & 4 & 5 & 7 \\ 2' & 2 & 1 & 1' & 2' \earr\right]
 \quad\text{then}\quad
 \phi^\top=
 \left[\barr{lllll} 1 & 1 & 2 & 2 &2 \\ 4 & 5' & 2' & 3 & 7'\earr\right],
  \ee
  for example.
In terms of the associated circled matrices, this operation is just the matrix transpose,
so it interchanges index-strict and value-strict {\biwords}.

One can observe the identities in the following theorem by comparing
Examples~\ref{sw-ex2} and \ref{sw-ex1}. 
  
\begin{theorem}
\label{sw-hm-thm}
If $\phi$ is  index-strict, then it holds that
$\PHM(\phi) = \oQSW(\phi^\top)$ and $ \QHM(\phi) = \oPSW(\phi^\top)$,
and it also holds that
$ \spPHM(\phi) = \spQSW(\phi^\top)$ and $ \spQHM(\phi) = \spPSW(\phi^\top).$
\end{theorem}

\begin{proof}
The desired identities generalize \cite[Thm. 6.10]{HaimanMixed} in the following sense.
As noted in Remarks~\ref{osp-remark} and \ref{hm-remark}, on index-strict (unprimed) {\biwords},
orthogonal and symplectic Sagan-Worley insertion restrict to the same map $\phi \mapsto (P_{\mathsf{SW}}(\phi), Q_{\mathsf{SW}}(\phi))$,
while orthogonal and symplectic mixed insertion restrict to the same map $\phi \mapsto (P_{\mathsf{HM}}(\phi), Q_{\mathsf{HM}}(\phi))$.
  \cite[Thm. 6.10]{HaimanMixed} asserts that if the index of $\phi$ is $1,2,\dots,n$ and the value of $\phi$
is a permutation of $1,2,\dots,n$,
 then 
$P_{\mathsf{HM}}(\phi) =  Q_{\mathsf{SW}}(\phi^\top)$ and $Q_{\mathsf{HM}}(\phi) =  P_{\mathsf{SW}}(\phi^\top)$.
This property extends to the case when $\phi$ is any (unprimed) {\biword} that is both index- and value-strict,
since then all of the relevant tableaux are obtained from the permutation case by applying appropriate order-preserving bijections to their entries.
 
Let $\phi$ be a primed {\biword} written as in \eqref{biword-eq}.
We will only prove that $\PHM(\phi) = \oQSW(\phi^\top)$ and $ \QHM(\phi) = \oPSW(\phi^\top)$, 
as the argument for the symplectic case is similar.
We first assume $\phi$ is both index-strict and value-strict.
Then we have
\be \label{par2}\ba
\unprime(\QHM(\phi))=\QHM( \unprime(\phi))&=\oPSW( \unprime(\phi^\top)) = \unprime(\oPSW(\phi^\top))  , \\
\unprime_{\diag}(\PHM(\phi)) = \PHM( \unprime(\phi))&= \oQSW(\unprime(\phi^\top))= \unprime_{\diag}( \oQSW(\phi^\top)) ,
\ea
\ee
using the preceding paragraph for the middle equalities, and the definitions of our insertion algorithms for the others.
Thus, we already know that if we ignore all primes then the corresponding entries are equal in 
$\QHM(\phi)$ and $\oPSW(\phi^\top)$, and also in $\PHM(\phi)$ and $\oQSW(\phi^\top)$.
More specifically, since the outputs of $\QHM$ and $\oPSW$ never have primed entries on the main diagonal,
to prove that $\PHM(\phi) = \oQSW(\phi^\top)$ and $ \QHM(\phi) = \oPSW(\phi^\top)$
we just need to show that 
each off-diagonal box is primed in $\QHM(\phi)$ if and only if it is primed in $\oPSW(\phi^\top)$,
and each diagonal box is primed in $\PHM(\phi)$ if and only if it is primed in $\oQSW(\phi^\top)$.

 We will demonstrate this by an inductive argument. Let $\hat\phi$ be the {\biword} formed from $\phi$
 by omitting its last column $\left[\barr{c} i_n \\ a_n \earr\right]$. Then $\hat\phi$ is still index- and value-strict,
 so we may assume by induction that 
  $\QHM( \hat\phi)=\oPSW( \hat\phi^\top) 
$
and
$ \PHM( \hat\phi)= \oQSW(\hat\phi^\top)$.
To deduce that these identities also hold for $\phi$,
we must understand how the shifted tableaux
$ \PHM( \phi)$, $ \QHM( \phi)$, $ \oPSW( \phi^\top)$, and $ \oQSW( \phi^\top)$
are   constructed from 
$ \PHM( \hat\phi)$, $ \QHM( \hat\phi)$, $ \oPSW( \hat\phi^\top)$, and $ \oQSW( \hat\phi^\top)$.

We consider the mixed insertion case first.
Define $\hat U$ from $\PHM(\hat\phi)$ by replacing 
each main diagonal entry $x$ by $(\epsilon,\lceil x\rceil )$ where $\epsilon={+}$ (respectively, $\epsilon={-}$) if $x$ is unprimed
(respectively, primed),
and then replacing each off-diagonal entry $x$ by $(\epsilon,x)$ 
where $\epsilon={+}$ (respectively, $\epsilon={-}$) if
 the entry in the same position of $\QHM(\hat\phi)$ is unprimed (respectively, primed).
 Construct $U$ from $\PHM(\phi)$ analogously.
Each box in these tableaux contains an entry of the form $(\epsilon,x)$ and we refer to $\epsilon$ as the \defn{sign} of the box.
Finally, let $\alpha = (\epsilon, \lceil a_n \rceil)$ where $\epsilon={+}$ (respectively, $\epsilon={-}$) if $a_n$ is unprimed
(respectively, primed). 
Then $U$ is obtained by inserting $\alpha$ into the first row of $\hat U$ according to the procedure in Definition~\ref{hm-def}.

The set of boxes in $U$ (respectively, $\hat U$) with negative sign
is the union of the sets of primed positions in $\QHM(\phi)$ (respectively, $\QHM(\hat\phi)$)
and diagonal primed positions in $\PHM(\phi)$ (respectively, $\PHM(\hat\phi)$).
From Definition~\ref{hm-def}, we see that the signs of the boxes in $ U$ are the same in $\hat U$,
except that if inserting $\alpha$ successively bumps a sequence of diagonal boxes $\cA_1, \cA_2,\dots, \cA_{p-1}$
and eventually terminates at a new box $\cA_{p}$, then box $\cA_1$  adopts the sign of $\alpha$
and box $\cA_{i+1}$  adopts the sign of box $\cA_i$ in $\hat U$ for each $i \in [p-1]$.
Notice that boxes $\cA_1,\cA_2,\dots,\cA_{p-1}$ are the main diagonal positions where $\unprime_{\diag}(\PHM(\hat\phi))$ differs from $\unprime_{\diag}(\PHM(\phi))$,
and that $\cA_p$ is the unique box of the second tableau that is not in the first.

We now examine the Sagan-Worley insertion case.
For any primed {\biword} $\psi$ form $ \tPSW( \psi)$ from $\oPSW( \psi)$
by adding a prime to each main diagonal box that is primed in $\oQSW(\psi)$.
The set of primed boxes in $ \tPSW( \psi)$
is the union of the sets of primed positions in $\oPSW(\psi)$ 
and diagonal primed positions in $\oQSW(\psi)$.
Let $\emptyset = T_0, T_1, T_2,\dots, T_n $
be the sequence of shifted tableaux formed by successively inserting the entries in the second row $\phi^\top$
according to the bumping procedure in Definition~\ref{sw-def}, but modified so that we do not 
remove primes from new boxes added to the main diagonal in step (3). 
Then we have $T_n = \tPSW(\phi^\top)$.
Define $\emptyset = \hat T_0, \hat T_1, \hat T_2,\dots, \hat T_{n-1} = \tPSW(\hat \phi^\top)$
to be the analogous sequence of shifted tableaux formed by successively inserting the entries in the second row $\hat\phi^\top$
by the same modified bumping procedure.

 Suppose  $b_1,b_2,\dots,b_n$ are the entries in the second row of $\phi^\top$
and $b_j$ is the largest entry in this list. Note that $b_j$ is either $i_n'$ or $i_n$ according to whether  $a_n$ is primed or unprimed.
Then $\hat\phi^\top$ is formed from $\phi^\top$
by omitting column $j$, so $T_i = \hat T_i$ for $0\leq i < j$ and $T_j$ is formed from $\hat T_{j-1}$ by adding $b_j$ to the end of the first row.
As we insert the remaining entries $b_{j+1}, b_{j+2},\dots, b_n$ into $T_j$ to form $T_k$ for $j<k\leq n$,
the maximal entry $b_j$ may be bumped to a new position but the remaining entries are almost the same as in $\hat T_{k-1}$.
The only difference is that whenever the unique maximal entry is bumped from a main diagonal position,
its prime is switched with the entry replacing it. 

Thus if the maximal entry is successively bumped 
from a sequence of main diagonal boxes $\cB_1,\cB_2,\dots, \cB_{q-1}$ and eventually ends up in some box $\cB_q$,
then box $\cB_1$ in $\tPSW(\phi^\top)$ retains the prime of $b_j$ (which is the prime of $a_n$),
while box $\cB_{i+1}$ in $\tPSW(\phi^\top)$  for each $i \in [q-1]$ retains the 
prime of whichever number ends up bumping the maximal entry from box $\cB_i$.
We can identify these primes as well as the boxes $\cB_1,\cB_2,\dots, \cB_{q}$ 
by comparing the associated recording tableaux:
the first $q-1$ boxes are the main diagonal positions where $\unprime_{\diag}(\oQSW(\hat\phi^\top))$ differs from $\unprime_{\diag}(\oQSW(\phi^\top))$,
as these positions indicate where a smaller entry would arrive at a later insertion step if the maximal entry $b_j$ were never inserted;
the primes of the bumping entries are the primes of these positions in $\oQSW(\hat\phi^\top)$, or equivalently in $\tPSW(\hat\phi^\top)$;
and   $\cB_q$ is the unique box of $\oQSW(\phi^\top)$ that is not in $\oQSW(\hat\phi^\top)$.
We conclude that the primes of the boxes in
$\tPSW(\phi^\top)$
are the same as in $\tPSW(\hat\phi^\top)$,
except box $\cB_1$  adopts the prime of $a_n$
and box $\cB_{i+1}$ adopts the prime of box $\cB_i$ in $\tPSW(\hat\phi^\top)$ for each $i \in [q-1]$.

Our hypothesis that 
 $\QHM( \hat\phi)=\oPSW( \hat\phi^\top) 
$
and
$ \PHM( \hat\phi)= \oQSW(\hat\phi^\top)$
implies   
$\hat U = \tPSW(\hat\phi^\top)$.
To show that 
 $\QHM( \phi)=\oPSW( \phi^\top) 
$
and
$ \PHM( \phi)= \oQSW(\phi^\top)$
it suffices by \eqref{par2}
to check that the negative boxes in $U$ have the same locations as the primed boxes in $\tPSW(\phi^\top)$.
Comparing our descriptions of these boxes above, we see that it is enough to show that 
$p=q$ and that the boxes $\cA_i=\cB_i$ coincide for all $i$, and this also follows by \eqref{par2}.
 
To finish the proof, suppose $\phi$ is any index-strict primed {\biword} with $n$ columns.
Form $\psi$ from $\phi$ by taking its transpose, then replacing the index by the consecutive numbers $1<2<\dots<n$,
and then taking the transpose again. For example, if 
\[
\phi= \left[\barr{lllll} 2 & 3 & 4 & 5 & 7 \\ 2' & 2 & 1 & 1' & 2' \earr\right]
 \quad\text{then}\quad
 \psi=
 \left[\barr{lllll} 1 & 2 & 3 & 4 &5\\ 4 & 5' & 2' & 3 & 7'\earr\right]^\top =
  \left[\barr{lllll} 2 & 3 & 4 & 5 & 7 \\ 3' & 4 & 1 & 2' & 5' \earr\right],
 \]
It is clear that $\oPSW(\phi^\top) = \oPSW(\psi^\top)$ and  $\QHM(\phi) = \QHM(\psi)$. 
Let $\cF : \{ 1'<1<2'<2<\dots< n'<n\} \to \{1'<1<2'<2<\dots\}$
be the unique map with $\cF(i) = j$ and $\cF(i') = j'$ if $j$ is the entry in the index of $\phi^\top$ in column $i$.
Then $\phi$ is formed by applying $\cF$ to the value of $\psi$,
and we have   $\cF(\oQSW(\psi^\top)) = \oQSW(\phi^\top)$ and $\cF(\PHM(\psi)) = \PHM(\phi)$.
As we already know that $\QHM(\psi)= \oPSW(\psi^\top) $ and 
 $ \PHM(\psi)=\oQSW(\psi^\top) $, the theorem follows.
\end{proof}

It would interesting to find a way to extend Definitions~\ref{hm-def} and \ref{hm-def2}
so that 
  Theorem~\ref{sw-hm-thm} holds for all primed {\biwords}, similar to what is done in
 \cite[\S3.4]{ShimWhite} for mixed insertion.

Recall that we identify  
$a=a_1a_2\cdots a_n$ with the {\biword} $\left[\barr{llll} 1 & 2 & \dots & n \\ a_1 & a_2 & \dots & a_n\earr\right]$.
 
\begin{corollary}
\label{hm-bijection-thm}
The map $a\mapsto (\PHM(a), \QHM(a))$
(respectively, $a\mapsto (\spPHM(a), \spQHM(a))$)
 is a bijection from the set of 
 primed words 
 with all positive letters
to the set of pairs $(P,Q)$
of shifted tableaux of the same shape, in which $P$ is semistandard,
$Q$ is standard, and $Q$ (respectively, $P$) has no primed entries on the main diagonal.
\end{corollary}

\begin{proof}
Primed words with positive letters correspond to circled matrices with exactly
one nonzero entry, given by $1$ or $1'$, in each of the first $\ell(a)$ rows, and no other nonzero rows.
By Theorem~\ref{circled-thm} and Remark~\ref{circled-thm-rmk}, the maps
   $\phi \mapsto (\oPSW(\phi), \oQSW(\phi))$ and  $\phi \mapsto (\spPSW(\phi), \spQSW(\phi))$ are bijections from the set of transposes of such {\biwords}
 to the
set of pairs of
shifted tableaux with the desired properties, but in reverse order. The result thus holds by Theorems~\ref{circled-thm}
and
\ref{sw-hm-thm}.
\end{proof}

\end{document}

%% file: preamble.tex
\usepackage{latexsym}
\usepackage{amssymb}
\usepackage{amsthm}
\usepackage{amscd}
\usepackage{amsmath}
\usepackage{mathrsfs}
\usepackage{graphicx}
\usepackage{shuffle}
\usepackage[colorlinks=true, pdfstartview=FitV, linkcolor=blue, citecolor=blue, urlcolor=blue]{hyperref}
\usepackage{mathtools}
\usepackage[bbgreekl]{mathbbol}
\usepackage{mathdots}
\usepackage{ytableau}
\usepackage[enableskew,vcentermath]{youngtab}
\usepackage{tikz-cd}

\usepackage{mathtools}
\usepackage{amssymb}

\usepackage[colorinlistoftodos]{todonotes}
\usetikzlibrary{chains,scopes}
\usetikzlibrary{shapes.geometric,positioning}

\DeclareSymbolFontAlphabet{\mathbb}{AMSb}
\DeclareSymbolFontAlphabet{\mathbbl}{bbold}

\def\Sp{\mathsf{Sp}}
\def\O{\mathsf{O}}

\definecolor{darkred}{rgb}{0.7,0,0} 
\newcommand{\defn}[1]{{\color{darkred}\emph{#1}}} 

\numberwithin{equation}{section}
\allowdisplaybreaks[1]

\theoremstyle{definition}
\newtheorem* {theorem*}{Theorem}
\newtheorem* {proposition*}{Proposition}
\newtheorem* {corollary*}{Corollary}
\newtheorem* {conjecture*}{Conjecture}
\newtheorem{theorem}{Theorem}[section]

\theoremstyle{definition}

\newtheorem* {claim}{Claim}
\newtheorem* {example*}{Example}

\newtheorem{lemma}[theorem]{Lemma}
\theoremstyle{definition}
\newtheorem{definition}[theorem]{Definition}
\theoremstyle{definition}

\newtheorem{conjecture}[theorem]{Conjecture}
\newtheorem{proposition}[theorem]{Proposition}
\newtheorem{corollary}[theorem]{Corollary}

\newtheorem{remark}[theorem]{Remark}
\newtheorem*{remark*}{Remark}
\theoremstyle{definition}
\newtheorem {example}[theorem]{Example}
\theoremstyle{definition}

\theoremstyle{definition}

\theoremstyle{definition}

\def\({\left(}
\def\){\right)}

\newcommand{\CC}{\mathbb{C}}

\newcommand{\cR}{\mathcal{R}}

\def\CC{\mathbb{C}}

\def\ZZ{\mathbb{Z}}

\def\GL{\textsf{GL}}

\def\fk{\mathfrak}

\def\barr{\begin{array}}
\def\earr{\end{array}}
\def\ba{\begin{aligned}}
\def\ea{\end{aligned}}
\def\be{\begin{equation}}
\def\ee{\end{equation}}

\def\qquand{\qquad\text{and}\qquad}
\def\quand{\quad\text{and}\quad}

\def\quord{\quad\text{or}\quad}

\def\I{\mathcal{I}}

\def\cH{\mathcal H}

\def\hs{\hspace{0.5mm}}

\def\ben{\begin{enumerate}}
\def\een{\end{enumerate}}

\def\bei{\begin{itemize}}
\def\eei{\end{itemize}}

\def\hs{\hspace{0.5mm}}

\def\fpf{{\textsf {fpf}}}

\def\D{\hat D}
\newcommand{\wfpf}{\Theta}

\def\Des{\mathrm{Des}}

\def\Ifpf{I^{\fpf}}

\newcommand{\Fl}{\textsf{Fl}}

\newcommand{\cA}{\mathcal{A}}
\newcommand{\cB}{\mathcal{B}}

\def\iR{\cR^\O}

\def\cHfpf{\cH_{\textsf{Sp}}}
\def\iH{\cH_{\textsf{O}}}

\def\cF{\mathcal{F}}

\newcommand{\row}{\operatorname{row}}

\def\isim{\mathbin{=_\textsf{O}}}

\def\simA{=_{\textsf{Br}}}

\def\simK{\overset{\textsf{K}}\sim}

\def\simICK{\overset{\textsf{O}}\sim}

\newcommand{\Incr}{\operatorname{Incr}}

\def\row{\textsf{row}}
\def\col{\textsf{col}}
\def\swdiag{\mathfrak{d}_{\textsf{SW}}}
\def\nediag{\mathfrak{d}_{\textsf{NE}}}

\def\path{\mathsf{path}}
\def\iarrow{\xleftarrow{\textsf{O}}}

\newcommand*\circled[1]{\tikz[baseline=(char.base)]{
\node[shape=circle,draw,inner sep=1](char) {#1};}}

\def\whSym
{\mathfrak{m}\textsf{Sym}}

\def\whQSym
{\mathfrak{m}\textsf{QSym}}

\def\D{\textsf{D}}
\def\SD{\textsf{SD}}

\def\GQ{G\hspace{-0.2mm}Q}
\def\GP{G\hspace{-0.2mm}P}

\def\KQ{K\hspace{-0.2mm}Q}
\def\KP{K\hspace{-0.2mm}P}

\def\GO{\GP^{\textsf{O}}}

\def\NN{\mathbb{N}}

\def\EG{\textsf{EG}}

\def\bPO{\widetilde{P}_{\textsf{EG}}^\O}
\def\bQO{\widetilde{Q}_{\textsf{EG}}^\O}

\def\spPHM{P^\Sp_{\textsf{HM}}}
\def\spQHM{Q^\Sp_{\textsf{HM}}}

\def\PHM{P^\O_{\textsf{HM}}}
\def\QHM{Q^\O_{\textsf{HM}}}

\def\spPSW{ P^\Sp_{\textsf{SW}}}
\def\spQSW{ Q^\Sp_{\textsf{SW}}}

\def\tPSW{{\tilde P}^\O_{\textsf{SW}}}
\def\oPSW{P^\O_{\textsf{SW}}}
\def\oQSW{Q^\O_{\textsf{SW}}}

\def\bbPSW{{\widetilde P}^\O_{\textsf{SW}}}

\def\ck{\textsf{ock}}
\def\shk{\mathsf{okn}}
\def\op{\mathsf{spkn}}

\def\simSW{\overset{\textsf{shK}}\sim}

\def\PO{P_{\textsf{EG}}^\O}
\def\QO{Q_{\textsf{EG}}^\O}

\def\PSp{P_{\textsf{EG}}^\Sp}
\def\QSp{Q_{\textsf{EG}}^\Sp}

\newcommand{\ytab}[1]{
\ytableausetup{boxsize = .5cm,aligntableaux=center}
{\small\begin{ytableau}  #1  \end{ytableau}}
}

\newcommand{\smalltab}[1]{
\ytableausetup{boxsize = .44cm,aligntableaux=center}
{\small\begin{ytableau}  #1  \end{ytableau}}
}

\def\fks{{\fk s}}
\def\shword{\textsf{shword}}

\def\T{\mathsf{T}}

\def\cyc{\mathsf{cyc}}
\def\marked{\mathsf{marked}}
\def\unprime{\mathsf{unprime}}

\def\simA{\xleftrightarrow{\ \cR\ }}

\def\isim{\xleftrightarrow{\ \iR\ }}
\def\wfpf{1_{\mathsf{fpf}}}
\def\Ifpf{I^{\mathsf{fpf}}}

\def\iR{\cR_{\mathsf{inv}}}
\def\cA{\mathcal{A}_{\mathsf{inv}}}

\def\fkd{\mathfrak{d}}

\def\primes{\#\mathsf{primes}}
\def\primesdiag{\primes_{\diag}}
\def\entries{\mathsf{entries}}
\def\dbump{\Delta^{\mathsf{bump}}}
\def\tpi{\tau}

\def\double{\mathsf{double}}

\def\SW{\mathsf{SW}}

\def\simA{\approx}
\def\iisim{\equiv}
\def\isim{\mathop{\hat \equiv}}
\def\cA{\mathcal{A}}
\def\Incr{\mathsf{Incr}}

\def\cseq{\mathsf{cseq}}

\def\regionSW{\raisebox{-0.5pt}{\tikz{\filldraw (0.25,0.25) circle (1pt);  \draw (0.25,0.25) -- (.25,0) -- (0,.25) -- (0.25,0.25);}}}
\def\regionNE{\raisebox{-0.5pt}{\tikz{\filldraw (0,0) circle (1pt);  \draw (0,0) -- (.25,0) -- (0,.25) -- (0,0);}}}

\def\wpath{\mathsf{path}^\leq}
\def\spath{\mathsf{path}^<}
\def\rwpath{\mathsf{rpath}^\leq}
\def\rspath{\mathsf{rpath}^<}
\def\cwpath{\mathsf{cpath}^\leq}
\def\cspath{\mathsf{cpath}^<}

\def\diag{\text{\emph{diag}}}
\def\init{\text{\emph{init}}}
\def\entriesdiag{\mathsf{diag}}

\def\free{\text{\emph{free}}}